\documentclass[a4paper,11pt]{article}
\usepackage[utf8]{inputenc}
\usepackage{tikz}

\usepackage{amsmath}
\usepackage{amsthm}
\usepackage{amsfonts}
\usepackage{amssymb}
\usepackage{dsfont}
\usepackage{tikz-cd}
\usepackage{import}
\usepackage{yfonts}
\usepackage{graphicx}
\RequirePackage{hypernat}
\usepackage{float}
\usepackage[Symbol]{upgreek}
\usepackage{booktabs}
\usepackage{mathtools}
\usepackage{subcaption}
\usepackage{array}
\usepackage{cclicenses}
\usepackage{geometry}
\usepackage[hypertexnames=false]{hyperref}
\hypersetup{colorlinks=false, pdfborder={0 0 0}}
\usepackage{hhline}

\usepackage{autonum}

\newcommand{\kk}{\underline{k}}
\newcommand{\m}{\mathbf{m}}
\newcommand{\n}{\mathbf{n}}
\newcommand{\p}{\bar p}
\newcommand{\w}{\bar w} 
\newcommand{\0}{\mathbf{0}}
\renewcommand{\geq}{\geqslant}
\renewcommand{\leq}{\leqslant}
\renewcommand{\preceq}{\preccurlyeq}

\newtheorem{theorem}{Theorem}[section]
\newtheorem{lemma}[theorem]{Lemma}
\newtheorem{proposition}[theorem]{Proposition}
\newtheorem{corollary}[theorem]{Corollary}
\newtheorem{definition}[theorem]{Definition}
\newtheorem{assumption}[theorem]{Assumption}
\newtheorem{remark}[theorem]{Remark}

\theoremstyle{definition}
\newtheorem{example}[theorem]{Example}

\newcommand{\RR}{\mathbf{R}}
\newcommand{\NN}{\mathbf{N}}
\newcommand{\ZZ}{\mathbf{Z}}

\newcommand{\EE}{\mathbb{E}}
\newcommand{\QQ}{\mathbf{Q}}

\newcommand{\mE}{\mathcal{E}}

\newcommand{\ve}{\varepsilon}

\title{Weak universality for stochastic reaction-diffusion models\\ with long-range correlated noise}
\author{Simon Gabriel\footnote{Email: \href{mailto:s.gabriel@berkeley.edu}{s.gabriel@berkeley.edu},
Department of Mathematics, University of California, Berkeley, USA.} 
\quad 
and 
\ \ 
Markus Tempelmayr\footnote{Email: \href{mailto:markus.tempelmayr@epfl.ch}{markus.tempelmayr@epfl.ch},
EPFL, Lausanne, Switzerland and Universität Münster, Germany.}}
\date{\today}

\begin{document}

\maketitle

\begin{abstract}
	We study the large-scale behaviour of a family of 
	stochastic reaction-diffusion equations driven by long-range correlated noise in a weakly nonlinear regime. 
	Depending on the decay of correlations of the noise and the strength of the nonlinearity, the universal behaviour is governed by a version of the dynamical $\Phi^p$ model with long-range correlated noise, and with a coupling constant determined by the reaction term of the microscopic model. 
	Our main result establishes the stochastic estimates and convergence of models required in the theory of regularity structures. 
	We adapt the multiindex-based approach to regularity structures using a suitable expansion of the reaction term tailored to the law of the noise. 
	This yields a systematic weak universality result, allowing for a particularly simple identification of the macroscopic limit throughout the full subcritical regime and beyond Gaussian noise.
 The method appears robust and applicable 
 to a broader class of singular stochastic PDEs.
\end{abstract}

\noindent\hspace{0.3cm}{\small\textit{{Keywords.}} Singular SPDEs; Regularity Structures;
	Weak universality; Long-range
	correlated noise.}

\noindent\hspace{.3cm}{\small\textit{{MSC classification.}} Primary: 60H15; 60L30. Secondary: 35R60.}

\tableofcontents

\section{Introduction}\label{sec:intro}

We study the large-scale behaviour of the stochastic
reaction-diffusion equation
\begin{equation}\label{micro}
	(\partial_{x_{0}}-\Delta)\phi = \varepsilon^a F(\phi) + \zeta
\end{equation}
on $\RR_{+} \times \RR^{d}$ in a weakly nonlinear regime ($\ve\to0$ and $a > 0$), with a sufficiently
smooth, odd nonlinearity $F$.
Here and throughout the paper, $ \Delta$ denotes the Laplace operator with respect to
the coordinates $ x_{1, \ldots, d}$.
More precisely, $ x = (x_{0}, x_{1, \ldots , d}) \in \RR^{1 +d}$ represents a time-space
variable (with $ x_{0}$ representing the time variable), and we equip the Euclidean space with
the parabolic distance
\begin{equation}
	\begin{aligned}
		|x| \coloneqq \Big(|x_{0}| + \sum_{i =1}^{d} |x_{i}|^{2}\Big)^{1/2}\,,
	\end{aligned}
\end{equation} 
and set $ D
\coloneqq 2+d$ for the effective dimension of $\RR^{1+d}$ equipped with this
distance.
The driving
field $\zeta$ is a smooth, stationary, centred space--time noise with long-range
correlations: For distant points, the covariance is approximately given by
\begin{equation}\label{eq_cov_heuristic}
	C(x) \coloneqq \EE[ \zeta (x) \zeta(0) ] \sim |  x |^{-
	(D-2s)}\,, \qquad s \in \big( 0, \tfrac{D}{2}\big)\,.
\end{equation}
The precise assumptions on $\zeta$ and $F$ are stated in Section~\ref{sec_ass}.

Our goal is to identify the macroscopic limit of \eqref{micro} under the scaling that preserves this covariance structure. We therefore introduce the rescaled noise
\begin{equation}\label{rescaled_noise}
\xi_\varepsilon(x) \coloneqq \varepsilon^{s-D/2}\zeta(x_{0}/\varepsilon^2,x_{1,
\ldots, d}/\varepsilon) \, ,
\end{equation}
and the rescaled solution
$\phi_{\varepsilon}(x) \coloneqq \varepsilon^{\alpha} \phi (x_{0}/\varepsilon^{2},
x_{1, \ldots, d}/\varepsilon)$, where $\alpha \coloneqq s-D/2+2$ is the expected regularity of the limiting field.
The rescaled field $\phi_{\varepsilon}$ then solves the macroscopic equation:
\begin{equation}\label{macro}
	(\partial_{x_{0}}-\Delta)\phi_\varepsilon = \varepsilon^{a+\alpha-2} F(\varepsilon^{-\alpha} \phi_\varepsilon) + \xi_\varepsilon \, .
\end{equation}
A prototypical choice is $a = 2+ 2 \alpha$, which leaves the cubic nonlinearity $F(\phi) = \phi^3$ 
invariant.
We will focus strictly on the singular regime $\alpha<0$, or equivalently $s<D/2-2$ (together with
$s>0$ this excludes $d\leq2$).

Although \eqref{macro} crucially depends on the exact form of $F$, universality suggests
that only the leading-order behaviour of $F$ (under the scaling with respect to $a$) contributes in the limit as $\varepsilon \to 0$. In other words, the macroscopic limit is expected to be insensitive to the finer structure of the microscopic model. 
For example, a wide class of microscopic models is expected to give rise to the $\Phi^{\kk+1}_d$ equation:
\begin{equation}
	(\partial_{x_{0}}-\Delta)\Phi = - \lambda \Phi^{\kk} + \xi\,,
\end{equation}
modulo suitable renormalisation.
Universality of such discrete microscopic systems was investigated in
\cite{2dkacising,ShenWeber,iberti2017convergence,
3dkacising}, and in the continuum in 
\cite{HX_Phi,FG_Phi,EX_Phi,SX_Phi,ZhuZhu,Duc25}.
Further prominent examples of this phenomenon in the SPDE 
literature are weak universality results for the KPZ equation 
\cite{GJ13,HQ_KPZ,HS_KPZ,HX_KPZ,GHM24,KWX_KPZ,kevin,HMW25}. 
However, to date, most works have restricted 
the microscopic noise $\zeta$ to short-range correlations, inevitably leading to a 
macroscopic equation driven by Gaussian white noise. To the best of our knowledge, 
the present work is the first to establish weak universality for singular SPDEs driven by
long-range correlated noise (not necessarily Gaussian) in the full subcritical regime.

For \eqref{macro} to converge to a probabilistically meaningful limit, first and foremost, $\xi_{\varepsilon}$
must converge to a limiting stochastic noise $\xi$. Since our scaling preserves the correlation decay, 
$\xi$ necessarily possesses long-range correlations. 
To see which part of the reaction term survives macroscopically, we fix an odd integer $\kk\geq3$ and choose $a=2+(\kk-1)\alpha>0$. If $F$ is analytic, we may expand it in a polynomial basis $\{W_k\}_{k=0}^\infty$ as
\begin{equation}
	F(\cdot) = \sum_{k =0}^{\infty} F_k W_k(\cdot)\,,
\end{equation}
where $F_k$ are the corresponding coefficients. With the above choice of $a$, the nonlinear term in \eqref{macro} formally becomes
\begin{equation}\label{eq_taylor_W}
\varepsilon^{a+\alpha-2} F(\varepsilon^{-\alpha}\phi_\varepsilon) = \sum_{k =0}^{\infty}
\varepsilon^{\kk\alpha} F_k W_k(\ve^{-\alpha}\phi_\varepsilon) \, .
\end{equation}
If $W_k(\phi)=\phi^k$ were the standard monomials, this power counting would suggest that
the terms with $k>\kk$ disappear in the limit, while the terms with $k<\kk$ have to be
absorbed into renormalising counterterms. This is almost the right heuristic, but
standard monomials are not adapted to the singular fluctuations of the linear equation
and variance divergences of higher-order monomials may cause contributions to lower-order terms, see Example~\ref{exmpl_H5}.

To simplify the identification of these contributions, we instead use a polynomial basis which is adapted to the law of the noise. 
Let $Z$ denote the unique stationary and centred solution of the linear equation
\begin{equation}
	(\partial_{x_{0}}-\Delta)Z = \zeta \quad \textnormal{on } \RR^{1+d} \, .
\end{equation}
Associated to the law of $Z=Z(0,0)$ is the following sequence of polynomials
\begin{equation}\label{eq:poly}
	W_{k}(\phi) \coloneqq \frac{\partial^{k}}{\partial \tau^{k}} \frac{e^{\tau \phi}}{\EE[ e^{\tau Z}]} \bigg\vert_{\tau = 0}\,.
\end{equation}
Note that $ Z $ possesses all moments as a
consequence of $
\zeta$ possessing all moments, see Remark~\ref{rem_momentsZ}, which is sufficient to
define \eqref{eq:poly}.
The polynomials $ W_{k}$ 
 form an Appell sequence, i.e.~$W_{k}' = k W_{k-1}$, and moreover satisfy\footnote{We stress that the polynomials are
not orthogonal, except when $ Z$ is Gaussian.}
$ \EE [ W_{k}(Z) ] =0$ whenever $k \neq 0 $, see Section~\ref{sec_propAppell}.
In this adapted basis, \eqref{eq_taylor_W} becomes
		\begin{equation}\label{eq_FasWeps}
		\begin{aligned}
			\varepsilon^{a+\alpha-2} F(\varepsilon^{-\alpha}\phi_\varepsilon) = \sum_{k =0}^{\infty}
		\varepsilon^{(\kk- k)\alpha} F_k W_{k, \varepsilon}(\phi_\varepsilon) \,,
	\end{aligned}
	\end{equation}
where $ W_{k , \varepsilon } (\cdot) \coloneqq \varepsilon^{ \alpha k} W_{k} ( \varepsilon^{-
\alpha} \cdot )  $.
When $\zeta$ (and consequently $Z$) is Gaussian, $W_{k}$ reduces to the standard $k$--th
Hermite polynomial
\begin{equation}
	\begin{aligned}
		W_{k , \varepsilon}( \cdot ) 
		= H_{k}(\cdot, \varepsilon^{2 \alpha}\EE[Z^{2}])
		=(\ve^{2\alpha}\EE[Z^2])^{k/2} H_k(\cdot/\sqrt{\ve^{2\alpha}\EE[Z^2]})
		\,.
	\end{aligned}
\end{equation}
This structure has been investigated in several works in regularity structures
\cite{HQ_KPZ,HS_KPZ,HX_KPZ,HX_Phi,EX_Phi,KZ25,KWX_KPZ},
paracontrolled distributions \cite{FG_Phi}, and energy solutions and 
stochastic heat kernel analysis \cite{kevin_ejp,kevin}. 
In the following, we work in the realm of regularity structures.

The theory of regularity structures, introduced in Hairer's seminal work \cite{Hai14},
provides a general framework of a local solution theory for subcritical singular SPDEs.
Its central idea is to replace classical Taylor expansions by expansions in a suitable
collection of abstract symbols that encode both polynomial contributions and singular
stochastic terms arising from the equation. In this way, one separates the probabilistic 
task of constructing and renormalising the
stochastic objects associated with the equation from the analytic task of controlling the
solution and its remainders relative to these objects.

We use the coefficient-based viewpoint introduced in \cite{OSSW25}. The relevant
model components can be viewed as formal derivatives of the solution with respect to the
coefficients $F_k$ of the reaction term and with respect to the coefficients of an analytic
perturbation. These derivatives are naturally indexed by multiindices and are related through a hierarchy of equations. Thus, instead of trees, we parametrise the solution in terms of multiindices $\beta$ over $\{k\in\NN\,:\,k\geqslant\kk\text{ and }k\text{ odd}\}\cup\NN^{1+d}$ of the form
\begin{equation}
	\begin{aligned}
		k \mapsto \beta(k) \textnormal{ for odd }k \geqslant \kk , \quad \text{and} \quad
		\mathbf{n} \mapsto \beta( \mathbf{n}) \textnormal{ for }\n\in\NN^{1+d}\,,
	\end{aligned}
\end{equation}
and make the ansatz that the solution of \eqref{macro} is described around a
given base point $ x \in \RR^{1+d}$ in terms of the formal power series
\begin{equation}\label{eq_formal_powser}
	\begin{aligned}
		\phi_{\varepsilon} (\cdot)
		``="
		\sum_{ \beta} z^{ \beta}[ \mathbf{F}, p] \Pi_{x \beta }^{\varepsilon}(\cdot)\,,
	\end{aligned}
\end{equation}
with the multinomials $ z^{\beta}$ parametrised by the coefficients $ \mathbf{F} \coloneqq \{F_{k}\}_{k
=\kk}^{\infty}$ and an analytic function~$p$:
\begin{equation}
	\begin{aligned}
		z^{\beta}[\mathbf{F}, p] 
		\coloneqq
		\prod_{k \geqslant \kk} F_{k}^{\beta(k)} \prod_{\mathbf{n}}
		\Big(
			\frac{1}{ \mathbf{n}!} 
			\partial^{\mathbf{n}}p (0)
		\Big)^{ \beta(\mathbf{n})}
		\,.
	\end{aligned}
\end{equation}
Part of the problem at hand is to determine the counterterms $F_k = F_{k, \varepsilon} $ for $k<\kk$, in order to obtain convergence of $\phi_\ve$.
For this reason we restrict ourselves to multiindices
with $k \geqslant \kk $ in the representation~\eqref{eq_formal_powser}.
Notice that the extension of the parametrisation to include $p$ is necessary, because already the linear equation \eqref{macro} with $ \mathbf{F} \equiv 0 $ can only be solved
up to analytic perturbations.
Further motivation for the ansatz \eqref{eq_formal_powser} can be found in \cite{BOT}.

The objects $\Pi_{x\beta}^{\varepsilon}$ should be understood as the building blocks of a
local description of the solution near $x$, in analogy with monomials in an ordinary
Taylor expansion, but adapted to the singular stochastic structure of the equation.
In fact, this ansatz leads to a hierarchy of PDEs that is indexed by multiindices $ \beta$.
While the formal power series \eqref{eq_formal_powser} is not expected to converge, the theory of regularity
structures allows to reconstruct the solution $ \phi_{\varepsilon}$ from $ \{ \Pi^\ve_{x
\beta} \}_{x, \beta}$. 
The family of stochastic objects $ \{ \Pi^\ve_{x \beta} \}_{x, \beta}$, together with their
change-of-basepoint maps, is called a \emph{model} in the theory of regularity
structures.

We are now ready to state an informal version of our main result:

\medskip

\noindent \textbf{Theorem.}
\textit{
Let $d\geqslant 3$, let $F$ be odd and analytic, and let $\zeta$ be a smooth,
stationary, centred noise with long-range correlations as in \eqref{eq_cov_heuristic},
with $s\in(0,D/2-2)\setminus\QQ$. Assume that $\zeta$ satisfies a suitable spectral gap
inequality and that the rescaled noises $\xi_\varepsilon$ converge in probability
to a limiting noise $\xi$.
Fix an odd integer $\kk\geqslant3$ such that
$a=2+(\kk-1)\alpha>0$. Then the lower-order coefficients
$F_k=F_{k,\varepsilon}$, $k<\kk$, can be chosen so that the model associated with
the solution $\phi_\varepsilon$ of \eqref{macro} converges in probability. Its
limit is the model associated with the renormalised
$\Phi^{\kk+1}_d$ equation driven by the long-range-correlated noise $\xi$:
\begin{equation}\label{eq_limit}
	( \partial_{x_{0}} - \Delta) \varphi_{\varepsilon} = F_{\kk}
	W_{ \kk , \varepsilon}(\varphi_{\varepsilon}) + \xi_{\varepsilon} + \sum_{k < \kk}
	\varepsilon^{(\kk - k ) \alpha} F_{k} W_{k, \varepsilon}( \varphi_{\varepsilon}) \,.
\end{equation}
Equivalently, the formal power series associated with $\phi_\varepsilon$
converges to the formal power series of this limiting dynamical
$\Phi^{\kk+1}_d$ model.
}

\medskip

The precise statement of our main result and the topology of convergence can be found in
Theorem~\ref{prop_main_reg} and Corollary~\ref{cor:convergence} in Section~\ref{sec_mainresult}.\\

Let us highlight the following observations of our main result: 
\begin{itemize}
	\item The above theorem covers the full subcritical regime $a > 0$. For
		instance, when $d=4$ and $\kk = 3$, then for every $\alpha \in (-1, 0)$
		(equivalently $s \in (0,1)$), the model associated to
		$\phi_{\varepsilon}$ converges to the limit of the following $\Phi^{4}_{4
		- 2s}$ model:
\begin{equation}
	( \partial_{x_{0}} - \Delta) \varphi_{\varepsilon} = F_{3}
	W_{3, \varepsilon}(\varphi_{\varepsilon}) + \xi_{\varepsilon} + \varepsilon^{2
	\alpha}F_{1} \varphi_{\varepsilon}\,.
\end{equation}
Notice that in order to see a non-trivial limit,  $\varepsilon^{2 \alpha} F_{1}$ must absorb divergent counterterms of
different orders (for sufficiently small $\alpha$), thus the parameter $F_{1} =
F_{1}(\varepsilon)$ must be carefully tuned with respect to $\varepsilon$. 

\item 
Related to the previous point and to simplify the bookkeeping later, we define $c^{(k)}=-\ve^{(\kk-k)\alpha}F_k$ for $k<\kk$ and use
\eqref{eq_FasWeps}, so that
\eqref{macro} reads
\begin{equation}\label{eq:spde}
	(\partial_{x_{0}} -\Delta)\phi_\ve 
= \sum_{k\geq\kk} \ve^{(\kk-k)\alpha} F_k W_{k , \varepsilon}(\phi_\ve) 
- \sum_{k<\kk} c^{(k)} W_{k , \varepsilon}(\phi_\ve) + \xi_\ve \, .
\end{equation}
Since we obtain the bound $|c^{(k)}|\lesssim\ve^{(\kk-k)\alpha+a}$,
and thus $|F_k|\lesssim\ve^a$ for $k<\kk$, 
the counterterms have a negligible effect on the level of the microscopic equation \eqref{micro}.
\end{itemize}

Moreover, we notice that the subcriticality assumption 
\begin{equation}\label{subcritical}
a=2+(\kk-1)\alpha>0 \, ,
\end{equation}
combined with the noise being in the long-range regime $s>0$, implies 
\begin{equation}\label{varianceblowup}
	\kk \alpha + \frac{D}{2} > 0 \, ,
\end{equation}
which prevents any variance blowups. 

\begin{remark}[Solution theory]\label{rem_solution}
Our main result is restricted to the construction of the model associated with
$\phi_{\varepsilon}$ in \eqref{macro}, and the identification of its limit with the
model corresponding to the $\Phi^{\kk+1}_{d}$ equation. We do not address here
the subsequent step of constructing solutions relative to these models.
For the $\Phi^{4}_{d}$ equation on the torus with periodic boundary
conditions, this program was recently carried out in \cite{BOS}. The corresponding
Cauchy problem on bounded domains is treated in the forthcoming work \cite{BOS26+},
whose only input is the model constructed in the theorem above.
In particular, when combined with \cite{BOS26+}, our main result also yields convergence
at the level of solutions, at least for polynomial nonlinearities $F$.
The same conclusion is expected to hold for non-polynomial nonlinearities; see
\cite{FG_Phi,HX_Phi,KZ25} for related results in the case of Gaussian noise. However,
this extension requires a refined analysis at the level of the solution theory.
\end{remark}

Weak universality results usually require two separate ingredients. The first is
the convergence of an enhanced noise, or model in the language of regularity
structures. This step is delicate because the relevant nonlinearities are often
formally supercritical unless one keeps track of the positive powers of $\ve$
from the weakly nonlinear scaling. The second ingredient is the identification
of the limiting model, and thus of the macroscopic equation obtained from the
microscopic one.

Early results on weak universality based on regularity structures introduce a
symbol~$\mE$ (cf.~\cite[p.~14]{HQ_KPZ}) in order to separate
the powers of $\varepsilon$ arising from the nonlinearity from the
mollification length-scale of the noise.
Subsequent works based on regularity structures, such as
\cite{HX_Phi,HX_KPZ,SX_Phi,KZ25,KWX_KPZ}, follow a similar strategy, while also
extending \cite{HQ_KPZ} to non-polynomial nonlinearities, non-Gaussian noise, and
smoothing operators other than the Laplacian.
In these works, convergence is obtained through problem-specific
decompositions, rather than through the general stochastic estimate machinery
developed in \cite{CH16,HairerSteele,BH25}.
In \cite{GHM24} the authors avoid introducing the symbol $\mE$ and instead
transform the microscopic equation into a suitable system of equations, which
allows to appeal to the general results of \cite{CH16}.\footnote{Actually, the
microscopic equation considered in \cite{GHM24} is quasilinear and driven by an
inhomogeneous noise, which creates additional difficulties that we do not discuss
here.}
While this makes the convergence result systematic in the full subcritical
regime, the identification of the limit still relies on explicit computations 
that become increasingly complex as one approaches criticality.
Alternatively to regularity structures, the paracontrolled approach of \cite{FG_Phi} based on Malliavin calculus, is
directly applicable to non-polynomial nonlinearities, but it does not cover the
full subcritical regime.
Similarly, using the flow equation approach, weak universality results are
established in the full subcritical regime in \cite{Duc25}, although the
identification of the limiting equation remains somewhat implicit. 

By contrast, in the present work, relying on multiindex-based regularity structures and a spectral gap assumption, estimates and convergence of the model are systematically obtained through an inductive scheme following \cite{LOTT24}.
Also, we do not require an extension of the regularity
structure by a symbol $\mE$, as the multiindex-based approach to the model via a power
series ansatz contains naturally the ``good'' powers of $\varepsilon$. This allows us to control directly model components which would
otherwise be supercritical. 
Moreover, the identification of the limit is particularly simple due to a suitable polynomial basis adapted to the law of the noise, with respect to which the multiindex-based model is built. 
To summarize, the present article provides a general approach to weak universality by
systematically constructing the
model and identifying its limit needed as input for regularity structures in
the
setting of not necessarily Gaussian noise, illustrated here for dynamical
$\Phi^{p}$ models.

\subsection*{Outline}

In Section~\ref{sec_assandmainresult} we discuss the precise assumptions on the microscopic
noise $ \zeta$ and state our main results, Theorem~\ref{prop_main_reg} (model estimates) and
Corollary~\ref{cor:convergence} (convergence to the $ \Phi^{ \kk+1}$--model).
The proofs are outlined in Section~\ref{sec_outline}. 
Thereafter, we provide the algebraic framework for the inductive proof in
Section~\ref{sec:setup}.
Sections~\ref{sec:regular}--\ref{sec:singular} contain the details of the induction
step in the proof of Theorem~\ref{prop_main_reg}, and form the bulk of this article. 
The appendix contains several technical ingredients necessary for the proof: 
Appendix~\ref{sec:alg_proofs} contains the proofs of the algebraic setup presented in
Section~\ref{sec:setup},
Appendix~\ref{sec_analytic} states basic analytic ingredients, which we include for
completeness, and 
Appendix~\ref{sec_propAppell} discusses properties of Appell polynomials.

\subsection*{Acknowledgements}

The authors thank Lucas Broux, Pawel Duch, Felix Otto, and Rhys Steele for helpful discussions.
Both authors acknowledge financial support through 
the Deutsche Forschungsgemeinschaft (DFG, German Research Foundation) under Germany's Excellence Strategy EXC 2044--390685587, Mathematics Münster:~Dynamics--Geometry--Structure, 
and the European Research Council (ERC) under the
European Union’s Horizon 2020 research and innovation programme (Grant agreement No.~101045082).
MT acknowledges financial support by the Deutsche Forschungsgemeinschaft (DFG, German Research Foundation) through a Walter Benjamin fellowship -- 552900305.

\section{Assumptions and main result}\label{sec_assandmainresult}

In this section, we discuss the precise assumptions on the driving noise $ \zeta$
of the microscopic equation \eqref{micro}, and state our main results: model estimates in
the sense of regularity structures, as well as convergence of the model to the
one associated to the $\Phi^{\kk+1}$ equation. 
Proofs are deferred to Section~\ref{sec_outline} below. 

\subsection{Assumptions on the law of the noise}\label{sec_ass}

The main quantitative assumption we impose on $\zeta$ is a spectral gap inequality, 
\begin{equation}\label{eq_sg_L2}
\EE|F-\EE F|^2 \lesssim \EE\Big\|\frac{\partial F}{\partial\zeta}\Big\|_{\dot H^{-s}(\RR^{1+d})}^2 \, ,
\end{equation}
assumed to hold for all integrable (i.e.~$\EE|F|<\infty$) and cylindrical functionals $F$, i.e.~of the form 
$F(\zeta) = f(\zeta(\varphi_1),\dots,\zeta(\varphi_N))$ for a smooth function $f\colon\RR^N\to\RR$ and
Schwartz functions $\varphi_1,\dots,\varphi_N\colon\RR^{1+d}\to\RR$. 
The derivative of such a functional is given by 
\begin{equation}\label{frechet}
\frac{\partial F}{\partial\zeta} [\zeta]
= \sum_{i=1}^N \partial_i f(\zeta(\varphi_1),\dots,\zeta(\varphi_N))\varphi_i \, .
\end{equation}
Moreover, the homogeneous fractional (parabolic) Sobolev norm $\|\cdot\|_{\dot H^s(\RR^{1+d})}$ is defined for $s\in\RR$ by 
\begin{equation}
\|f\|_{\dot H^s(\RR^{1+d})}^2
= \int_{\RR^{1+d}} dq \, |q|^{2s} |\hat f(q)|^2 \, ,
\end{equation}
where $\hat f$ denotes the Fourier transform of $f$. 

Overall, we make the following assumptions on the noise:

\begin{assumption}\label{ass_zeta}
The random field $\zeta$ is
\begin{itemize}
\item centered, stationary, invariant under spatial reflection, and invariant under sign change, 
\item smooth and all its derivatives have all moments, i.e.~$\EE|\partial^\n\zeta(0)|^p<\infty$ for
	all $\n\in\NN^{1+d}$ and $p<\infty$, 
\item and it satisfies the spectral gap inequality \eqref{eq_sg_L2} with $s\in(0,D/2-2)\setminus\QQ$. 
\end{itemize} 
Furthermore, the derivative operator defined on cylindrical functionals by \eqref{frechet}
is closable with respect to the topologies induced by $(\EE|\cdot|^2)^{1/2}$
and $(\EE\|\cdot\|^2_{\dot H^{-s}})^{1/2}$.
\end{assumption}

\begin{remark}\label{rem_momentsZ}
Since $\zeta$ and all its derivatives have moments of every order,
the stationary solution $Z$ of the linear equation also has all moments, see Proposition~\ref{prop_qualitative_pi} 
(for $\beta=0$ and $\ve=1$ such that $\Pi^{-,\ve}_{x\beta}=\zeta$ and $\Pi^\ve_{x\beta}=Z$). 
In particular, this justifies the definition of the polynomials $W_k$ in \eqref{eq:poly}. 
\end{remark}

\begin{example}
A basic example of a random field $\zeta$ satisfying Assumption~\ref{ass_zeta} 
is given by a stationary, centred, Gaussian noise with long-range correlations, e.g.~with covariance given by 
\begin{equation}
C(x-y)
\coloneqq \EE[\zeta(x)\zeta(y)]
= (\rho*|\cdot|^{2s-D})(x-y) \,,
\end{equation}
where $\rho$ is a Schwartz function which is invariant under spatial reflection and has a  non-negative Fourier transform.

We first check the spectral gap assumption.
It is well known that a Gaussian random field satisfies the spectral gap inequality \eqref{eq_sg_L2}, 
where the norm on the right-hand side of \eqref{eq_sg_L2} is replaced by the dual of
the Cameron-Martin norm \cite[Theorem~5.5.1]{Bog98}. 
In turn, the dual norm is given by 
\begin{equation}
\Big(\int_{\RR^{1+d}} dx \int_{\RR^{1+d}} dy\, f(x) C(x-y) f(y) \Big)^{1/2}
= \Big(\int_{\RR^{1+d}} dq \, \hat C(q) |\hat f(q)|^2 \Big)^{1/2} \, .
\end{equation}
Since the Fourier transform of the covariance $C$ is up to a constant given by $\hat\rho |\cdot|^{-2s}$, 
which is (again up to a constant) bounded by $|\cdot|^{-2s}$, 
the random field actually satisfies \eqref{eq_sg_L2} as required. 

It remains to identify the scaling limit.
Under the scaling \eqref{rescaled_noise}, the covariance of $\xi_\ve(x)\coloneqq\ve^{s-D/2}\zeta(x_0/\ve^2,x_{1,\dots,d}/\ve)$
is given by $\ve^{2s-D}C(x_0/\ve^2,x_{1,\dots,d}/\ve)$, 
which converges to $|x|^{2s-D}$ as $\ve\to0$.
Hence $\xi_\ve$ converges to the stationary, centred, Gaussian noise 
with Cameron-Martin space $\dot H^s(\RR^{1+d})$ (which still satisfies \eqref{eq_sg_L2}).
In particular, the assumption of Corollary~\ref{cor:convergence} is satisfied.
\end{example}

\begin{example}\label{exmpl_long_range_ng}
We can also start from a not necessarily Gaussian macroscopic noise $\xi$ which satisfies the spectral gap inequality \eqref{eq_sg_L2} with $s\in(0,D/2-2)\setminus\QQ$. 
Then the ($\ve$--dependent\footnote{If the law of $\xi$ is invariant under the scaling $\ve^{D/2-s}\xi(\ve\cdot)$, then the law of $\zeta$ is independent of $\ve$.}) ``smeared out'' version
$\zeta\coloneqq\rho*\ve^{D/2-s}\xi(\ve\cdot)$ for a Schwartz function $\rho$ satisfies the same spectral gap inequality uniformly in $\ve$.\footnote{If $\xi$ satisfies \eqref{eq_sg_L2} then $\ve^{D/2-s}\xi(\ve\cdot)$ satisfies \eqref{eq_sg_L2} with the same constant, and if $\xi$ satisfies \eqref{eq_sg_L2} then $\rho*\xi$ satisfies \eqref{eq_sg_L2} with the same constant multiplied by $\|\rho\|_{L^1(\RR^{1+d})}$.}
Hence $\zeta$ satisfies Assumption~\ref{ass_zeta}, provided $\xi$ is stationary, centred, and invariant under spatial reflection and sign change, and provided $\rho$ is invariant under spatial reflection.

With this choice, the rescaled noise $\xi_\ve(x)\coloneqq\ve^{s-D/2}\zeta(x_0/\ve^2,x_{1,\dots,d}/\ve)$ coincides with $\rho_\ve*\xi$, where $\rho_\ve(x)\coloneqq\varepsilon^{ - D} \rho ( x_{0}/ \varepsilon^2, x_{1, \ldots , d}/ \varepsilon)$, and thus converges to $\xi$ as $\ve\to0$ provided $\int \rho = 1$.
\end{example}

\subsection{Setup and main result}\label{sec_mainresult}

Before stating the main result, we make rigorous the building blocks $\Pi_{x\beta}^{ \varepsilon}$ appearing in the informal ansatz \eqref{eq_formal_powser}. 
First, recall that we consider multiindices $\beta$ over $\{k\geq\kk\,:\, k
\text{ odd}\}\cup\NN^{1+d}$. 
Motivated by the scaling behaviour of $\Pi_{x\beta}^{\varepsilon}$, see \cite[Section~1.10]{BOT}, 
such multiindices are equipped with a natural notion of homogeneity, which records the expected scaling of the corresponding model component under parabolic rescaling.
Furthermore, only a subset of all multiindices is relevant for the ansatz \eqref{eq_formal_powser}, which motivates the following definition. 

\begin{definition}[Homogeneity and population]
The \emph{homogeneity} of a multiindex is defined by 
\begin{equation}\label{e_def_homo}
| \beta|
\coloneqq
 \alpha \Big(1+ (\kk-1) \sum_{k\geq\kk} \beta (k) - \sum_{ \bf{n}} \beta ( {\bf{n}} ) \Big)
+
2 \sum_{k\geq\kk} \beta (k)
+
\sum_{\bf{n}}|{\bf{n}}| \beta ({ \bf{ n}})\,,
\end{equation}
where 
\begin{equation}
|\n|\coloneqq 2n_0+n_1+\dots+n_d \, .
\end{equation}
Here and in the following, all sums and products over $\n$ are indexed by $\n=(n_0,\dots,n_d)\in\NN^{1+d}$.

A multiindex $\beta$ is called \emph{populated} if 
\begin{equation}
\beta = \delta_\n
\quad\textnormal{or}\quad
\beta=\delta_{\kk}+\delta_{\n_1}+\dots+\delta_{\n_{\kk}}
\quad\textnormal{or}\quad
[\beta]\coloneqq\sum_{k\geq\kk}(k-1)\beta(k)-\sum_\n\beta(\n)
\geq0\,.
\end{equation}
If $\beta=\delta_\n$, we say that $ \beta$ is \emph{purely polynomial}.
\end{definition}

It is convenient to subsume $\Pi_{x\beta}^{\varepsilon}$ for all $\beta$ into a single object
$\Pi_{x}^{\varepsilon}$. 
We therefore work with the space of formal power series in the dummy variables $z_k$ for $k\geq\kk$ and $z_\n$ for $\n\in\NN^{1+d}$, denoted by $ \RR [\![ z_{k},
z_{\n} ]\!]$. An element $ \pi \in  \RR [\![ z_{k},
z_{\n} ]\!]$ is of the form 
\begin{equation}
\begin{aligned}
\pi = \sum_{ \beta} \pi_{ \beta} z^{\beta} \quad \text{with}\quad  z^{ \beta}=
\prod_{k\geq\kk} z_{k}^{\beta (k)} \prod_{\n} z_{\n}^{\beta(\n)}\,.
\end{aligned}
\end{equation}
Since $\Pi_{x\beta}^{\varepsilon}$ is a function/Schwartz distribution, we have to point evaluate/test with Schwartz functions $\mathcal{S}(\RR^{1+d})$ in order to obtain a power series in $\RR [\![ z_{k},z_{\n} ]\!]$.
In the proof it is convenient 
to formulate all estimates in terms of a single test function, which is the semigroup
generated by $(\partial_{x_0}-\Delta)(-\partial_{x_0}-\Delta)$. 
We thus define
\begin{equation}
\begin{aligned}
\Pi^\ve_{x \beta t} (y)\coloneqq \big(\Pi^\ve_{x \beta}  \ast \Psi_{t}\big) (y) \,,
\end{aligned}
\end{equation}
where $\Psi_{t}$ is the unique solution of 
\begin{equation}\label{eq_Psi}
\begin{aligned}
\partial_{t} \Psi_{t}+ (\partial_{x_0}-\Delta)(-\partial_{x_0}-\Delta)\Psi_{t}=0
\quad\textnormal{on}\quad
(0,\infty)\times\RR^{1+d} \,,
\end{aligned}
\end{equation}
with Dirac initial condition at $t=0$. 
Estimating $\Pi^\ve_{x\beta}$ with respect to $\Psi_t$ is no loss of generality, 
as estimates with respect to any other Schwartz function can be obtained \cite[Section~3.4]{BOT}.
Note that $\Psi_t$ is a Schwartz function which satisfies the semigroup property $\Psi_t*\Psi_s=\Psi_{t+s}$ and 
\begin{equation}\label{eq_Psi_scale}
\Psi_t(x) = (\sqrt[4]{t})^{-D}
\Psi_{t=1}\big(\tfrac{x_0}{\sqrt{t}},\tfrac{x_{1,\dots,d}}{\sqrt[4]{t}}\big) \, .
\end{equation}
Hence, for $x,y\in\RR^{1+d}$, $\n\in\NN^{1+d}$, $\eta\in\RR$, and $\ve\geq0$
\begin{equation}\label{eq_momentbound_semigroup}
\int_{\RR^{1+d}} dz \, | \partial^\n \Psi_t(x-z) | (\ve+\sqrt[4]{t}+|x-y|+|x-z|)^{\eta}
\lesssim (\sqrt[4]{t})^{-|\n|} (\ve+\sqrt[4]{t}+|x-y|)^{\eta} \, .
\end{equation}

With these notions we can give the definition of a model as it is required for a solution
theory, see \cite[Definition~1]{BOS} and the forthcoming work \cite{BOS26+}.

\begin{definition}[Model]\label{def_model}
A (random) \emph{model} is a triple $(\Pi,\Pi^-,\Gamma)$, 
where $\Pi_x,\Pi^-_x\colon\mathcal{S}(\RR^{1+d})\to\RR[\![z_k,z_\n]\!]$ are a family of (random) linear maps indexed by $x\in\RR^{1+d}$ 
and $\Gamma_{xy}\in\mathrm{End}(\RR[\![z_k,z_\n]\!])$ is a family of (random) endomorphisms indexed by $x,y\in\RR^{1+d}$ such that the following holds (almost surely).

For all $\beta$ and all $\kappa>0$ 
\begin{align}
\|\Pi\|_\beta &\coloneqq 
\sup_{x\in\RR^{1+d}} \sup_{0<t<1} (\sqrt[4]{t})^{-|\beta|+\kappa} (1+|x|)^{-2\kappa}
|\Pi_{x\beta t}(x)| < \infty \, , \\
\|\Pi^-\|_\beta &\coloneqq 
\sup_{x\in\RR^{1+d}} \sup_{0<t<1} (\sqrt[4]{t})^{-|\beta|+2+\kappa} (1+|x|)^{-2\kappa}
|\Pi^-_{x\beta t}(x)| < \infty \, , \\
\|\Gamma\|_\beta &\coloneqq
\sup_\gamma \sup_{x\in\RR^{1+d}} \sup_{x'\neq y'\in B_1(x)} |x'-y'|^{-|\beta|+|\gamma|+\kappa} (1+|x|)^{-2\kappa} |(\Gamma_{x'y'})_\beta^\gamma| < \infty \, ,
\end{align}
where $(\Gamma_{xy})_\beta^\gamma\coloneqq(\Gamma_{xy}z^\gamma)_\beta$.
The map $\Gamma$ is related to $\Pi$ and $\Pi^-$ by
\begin{equation}\label{model_recentre}
\Gamma_{xy}\Pi_y = \Pi_x \, , \quad
(\Gamma_{xy}\Pi^-_y)_\beta = \Pi^-_{x\beta} + \textnormal{ polynomial of degree }\leq|\beta|-2 \, , 
\end{equation}
it is transitive and triangular with respect to the homogeneity in the sense of
\begin{equation}\label{model_gamma_props}
\Gamma_{xy}\Gamma_{yz} = \Gamma_{xz} \, , \quad
(\Gamma_{xy}-\mathrm{id})_\beta^\gamma \neq 0 \implies |\gamma|<|\beta| \, ,
\end{equation}
and satisfies\footnote{with the understanding that $\binom{\n}{\m}$ vanishes if the componentwise $\m\leq\n$ is violated}
\begin{equation}\label{model_gamma_def}
\Gamma_{xy}(\pi_1\pi_2) = (\Gamma_{xy}\pi_1)(\Gamma_{xy}\pi_2) \, , \quad
\Gamma_{xy} z_k = z_k \, , \quad
(\Gamma_{xy})_{\delta_\n}^\gamma = \sum_\m \tbinom{\n}{\m} (y-x)^{\n-\m} 1_{\gamma=\delta_\m} \, .
\end{equation}
For purely polynomial multiindices $\delta_\n$ and for multiindices of the form
$\beta=\delta_{\kk}+\delta_{\n_1}+\dots+\delta_{\n_{\kk}}$ it holds
\begin{equation}\label{model_poly}
\Pi_{x\delta_\n}=(\cdot-x)^\n \, , \quad
\Pi^-_{x\delta_\n} = 0 \, , \quad
\Pi_{x\beta}=0 \, , \quad
\Pi^-_{x\beta} = \tfrac{\kk!}{\prod_\n \beta(\n)!} (\cdot-x)^{\n_1+\dots+\n_{\kk}} \, .
\end{equation}
Finally, for $\beta$ not populated and $\gamma$ populated, 
\begin{equation}\label{model_pop}
\Pi_{x\beta} = 0 \, , \quad
\Pi^-_{x\beta} = 0 \, , \quad
(\Gamma_{xy})_\beta^\gamma = 0 \, .
\end{equation}
\end{definition}

The definition of a model is stated generally for the case where 
$\Pi_{x\beta}$ and
$\Pi^-_{x\beta}$ are Schwartz distributions. 
In our setting, where we aim to build a model for \eqref{macro} driven by the smooth
noise $\xi_\ve$, both $\Pi_{x\beta}^{\varepsilon}$ and $\Pi_{x\beta}^{-,\varepsilon}$ turn out to be smooth functions
and we have further information available on their relation. 
The hierarchy below is obtained by inserting the formal expansion \eqref{eq_formal_powser} for $\phi_\ve$ and the analogous expansion
\begin{equation}
c^{(k)}=\sum_\beta z^\beta[\mathbf{F}] c^{(k)}_{\beta,\ve} 
\end{equation}
for the counterterms into \eqref{eq:spde}, and then comparing coefficients using Lemma~\ref{lem:bruno}. Although this derivation is only informal, we take the resulting hierarchy as the rigorous definition of the building blocks $\Pi^\ve_{x\beta}$. More precisely, we obtain
\begin{equation}\label{pde_Pi}
(\partial_{x_0}-\Delta)\Pi^\ve_{x \beta}
=
\Pi_{x \beta}^{-,\ve} \,,
\end{equation}
for $[\beta]\geq0$, where 
\begin{equation}\label{hierarchy}
\begin{aligned}
\Pi_{x \beta}^{-,\ve}
&\coloneqq 
 \sum_{k\geq\kk} \ve^{(\kk-k)\alpha} 
\sum_{l\geq0}
\sum_{\substack{\delta_k+\beta_1+\dots+\beta_{l}=\beta\\ \beta_i\neq0\forall i}} 
\tbinom{k}{l}
(\mathrm{id}-T_x^{<|\beta|-2})\big(W_{k-l,\ve}(\Pi^\ve_{x 0})
\Pi^\ve_{x \beta_1}\cdots\Pi^\ve_{x \beta_{l}}\big) \\
&\ - \sum_{k<\kk} \sum_{l\geq0} \sum_{\substack{\beta_0+\dots+\beta_l=\beta \\ \beta_i\neq0\forall i=1,\dots,l}} 
\tbinom{k}{l} c^{(k)}_{\beta_0} W_{k-l,\ve}(\Pi^\ve_{x0}) \Pi^\ve_{x\beta_1}\cdots\Pi^\ve_{x\beta_l}
+ 1_{\beta=0} \xi_\ve \, .
\end{aligned}
\end{equation}
Here, $T_x^{<|\beta|-2}$ denotes the operation of taking a Taylor polynomial of (parabolic) order less than $|\beta|-2$ centred at $x$, i.e.
\begin{equation}\label{eq_taylor}
T_x^{<|\beta|-2} f 
\coloneqq \sum_{|\n|<|\beta|-2} \tfrac{1}{\n!} \partial^\n f(x) (\cdot-x)^\n \, .
\end{equation}
Examples of $\Pi^{-,\ve}_{x\beta}$ for simple multiindices are provided in Example~\ref{ex:hierarchy}.

To shorten some expressions it is convenient to use equalities of power series instead of
their componentwise versions, i.e.~we write
\begin{equation}\label{eq_piminus}
\Pi_x^{-,\ve} 
= \sum_{k\geq\kk} \ve^{(\kk-k)\alpha} (\mathrm{id}-T_x^{<|\cdot|-2}) \big(z_k W_{k,\ve}(\Pi^\ve_x)\big) 
- \sum_{k<\kk} c^{(k)} W_{k,\ve}(\Pi^\ve_x) 
+ \xi_\ve 1 \, ,
\end{equation}
where $1$ is the neutral (with respect to multiplication) element of $\RR [\![ z_{k},
z_{\n} ]\!]$. Here $T_x^{<|\cdot|-2}$ has to be understood as acting on power series with
function-valued coefficients: 
for every $\beta$-coefficient we take the Taylor polynomial centred at $x$ to order $<|\beta|-2$. 

Finally, we are ready to state our main result.

\begin{theorem}[Model estimates]\label{prop_main_reg}
Let the random field $\zeta$ satisfy Assumption~\ref{ass_zeta} and denote its rescaling for $\ve>0$ as in \eqref{rescaled_noise} by $\xi_\ve$.
Let $\kk\geq3$ be an odd integer satisfying \eqref{subcritical}.
Then there exist (unique, deterministic, $\ve$-dependent) $c^{(k)}$ for $k<\kk$ satisfying
\begin{equation}
c^{(k)}_\beta\neq0
\quad\implies\quad 
|\beta|<2+k\alpha \, , \quad
\sum_\n\beta(\n)=0 \, , \quad
\sum_{k\geq\kk}\beta(k)\geq2 \, ,
\end{equation}
such that the following holds almost surely. 

For all $x\in\RR^{1+d}$ and $\beta$ there exist smooth $\Pi^\ve_{x\beta}$ and $\Pi^{-,\ve}_{x\beta}$ satisfying \eqref{model_poly} and the first two items of \eqref{model_pop}, 
$\Pi^{-,\ve}_{x\beta}$ is given by \eqref{hierarchy}, and $\Pi^\ve_{x\beta}$ satisfies \eqref{pde_Pi} for $[\beta]\geq0$. 
For $\kappa>0$ sufficiently small (depending on $\beta$) and for all $1\leq p<\infty$ 
\begin{equation}\label{est_Pi_main}
\begin{aligned}
\EE^{ \frac{1}{p}} \big|
\Pi^\ve_{x \beta t} (x)\big|^{p} 
\lesssim \ve^{\kappa\sum_{k>\kk}\beta(k)} \big(\sqrt[4]{t} \big)^{|\beta|-\kappa\sum_{k>\kk}\beta(k)} \, ,
\end{aligned}
\end{equation}
where the implicit constant only depends on $\alpha$, $ \beta$, $\kappa$, and $p$, 
but is independent of $\ve>0$, $x\in\RR^{1+d}$, and $t>0$.

For all $x,y\in\RR^{1+d}$ and all $\beta,\gamma$ 
there exists $(\Gamma^\ve_{xy})_\beta^\gamma$ satisfying \eqref{model_recentre}, \eqref{model_gamma_props}, \eqref{model_gamma_def}, and the last item of \eqref{model_pop}. 
For $\kappa>0$ sufficiently small (depending on $\beta$) and for all $1\leq p<\infty$ 
\begin{equation}\label{est_Gamma_main}
\begin{aligned}
\EE^{ \frac{1}{p}} \big|
\big(\Gamma^\ve_{xy} \big)_{ \beta}^{\gamma} \big|^{p} 
\lesssim \ve^{\kappa \sum_{k>\kk}(\beta(k)-\gamma(k))} |x-y|^{|\beta|-|\gamma|-\kappa\sum_{k>\kk}(\beta(k)-\gamma(k))} \, ,
\end{aligned}
\end{equation}
where the implicit constant only depends on $\alpha$, $ \beta$, $\kappa$, and $p$,
but is independent of $ \ve>0$, $x, y\in\RR^{1+d}$, and $\gamma$. 
\end{theorem}

In particular, by an application of Kolmogorov's argument (see e.g.~\cite[Proposition~B.3]{HairerSteele} and \cite[Proof of Theorem~1.4]{Tem} for an application to the current setting)
the annealed estimates \eqref{est_Pi_main} and \eqref{est_Gamma_main} can be turned into the quenched estimates
in Definition~\ref{def_model}, to obtain:

\begin{corollary}\label{cor:model}
The triple $(\Pi^\ve,\Pi^{-,\ve},\Gamma^\ve)$ in Theorem~\ref{prop_main_reg} is a model. 
\end{corollary}

We also obtain convergence of models provided we assume convergence of the rescaled random field $\xi_\ve$.

\begin{corollary}[Convergence to the dynamical $ \Phi^{\kk+1}$ model]\label{cor:convergence}
Let the random field $\zeta$ satisfy Assumption~\ref{ass_zeta}, 
and assume that its rescaling $\xi_\ve$ defined in \eqref{rescaled_noise} converges in law in the sense of Schwartz distributions to a random field $\xi$. 
Then 
\begin{itemize}
\item the smooth model $(\Pi^\ve,\Pi^{-,\ve},\Gamma^\ve)$ of Theorem~\ref{prop_main_reg} 
converges in law to a model $(\Pi,\Pi^-,\Gamma)$ with respect to the distance induced by
\begin{equation}
\max_{|\beta|\leq2} \big\{\|\Pi\|_\beta+\|\Pi^-\|_\beta+\|\Gamma\|_\beta\big\} \, ,
\end{equation}
\item $\Pi$ and $\Pi^-$ satisfy almost surely \eqref{pde_Pi} for $[\beta]\geq0$.
\end{itemize}
The convergence in the first item also holds in probability (respectively in $L^p$ for all $p<\infty$), provided the underlying sequence of random fields $\xi_\ve$ converges in probability (respectively in $L^p$ for all $p<\infty$) to a random field $\xi$.

Moreover, if we repeat the construction of $(\Pi^\ve,\Pi^{-,\ve},\Gamma^\ve)$ in Theorem~\ref{prop_main_reg} with multiindices only over $\{\kk\}\cup\NN^{1+d}$ 
(in particular all sums and products over $k\geq\kk$ collapse to $k=\kk$)
to obtain a model $(\Pi^\ve_{\kk},\Pi^{-,\ve}_{\kk},\Gamma^\ve_{\kk})$ for \eqref{eq_limit}, 
then $(\Pi^\ve_{\kk},\Pi^{-,\ve}_{\kk},\Gamma^\ve_{\kk})$ converges to a limiting model $(\Pi_{\kk},\Pi^-_{\kk},\Gamma_{\kk})$ which coincides with 
$(\Pi,\Pi^-,\Gamma|_{\RR[\![z_{\kk},z_\n]\!]})$ in the following sense:
For multiindices $\beta$ with $\sum_{k>\kk}\beta(k) = 0$ it holds 
\begin{equation}
\Pi_\beta=\Pi_{\kk,\beta} \, , \quad
\Pi^-_\beta=\Pi^-_{\kk,\beta} \, , \quad
\textnormal{and}\quad (\Gamma)_\beta^\gamma = (\Gamma_{\kk})_\beta^\gamma 
\quad\textnormal{for all $\gamma$ with $\sum_{k>\kk}\gamma(k)=0$} \, , 
\end{equation}
while for $\beta$ satisfying $\sum_{k>\kk}\beta(k)>0$ it holds
\begin{equation}
\Pi_\beta=0 \, , \quad
\Pi^-_\beta=0 \, , \quad
\textnormal{and}\quad (\Gamma)_\beta^\gamma=0 
\quad\textnormal{for all $\gamma$ with $\sum_{k>\kk}\gamma(k)=0$} \, .
\end{equation}
\end{corollary}

\begin{remark}[Positive recentering of $\Pi^{-,\ve}$]\label{rem_pos_recentering}
Let us briefly explain the Taylor remainder $\mathrm{id}-T_x^{<|\beta|-2}$ in the
definition of $\Pi^{-,\ve}_x$ in \eqref{hierarchy}. This term was not present in earlier
works on multiindex-based regularity structures \cite{LOTT24, GT, BOT}. In the present
work it enforces the vanishing of $\Pi^{-,\ve}_{x\beta}$ at $x$ to the order suggested by the expected vanishing of $\Pi^\ve_{x\beta}$.

Indeed, since $\Pi^\ve_{x\beta}$ is expected to vanish at $x$ to order $|\beta|$, 
its right-hand side $\Pi^{-,\ve}_{x\beta}$ should vanish to order $|\beta|-2$. 
This was automatic in \cite{LOTT24, GT, BOT}. 
In the present weak universality setting, the positive powers of $\ve$ in \eqref{hierarchy},
which are accounted for in the homogeneity, create a mismatch between the actual degree of
vanishing $(k-l)\alpha+|\beta_1|+\dots+|\beta_l|$ of
$W_{k-l,\ve}(\Pi^\ve_{x0})\Pi^\ve_{x\beta_1}\cdots\Pi^\ve_{x\beta_l}$ and the expected
degree $|\beta|-2$: the difference is $(\kk-k)\alpha$. We therefore insert the required Taylor remainder ``by hand''. 
Similar adaptations are needed for weak universality results in tree-based regularity structures, 
see e.g.~\cite[Display~(3.23)]{HQ_KPZ}.

By contrast, the counterterm contribution in \eqref{hierarchy} does not involve
positive powers of $\ve$. In view of the previous discussion, no Taylor polynomial has to
be subtracted there, see also \eqref{eq_supp3_regRec} in the proof of Proposition~\ref{prop:regrec}.

The same discussion suggests that the Taylor polynomial is absent for multiindices
only charging $k=\kk$, i.e.~multiindices satisfying $\sum_{k>\kk}\beta(k)=0$, and thus
belonging to the dynamical $\Phi^{\kk+1}$ model (cf. Lemma~\ref{lem:no_taylor_for_phi4}).

We finally note that the presence of the Taylor polynomial has no effect on the solution theory. 
This is due to the fact that the construction of a solution in the forthcoming work \cite{BOS26+} uses only $\beta$-components of the model for $|\beta|\leq2$, 
and for these model components the Taylor polynomial $T_x^{<|\beta|-2}$ is not present.
\end{remark}

\subsection{Outline of the proof of the main results}\label{sec_outline}

The proof of Theorem~\ref{prop_main_reg} proceeds by induction over multiindices. This
induction is carried out in
Sections~\ref{sec:regular},~\ref{sec:intermediate}, and~\ref{sec:singular}, while the
base case is addressed in Section~\ref{sec:inductionstart}.
In fact, Section~\ref{sec:inductionstart} proves more than the base case $ \beta
= \delta_{\n=\mathbf{0}}$: 
it establishes the model estimates of Theorem~\ref{prop_main_reg} for all indices $
\beta$ satisfying $ [ \beta] <0$. Consequently,
Sections~\ref{sec:regular},~\ref{sec:intermediate}, and~\ref{sec:singular} are devoted to
the case $ [ \beta] \geqslant 0$ only.

Before setting up the induction, we must specify an order on multiindices. This order
has to be compatible with the recursive definition of $\Pi^{-,\ve}_{x\beta}$ through the
hierarchy~\eqref{hierarchy}. Namely, at each induction level, the right-hand side of
\eqref{hierarchy} should involve only multiindices belonging to earlier levels.
Similar triangularity conditions are necessary for all auxiliary objects that are defined recursively throughout the induction.
The homogeneity defined in~\eqref{e_def_homo} is not suited to this purpose (the quantity
$ d \Gamma$ introduced in Lemma~\ref{lem_dGamma} below is not triangular with respect to it).
We therefore
introduce the order $|\cdot|_{\prec}$ in~\eqref{e_def_ord}, with respect to which 
all upcoming objects are triangular in the required
sense. 

The estimates stated in Theorem~\ref{prop_main_reg} for one induction level are,
however, not by themselves sufficient to obtain the corresponding estimates at the
next level from~\eqref{hierarchy}. One has to propagate a larger family of auxiliary
estimates. This forms the main part of the paper. Ignoring these auxiliary estimates
for the moment, the basic strategy is to obtain stochastic bounds on
$\Pi^{-,\ve}_{x\beta}$ by a reconstruction argument. The simplest instance of this
argument is explained at the beginning of Section~\ref{sec:regular}.

There are three distinct classes of multiindices, each requiring a different
reconstruction argument.
\begin{itemize}
	\item \emph{Regular multiindices,} which are those of homogeneity $ | \beta|
		>2$. In this
    		case, stochastic estimates for $\Pi^{-,\ve}_{x\beta}$ follow directly from
    		reconstruction. 
		Regular multiindices are treated in Section~\ref{sec:regular}.
    
	\item \emph{Intermediate multiindices}\footnote{In \cite{BOT}, it is not necessary
			to treat these multiindices separately for the dynamical
		$\Phi^{4}_{d}$--model, because $ |
	\beta| = 2 $ implies $ [ \beta] <0 $. Thus, they are already treated in the base case.},
	which are those of homogeneity $ |
		\beta|=2$.
		Here, the reconstruction argument used in the regular case does not apply.
		However, these multiindices have a very restricted structure, which allows the
    		required estimate for $\Pi^{-,\ve}_{x\beta}$ to be verified directly.
		Intermediate multiindices are treated in Section~\ref{sec:intermediate}.
		 
	\item \emph{Singular multiindices,} which are those of homogeneity $ | \beta| <
		2$.  
		In this case,
    		the reconstruction argument requires a more refined analysis. Rather than working
    		directly with $\Pi^{-,\ve}_{x\beta}$, we first reconstruct its Malliavin derivative
    		$\delta \Pi^{-,\ve}_{x\beta}$ and then recover estimates on $\Pi^{-,\ve}_{x\beta}$ by
   		means of the spectral gap inequality. The use of Malliavin calculus in turn
    		requires an additional family of estimates, which must also be propagated through
    		the induction.
		
		The singular case is also where the counterterms $c$ are chosen. These
    		counterterms are fixed so as to ensure the required control of
    		$\Pi^{-,\ve}_{x\beta}$, see Proposition~\ref{prop_bphz_choice}.
		Singular multiindices are treated in Section~\ref{sec:singular}.
\end{itemize}
Once the stochastic estimates for $\Pi^{-,\ve}_{x\beta}$ have been established, they are
lifted to estimates for $\Pi^\ve_{x\beta}$ by an integration argument. The estimates for
$\Gamma^\ve_{xy}$ and $d\Gamma^\ve_{xy}$ are obtained inductively in the same spirit. At the
end of each induction step, the components $\pi^{(\n),\ve}_{xy\beta}$ and
$d\pi^{(\n),\ve}_{xy\beta}$, 
which are introduced in
Lemmas~\ref{lem_def_gamma} and~\ref{lem_dGamma} below, are chosen so as to recentre the newly constructed model
components $\Pi^\ve_{y\beta}$ and $\delta\Pi^\ve_{y\beta}$.

\medskip

Let us point out some key differences to previous works on stochastic estimates and convergence of multiindex-based models \cite{LOTT24, GT, BOT, Tem}. 
The main conceptual difference is that in the present work, the equation under consideration is subcritical only because of the extra factor $\ve^a$. 
Hence, the reconstruction argument (Proposition~\ref{prop_recIII}), which is at the core of the proof, has to be adapted to take this into account. 
This mainly requires to distinguish small scales ($\leq\ve$) and large scales ($\geq\ve$), and to use ultraviolet-divergent bounds (cf.~\eqref{est_Pi_bounded} and \eqref{est_Pi-_bounded}) in a suitable way. 

On a more technical level, the extra factor $\ve^a$ affects the notion of homogeneity, 
which has both an algebraic and an analytic consequence. 
On the one hand, we need to ensure that the algebraic construction of $\Gamma^\ve_{xy}$ and $d\Gamma^\ve_{xy}$ can still be carried out, and that their desired algebraic properties hold. 
This is carried out in Section~\ref{sec:setup}, where the respective proofs, which are straightforward adaptations of the corresponding proofs in \cite{LOTT24, OST23}, are given for completeness in Appendix~\ref{sec:alg_proofs}.
On the analytic side, the modified homogeneity has the effect that products of model components do not necessarily vanish to the desired order. 
As a consequence, the negative part of the model (i.e.~$\Pi^{-,\ve}$) requires additional (positive) recentering, cf.~Remark~\ref{rem_pos_recentering}.
For this reason $\Gamma^\ve_{xy}$ recentres $\Pi^{-,\ve}_y$ only up to a polynomial,
which has to be carried throughout the proof. 
In particular, this polynomial has to be estimated in the proof of Proposition~\ref{prop_2} and in the proofs of the reconstruction arguments Propositions~\ref{prop:regrec} and \ref{prop_recIII}.

A final contribution of the present work is to prove systematically that many model components do not contribute to the limit. Establishing strong enough estimates to ensure this requires a careful analysis which is carried through the entire proof. Vanishing of these model components leads to a particularly simple identification of the limit, which is based on the expansion of the nonlinearity by polynomials adapted to the law of the solution of the linearised equation (in particular depending on the law of the noise and the diffusion $\Delta$). In turn, this requires only minor modifications of the argument compared to the setting with standard monomials.

\medskip

We now turn to the proof of Corollary \ref{cor:convergence}.

\begin{proof}[Proof of Corollary \ref{cor:convergence}]
Convergence (to $0$) of the model components for multiindices $\beta$ with $\sum_{k>\kk}\beta(k)>0$ is already contained in Theorem~\ref{prop_main_reg}.
Convergence of the remaining model components, i.e.~for multiindices $\beta$ with $\sum_{k>\kk}\beta(k)=0$, can be obtained analogously to \cite[Theorem~1.4]{Tem} as follows. 
The main idea is to prove that the model is a Cauchy sequence in $\ve$ by repeating the arguments of the proof of Theorem~\ref{prop_main_reg} for increments in $\ve$.
Compared to the proof of Theorem~\ref{prop_main_reg} this requires mainly two modifications. 
\begin{itemize}
\item First, we rewrite an increment of products as a sum of products of increments by telescoping.
As in \cite[Display (3.14)]{Tem} this is simply done by 
\begin{equation}\label{telescoping}
\pi_1\cdots\pi_k - \pi'_1\cdots\pi'_k
= \sum_{i=1}^k \pi_1\cdots\pi_{i-1}(\pi_i-\pi'_i)\pi'_{i+1}\cdots\pi'_k \, .
\end{equation}
This telescoping preserves algebraic properties like triangularity, 
so that the inductive structure of the proof does not change, 
and also the homogeneities add up in the same way. 
On the right-hand side of \eqref{telescoping} one can therefore apply the usual estimates on the first $i-1$ terms as well as on the last $k-i$ terms,
while the $i$-th term is estimated by (inductively) applying an estimate of increments. 

We emphasize that the identity \eqref{telescoping} is sufficient in all instances of the proof, 
despite the choice of the monomial basis $W_k$. 
To see this, note that potential increments in $\ve$ of $W_{k,\ve}(\Pi^\ve_x)$ would only come up in the two reconstruction arguments, Propositions~\ref{prop:regrec} and \ref{prop_recIII}.
These occur in the respective second steps of the proofs, where continuity in the base point is established,
more precisely in \eqref{tmp05} for the former reconstruction argument
and in \eqref{eq_germ} for the latter. 
However, in both instances it appears in the form of 
\begin{equation}
\Big(A T_y^{<|\cdot|-2}\big(z_kW_{k,\ve}(\Pi^\ve_y)\big)\Big)_\beta 
\quad\textnormal{for } A\in\{\mathrm{id},\Gamma^\ve_{xy},d\Gamma^\ve_{xy}\} \, .
\end{equation}
This term vanishes because we consider $\beta$ with $\sum_{k>\kk}\beta(k)=0$:
any $A\in\{\mathrm{id},\Gamma_{xy},d\Gamma_{xy}\}$ satisfies for such $\beta$ that 
$A_\beta^\gamma=0$ if $\sum_{k>\kk}\gamma(k)>0$ 
(a consequence of \eqref{e_def_gamma} and \eqref{e_def_dgamma}),
and the claim follows from Lemma~\ref{lem:no_taylor_for_phi4}.
\item Second, we ``trade regularity for convergence''.
Actually, the available estimate for the $i$-th term of the right-hand side in \eqref{telescoping} typically comes with an infinitesimally smaller exponent compared to the first $i-1$ and last $k-i$ terms, 
but in return comes with an (infinitesimally small) rate of convergence. 
To propagate this rate of convergence through the inductive proof requires to carry out the same inductive loop of arguments, but for estimates with infinitesimally smaller exponents. 
That this is indeed possible is in complete analogy to the estimates obtained in Theorem~\ref{prop_main_reg} for multiindices with $\sum_{k>\kk}\beta(k)>0$, and we therefore do not repeat the arguments for multiindices with $\sum_{k>\kk}\beta(k)=0$. 
\end{itemize}
This establishes the first part of the corollary concerning convergence of $(\Pi^\ve,\Pi^{-,\ve},\Gamma^\ve)$. 

We turn to the identification of the limit. 
One can repeat Theorem~\ref{prop_main_reg} to construct and estimate a model $(\Pi^\ve_{\kk},\Pi^{-,\ve}_{\kk},\Gamma^\ve_{\kk})$ corresponding to \eqref{eq_limit} by restricting to multiindices over $\{\kk\}\cup\NN^{1+d}$.
Also Corollary~\ref{cor:model} and the first part of Corollary~\ref{cor:convergence} (concerning convergence) hold verbatim for $(\Pi^\ve_{\kk},\Pi^{-,\ve}_{\kk},\Gamma^\ve_{\kk})$.
The identification of the limit is then a simple consequence of the inductive construction, 
since already for $\ve>0$ and multiindices $\beta$ with $\sum_{k>\kk}\beta(k)=0$ it holds
\begin{equation}
\Pi^\ve_\beta=\Pi^\ve_{\kk,\beta} \, , \quad
\Pi^{-,\ve}_\beta=\Pi^{-,\ve}_{\kk,\beta} \, , \quad
\textnormal{and}\quad (\Gamma^\ve)_\beta^\gamma = (\Gamma^\ve_{\kk})_\beta^\gamma 
\quad\textnormal{for all $\gamma$ with $\sum_{k>\kk}\gamma(k)=0$} \, .
\end{equation}
For multiindices $\beta$ with $[\beta]<0$ this is clear from the definitions, cf.~Section~\ref{sec:inductionstart}. 
For multiindices $\beta$ with $[\beta]\geq0$ we can argue by induction: 
Assuming that $\Pi^\ve_{\beta'}=\Pi^\ve_{\kk,\beta'}$ for all $\beta'\prec\beta$ 
yields by Lemma~\ref{lem_dep_pi} that the respective counterterms $c^{(\beta(\0))}_{\beta-\beta(\0)\delta_\0}$ chosen in Proposition~\ref{prop_bphz_choice} coincide, 
and thus also that $\Pi^{-,\ve}_\beta=\Pi^{-,\ve}_{\kk,\beta}$.
The uniqueness part of Proposition~\ref{prop_integration} then yields $\Pi^\ve_\beta=\Pi^\ve_{\kk,\beta}$.
To see that $(\Gamma^\ve)_\beta^\gamma = (\Gamma^\ve_{\kk})_\beta^\gamma$ it suffices to appeal to Lemma~\ref{lem_triang} and the choice of the respective $\pi^{(\n)}_{\beta}$ in Proposition~\ref{prop_3pt}.
That the limiting model $(\Pi,\Pi^-,\Gamma)$ vanishes for multiindices $\beta$ with $\sum_{k>\kk}\beta(k)>0$ follows already from the estimates \eqref{est_Pi_main} and \eqref{est_Gamma_main}.
\end{proof}

\section{Algebraic framework}\label{sec:setup}

Since the remainder of the article is devoted to the proof of Theorem~\ref{prop_main_reg}, which involves $(\Pi^\ve,\Pi^{-,\ve},\Gamma^\ve)$ only for $\ve>0$, we omit the $\ve$--dependence in the notation from now on.

\subsection{Homogeneity and induction order}

The inductive proof is carried out with respect to the \emph{order} of a multiindex, defined by 
\begin{equation}\label{e_def_ord}
\begin{aligned}
| \beta|_{\prec}
\coloneqq
| \beta| + \tfrac{D}{2} \left( 1+ [ \beta] \right)\,.
\end{aligned}
\end{equation}
We write 
\begin{equation}
\begin{aligned}
\beta ' \prec \beta \quad  \iff \quad | \beta'|_{\prec} < | \beta|_{\prec}\,,
\end{aligned}
\end{equation}
and 
\begin{equation}
\begin{aligned}
\beta ' \preceq \beta   \quad  \iff \quad \beta ' \prec \beta \quad \text{or}
\quad \beta' = \beta \,.
\end{aligned}
\end{equation}
It is also useful to define a ``discounted'' homogeneity, 
which discounts multiindices that would not appear when constructing a model for the $\Phi^{\kk+1}$ equation \eqref{eq_limit}.
A suitable candidate is 
\begin{equation}\label{eq_bracket}
\langle\beta\rangle
\coloneqq |\beta| - \kappa \sum_{k>\kk} \beta(k) \,,
\end{equation}
for $0<\kappa<\min\{a,-2\alpha\}$ and such that\footnote{Since $\langle\beta\rangle$ is of the form $\alpha m + \kappa n + l$ for $m,n,l\in\ZZ$ with $(m,n)\neq(0,0)$ for $[\beta]\geq0$, this is ensured e.g.~by choosing $\kappa\in\RR\setminus(\alpha\QQ+\QQ)$.} $\langle\beta\rangle\not\in\NN$ for all $\beta$ with $[\beta]\geq0$. 
At each induction step for a multiindex $\beta$ in Sections~\ref{sec:regular}--\ref{sec:singular}, we choose $\kappa$ small enough that
$n<|\gamma|$ implies $n<\langle\gamma\rangle$ for all populated $\gamma\preceq\beta$ and
$n\in\NN$.
This is possible since by the upcoming Lemma~\ref{lem_ind_legal} there are only finitely many such $\gamma$ and hence also finitely many $n$. 
We anticipate that $\kappa$ is further restricted in the beginning of Section~\ref{sec:singular}.

We then have the following results for the homogeneity, the discounted homogeneity, and the order.

\begin{lemma}\label{lem_hom_props}
The homogeneity defined in \eqref{e_def_homo} has the following properties: 
\begin{itemize}
\item the map $\beta \mapsto | \beta | - |0|$ is additive, 
\item $ | \beta| -|0| \geqslant 0 $,
\item $ | \beta| = |0| $ if and only if $ \beta = 0$.
\end{itemize}
The same properties hold for the discounted homogeneity $\langle\cdot\rangle$.
\end{lemma}

The proof of this Lemma is given in Appendix~\ref{sec:alg_proofs}, 
where one can also find the proofs of all upcoming Lemmas of this section. 

\begin{lemma}\label{lem_ord_props}
The order $| \cdot|_{\prec}$ defined in \eqref{e_def_ord} has the following
properties: 
\begin{itemize}
\item the map $ \beta \mapsto | \beta|_{\prec} - | 0 |_{\prec} $ is additive,

\item if $ [\beta ] \geqslant 0$, then $ | \beta|_{\prec}- | 0|_{\prec}
\geqslant 0$,

\item if $ [\beta ] \geqslant 0$, then  $ | \beta|_{\prec}= | 0|_{\prec}$ if and only if $ \beta =0$,

\item if $ [ \beta] \geqslant 0$, then $| \beta|_{\prec} \geqslant | \beta| +
\tfrac{D}{2} \geqslant \alpha+ \tfrac{D}{2}$,

\item if $\beta$ is populated, then  $| \beta|_{\prec} \geqslant | \beta|$.
\end{itemize}
\end{lemma}

\begin{lemma}\label{lem_ind_legal} 
The order $\prec$ satisfies the following.
\begin{enumerate}
\item For every $ \beta$ we have 
\begin{equation}
\begin{aligned}
\left| \left\{ \gamma \text{ populated} \,:\, \gamma \prec \beta  \right\}  \right| < \infty\,.
\end{aligned}
\end{equation}
\item If $ \beta $ is populated, then $ \delta_{\n=\0}\preceq \beta $.
\end{enumerate}
\end{lemma}

\subsection{Recentering}

In this subsection we provide an algebraic skeleton which is useful later
to define a recentering structure that is triangular with
respect to the order $| \cdot |_{\prec}$, see Lemma~\ref{lem_triang}
below, and allows for an inductive propagation of stochastic estimates.
We point out that the objects $\Gamma$, $d\Gamma$, $T$, and $\widetilde T$ we introduce in the following, have been denoted by $\Gamma^*$, $d\Gamma^*$, $T^*$, and $\widetilde T^*$, respectively, in earlier works on multiindex-based regularity structures \cite{LOT23}. 
This is because the objects are dual in an appropriate sense to the corresponding objects in \cite{Hai14}. 
Since we never make use of this duality, we choose to drop the star to simplify the notation.

\begin{lemma}[Definition and triangularity of $\Gamma$]\label{lem_def_gamma} 
Let $ (\pi^{(\n)} )_{\n\in\NN^{1+d}} \subset \RR [\![ z_{k}, z_{\n} ]\!]$ such that for all $\beta$ and $\n$ it holds that $\pi^{(\n)}_{ \beta} \neq 0$ implies $ |\n| < | \beta| $.
Then 
\begin{enumerate}
\item the map $ \Gamma :  \RR [\![ z_{k}, z_{\n} ]\!] \to  \RR [\![ z_{k}, z_{\n} ]\!]$ defined by 
\begin{equation}\label{e_def_gamma}
\begin{aligned}
\Gamma z_{k}\coloneqq z_{k}\,, \quad 
\Gamma z_{\n}\coloneqq z_{\n}+ \pi^{(\n)}\,, \quad 
\Gamma 1\coloneqq 1\,, \quad \text{and}\quad 
\Gamma ( \pi_{1} \pi_{2}) = (\Gamma \pi_{1}) ( \Gamma \pi_{2})
\end{aligned}
\end{equation} 
for $\pi_1,\pi_2\in\RR[\![z_k,z_\n]\!]$
is well-defined on $\RR [\![ z_{k}, z_{\n} ]\!]$, 
\item for all $\beta,\gamma$ it holds that $ (\Gamma - \mathrm{id})_{ \beta}^{ \gamma}
\neq 0$ implies $ | \gamma| <
| \beta| $ and $\langle\gamma\rangle<\langle\beta\rangle$, and
\item if $ (\pi^{(\n)} )_{\n\in\NN^{1+d}} $ is such that $ \pi^{(\n)}_{ \beta'} = 0$ for all non populated $\beta'$, then for all $\beta,\gamma$
\begin{equation}
\begin{aligned}
(\Gamma - \mathrm{id})_{ \beta}^{ \gamma} \neq 0 \quad \implies \quad \gamma
\prec \beta\,.
\end{aligned}
\end{equation}
\end{enumerate}
\end{lemma}

In Proposition~\ref{prop_3pt} we choose specific $(\pi^{(\n)}_{xy})_\n$ such that the corresponding $\Gamma_{xy}$ satisfies also \eqref{model_recentre}, \eqref{model_gamma_props}, \eqref{model_gamma_def}, and \eqref{model_pop}. 
This choice also respects the population conditions assumed in Lemma~\ref{lem_def_gamma} and the upcoming Lemmas~\ref{lem_sets} and \ref{lem_triang}, i.e.
\begin{equation}\label{eq_pop_pi}
\pi^{(\n)}_{xy\beta}\neq0
\quad\implies\quad
|\n|<|\beta| \textnormal{ and either }
[\beta]\geq0 \textnormal{ or } \beta \textnormal{ is purely polynomial.}
\end{equation}

\begin{lemma}[Definition and triangularity of $d\Gamma$]\label{lem_dGamma} 
Let $\p\geq2$ and $ (d\pi^{(\n)} )_{\n\in\NN^{1+d}} \subset \RR [\![ z_{k}, z_{\n} ]\!]$ such that for all $\beta$ and $\n$ it holds that $
d\pi^{(\n)}_{ \beta} \neq 0$ implies $ |\n| < \alpha+ D/\p $.
Then 
\begin{enumerate}
\item the map $ d\Gamma :  \RR [\![ z_{k}, z_{\n} ]\!] \to  \RR [\![ z_{k}, z_{\n} ]\!]$ defined by 
\begin{equation}\label{e_def_dgamma}
\begin{aligned}
d\Gamma z_{k}\coloneqq 0\,, \quad 
d\Gamma z_{\n}\coloneqq d\pi^{(\n)}\,, \quad \text{and}\quad 
d\Gamma ( \pi_{1} \pi_{2}) 
= (d\Gamma \pi_{1}) ( \Gamma \pi_{2}) + (\Gamma\pi_{1}) ( d\Gamma \pi_{2})
\end{aligned}
\end{equation} 
for $\pi_1,\pi_2\in\RR[\![z_k,z_\n]\!]$ is well-defined on $\RR [\![ z_{k}, z_{\n} ]\!]$,
\item for all $\beta,\gamma$ it holds that $ d\Gamma_{ \beta}^{ \gamma} \neq 0$ implies $ | \gamma| <
| \beta| + D/\p $ and $\langle\gamma\rangle<\langle\beta\rangle+D/\p$, and 
\item if $ (d\pi^{(\n)} )_{\n\in\NN^{1+d}} $ is such that $ d\pi^{(\n)}_{ \beta'} = 0$ for all $\beta'$ with $[\beta']<0$, then for all $\beta,\gamma$
\begin{equation}
\begin{aligned}
d\Gamma_{ \beta}^{ \gamma} \neq 0 \quad \implies \quad \gamma
\prec \beta\,.
\end{aligned}
\end{equation}
\end{enumerate}
\end{lemma}

As for $\Gamma_{xy}$, we choose in Proposition~\ref{prop_intIII} specific $(d\pi^{(\n)}_{xy})_\n$ such that the corresponding $d\Gamma_{xy}$ models $\delta\Pi_x$ via $\Pi_y$, see \eqref{eq:intIII_qualitative}.
This choice also respects the population conditions assumed in Lemma~\ref{lem_dGamma}, i.e.
\begin{equation}\label{eq_pop_dpi}
d\pi^{(\n)}_{xy\beta}\neq0
\quad\implies\quad
|\n|<\alpha+D/\p 
\quad\textnormal{and}\quad 
[\beta]\geq0\,,
\end{equation}
where $\p$ is suitably chosen in Section~\ref{sec:singular}. 

Next, let us introduce subspaces of formal power series that come in handy
throughout. 
First, we define
\begin{equation}
\begin{aligned}
T 
\coloneqq
\left\{ 
\pi = \sum \pi_{ \beta} z^{ \beta} \in \RR [\![ z_{k}, z_{\n} ]\!] \, :\,
\pi_{\beta} \neq 0 \text{ only if } \beta \text{ is populated}
 \right\}\,,
\end{aligned}
\end{equation}
and the subspace
\begin{equation}
\begin{aligned}
\widetilde{T} & \coloneqq \left\{ 
\pi = \sum \pi_{ \beta} z^{ \beta} \in \RR [\![ z_{k}, z_{\n} ]\!] \, :\,
\pi_{\beta} \neq 0 \text{ only if } [\beta] \geqslant 0   \right\}\,. 
\end{aligned}
\end{equation}
The maps $ \Gamma$ and $ d \Gamma$ satisfy the following mapping properties.

\begin{lemma}[Mapping properties of $\Gamma$ and $d\Gamma$]\label{lem_sets} 
We have
\begin{enumerate}
\item $D^{\n} T \subseteq \widetilde{T}$,

\item $ \Gamma \widetilde{T} \subseteq \widetilde{T}$, 
and if $\pi^{(\n)}_\beta\neq0$ implies $[\beta]\geq0$ or $\beta=\delta_\n$, then $ \Gamma T \subseteq T$,

\item $d \Gamma T \subseteq \widetilde{T}$.
\end{enumerate}
\end{lemma}

Additionally, $\Gamma$ and $d\Gamma$ satisfy the following triangularity properties.

\begin{lemma}[Triangular dependence of $\Gamma$ and $d\Gamma$]\label{lem_triang} 
~
\begin{enumerate}
\item If $\pi^{(\n)}$ is such that $\pi^{(\n)}_{\beta'}=0$ for non populated $\beta'$, then for all $\beta,\gamma$
\begin{equation}
\begin{aligned}
\begin{rcases*}
\text{if } \gamma \text{ is populated} \\
\hphantom{\text{if }\gamma}\text{ and not purely polynomial}
\end{rcases*} \quad \text{ then }
\Gamma_{\beta}^{\gamma} \text{ depends on } \pi_{ \beta'}^{(\n)}
\text{ only if} \quad
\begin{cases} 
\beta' \preceq \beta \,.\\
\beta' \prec \beta \,.
 \end{cases} 
\end{aligned}
\end{equation}

\item For all $\beta,\gamma$
\begin{equation}
\begin{aligned}
\begin{rcases*}
\text{if } \gamma \text{ is populated} \\
\hphantom{\text{if }\gamma} \text{ and not purely polynomial} 
\end{rcases*} \quad \text{ then }
d\Gamma_{\beta}^{\gamma} \text{ depends on } d\pi_{ \beta'}^{(\n)}
\text{ only if} \quad
\begin{cases} 
\beta' \preceq \beta \,.\\
\beta' \prec \beta \,.
 \end{cases} 
\end{aligned}
\end{equation}
\end{enumerate}
\end{lemma}

\subsection{The centred model}

We return to a discussion of $\Pi^-_{x\beta}$ defined in \eqref{hierarchy} and first give a couple of explicit expressions for simple multiindices. 

\begin{example}\label{ex:hierarchy}
For the simplest multiindex $\beta=0$ we have 
\begin{equation}
\Pi^-_{x0} = \xi_\ve \, ,
\end{equation}
where we directly used that $c^{(k')}_{\beta=0}=0$, see Remark~\ref{rem:bphz_inductive}. 
Using also that $c^{(k')}_{\delta_k}=0$, see
\eqref{eq_counterterm_deltak}, we have
\begin{equation}
\Pi^-_{x\delta_k+k'\delta_\0} = \ve^{(\kk-k)\alpha} \tbinom{k}{k'}W_{k-k',\ve}(\Pi_{x0}) \, ,
\end{equation}
with the understanding that the expression vanishes for $k'>k$.
For $\kk=3$
\begin{equation}
\Pi^-_{x2\delta_k+\delta_\0} = \ve^{(3-k)\alpha} \Big( k W_{k-1,\ve}(\Pi_{x0}) \Pi_{x\delta_k+\delta_\0} + k(k-1)W_{k-2,\ve}(\Pi_{x0}) \Pi_{x\delta_k} \Big) 
- c^{(1)}_{2\delta_k} \, ,
\end{equation}
and 
\begin{equation}
\Pi^-_{x2\delta_k} = \ve^{(3-k)\alpha} k W_{k-1,\ve}(\Pi_{x0}) \Pi_{x\delta_k} 
- c^{(1)}_{2\delta_k} W_{1,\ve}(\Pi_{x0}) \, .
\end{equation}
For $\kk=3$ and $k\neq k'$ we have
\begin{align}
\Pi^-_{x\delta_k+\delta_{k'}+\delta_\0} 
&= 
\ve^{(3-k)\alpha} 
\Big(
k
W_{k-1,\ve} ( \Pi_{x0}) \Pi_{x\delta_{k'}+ \delta_{\bf{0}}}
+k (k-1)
W_{k -2}( \Pi_{x0}) \Pi_{x\delta_{k'}} \Big) \\
&\,+ \ve^{(3- k')\alpha}
\Big( 
k'
W_{k'-1} ( \Pi_{x0} ) \Pi_{x\delta_{k}+ \delta_{\bf{0}}}
+k' (k'-1)
W_{k' -2}( \Pi_{x0}) \Pi_{x\delta_{k}} \Big) 
- c^{(1)}_{\delta_k+\delta_{k'}} \, ,
\end{align}
and 
\begin{equation}
\Pi^-_{x\delta_k+\delta_{k'}} =  
\ve^{(3-k)\alpha} k W_{k-1,\ve}(\Pi_{x0}) \Pi_{x\delta_{k'}}
+ \ve^{(3-k')\alpha} k' W_{k'-1,\ve}(\Pi_{x0}) \Pi_{x\delta_k} 
- c^{(1)}_{\delta_k+\delta_{k'}} W_{1,\ve}(\Pi_{x0}) \, .
\end{equation}
\end{example}

\begin{example}\label{exmpl_H5}
To illustrate the advantage of expanding $F$ in $W_k$ instead of the standard monomials, let us consider the Gaussian setting where $W_{k,\ve}(\cdot)$ is given by the standard Hermite polynomial 
$H_k(\cdot,\ve^{2\alpha}\EE[Z^2])$.
For $\alpha=-1/2-$ and $D=5$ one can check that $H_5(\Pi^\ve_{x0},\ve^{2\alpha}\EE[Z^2])$ borderline diverges as $\ve\to0$, since $5\alpha+D/2$ is borderline negative. 
Choosing $\kk=3$, we thus observe that $\Pi^{-,\ve}_{x\delta_5}=\ve^{-2\alpha}H_5(\Pi^\ve_{x0},\ve^{2\alpha}\EE[Z^2])$ converges to $0$. 
Similarly, all model components $\Pi^{-,\ve}_{x\delta_k}$ for $k>5$ vanish as $\ve\to0$. 

In contrast, if we worked with the standard monomials $W_k=\cdot^k$, the corresponding model component $\Pi^{-,\ve}_{x\delta_5}$ would be given by 
\begin{equation}
\ve^{-2\alpha} (\Pi^\ve_{x0})^5 - 15 \ve^{2\alpha} \EE[Z^2]^2 \Pi^\ve_{x0} \, .
\end{equation}
Rewriting 
\begin{equation}
(\Pi^\ve_{x0})^5
= H_5(\Pi^\ve_{x0},\ve^{2\alpha}\EE[Z^2])
+ 10 \ve^{2\alpha}\EE[Z^2] H_3(\Pi^\ve_{x0},\ve^{2\alpha}\EE[Z^2]) 
+ 15 (\ve^{2\alpha}\EE[Z^2])^2 \Pi^\ve_{x0} \, ,
\end{equation}
we see that in this case $\Pi^{-,\ve}_{x\delta_5}$ is not vanishing as $\ve\to0$, but converging to the same limit as $10\EE[Z^2]H_3(\Pi^\ve_{x0},\ve^{2\alpha}\EE[Z^2])$. 
Similarly, all model components $\Pi^{-,\ve}_{x\delta_k}$ for $k>5$ converge to a
multiple of the limit of $H_3(\Pi^\ve_{x0},\ve^{2\alpha}\EE[Z^2])$, making the
identification of the equation that the limiting model corresponds to more involved. 
\end{example} 

As Example~\ref{ex:hierarchy} suggests, the following lemma shows that the system of equations \eqref{pde_Pi} and \eqref{hierarchy} is actually a hierarchy of equations.

\begin{lemma}[Triangular dependence of $\Pi^-$]\label{lem_dep_pi} 
~
\begin{enumerate}
\item For every $ \beta$ populated, if $ \Pi^{-}_{x \beta}$ depends on $ \Pi_{x \beta'}$, then $ \beta' \prec
\beta$.

\item For every $ \beta$ populated with $\beta(\0)=0$ and $k,k'<\kk$, 
if $ \Pi^-_{x \, \beta + k\delta_{\0}}+c^{(k)}_\beta$ depends on $c^{(k')}_{ \beta'}$,
then $ \beta '+k'\delta_\0 \prec \beta+k\delta_\0$.
\end{enumerate}
\end{lemma}

\begin{remark}[Choice of counterterms]\label{rem:bphz_inductive}
In the recursive construction of $\Pi^-_x$ we also have to choose $c^{(k)}$ for $k<\kk$.
Since the linear equation ($\mathbf{F}=0$) is not in need of renormalisation, we can set $c^{(k)}_{\beta=0}=0$.
The counterterms $c^{(k)}$ are power series in $\mathbf{F}=\{F_k\}_{k=\kk}^\infty$ only,
so we set $c^{(k)}_\beta=0$ if $\sum_\n\beta(\n)>0$.
For the remaining multiindices we choose $c^{(k)}_\beta$ to make the ensemble and space-time average of $\Pi^-_{x\,\beta+k\delta_\0}$ vanish,
provided $|\beta+k\delta_\0|<2$. 
If $|\beta+k\delta_\0|\geq2$ we simply set $c^{(k)}_\beta=0$. 
For this to be compatible with the recursive construction of $\Pi^-_x$, 
we need to know all $c^{(k')}_{\beta'}$ that show up in $\Pi^-_{x\,\beta+k\delta_\0} + c^{(k)}_\beta$, which are chosen when looking at the multiindices $\beta'+k'\delta_\0$, which in turn are strictly smaller than $\beta+k\delta_\0$ with respect to the ordering $\prec$ by the above lemma. 
\end{remark}

By definition, $\Pi_{x\beta}$ vanishes unless $[\beta]\geq0$ or $\beta$ is purely polynomial.
The following lemma justifies the notion of being populated: 
\begin{lemma}[Population]\label{lem_pop}
Fix $\beta$ and assume that $\Pi_{x\beta'}$ vanishes for $\beta'\prec\beta$ unless $[\beta']\geq0$ or $\beta'=\delta_\n$ is purely polynomial. 
For $k'<\kk$ assume further that $c^{(k')}\in\RR[\![z_k]\!]$ and $c^{(k')}_{\beta=0}=0$.
Then 
	$\Pi^-_{x\beta}$ given by \eqref{hierarchy} vanishes unless $[\beta]\geq0$ or
	$\beta=\delta_{\kk}+\delta_{\n_1}+\dots+\delta_{\n_{\kk}}$.
\end{lemma}

\subsection{Purely polynomial multiindices and base case}\label{sec:inductionstart}

Before we enter the inductive proof of Theorem~\ref{prop_main_reg} for multiindices $[\beta]\geq0$, 
we treat multiindices $[\beta]<0$, including those of the form $\beta=\delta_\n$ and $\beta=\delta_{\kk}+\delta_{\n_1}+\dots+\delta_{\n_{\kk}}$.
We prove that for such multiindices the statements of Theorem~\ref{prop_main_reg} hold
true.
In particular, this covers
the base case $ \beta = \delta_{\n = \mathbf{0}}$ of the induction, see
Lemma~\ref{lem_ind_legal} Item~2.

First, we consider $\Pi_{x\beta}$ and $\Pi^-_{x\beta}$, when viewing \eqref{model_poly}
and the first two items of \eqref{model_pop} as definition of the quantities. 
We emphasize that this definition of $\Pi_{x\beta}$ and $\Pi^-_{x\beta}$ for $[\beta]<0$ is consistent with \eqref{hierarchy}.
Furthermore, $\Pi_{x\beta}$ satisfies trivially the estimate \eqref{est_Pi_main}.

We turn to $(\Gamma_{xy})_\beta^\gamma$, where we define 
\begin{equation}
\pi_{xy\beta}^{(\m)}
\coloneqq 
\begin{cases}
\tbinom{\n}{\m} (y-x)^{\n-\m} \, 1_{\n\neq\m} & \textnormal{for }\beta=\delta_\n \, , \\
0 & \textnormal{for } \beta=\delta_{\kk}+\delta_{\n_1}+\dots+\delta_{\n_{\kk}} 
\textnormal{ or } \beta \textnormal{ not populated}\, ,
\end{cases}
\end{equation}
which satisfies the population condition \eqref{eq_pop_pi}.
As a consequence of \eqref{e_def_gamma}, 
the first two items of \eqref{model_gamma_def} hold 
and $(\Gamma_{xy})_{\delta_\n}^\gamma = \sum_\m (z_\m+\pi^{(\m)}_{xy})_{\delta_\n} 1_{\gamma=\delta_\m}$, so that by the latter also the last item of \eqref{model_gamma_def} holds. 
This last item of \eqref{model_gamma_def} also ensures that \eqref{model_recentre} and \eqref{model_gamma_props} hold for $\beta=\delta_\n$, where the respective first items rely additionally on the binomial formula.
Similarly, \eqref{model_recentre} and \eqref{model_gamma_props} hold as well for $\beta=\delta_{\kk}+\delta_{\n_1}+\dots+\delta_{\n_{\kk}}$ by the binomial formula, 
since \eqref{e_def_gamma} and the last item of \eqref{model_gamma_def} imply 
for $\gamma$ populated 
\begin{align}
(\Gamma_{xy})_{\delta_{\kk}+\delta_{\n_1}+\dots+\delta_{\n_{\kk}}}^\gamma
=
\sum
\tbinom{\bar\n_1}{\m_1}\cdots\tbinom{\bar\n_{\kk}}{\m_{\kk}} 
(y-x)^{\bar\n_1+\dots+\bar\n_{\kk}-\m_1-\dots-\m_{\kk}}\, 
1_{\gamma=\delta_{\kk}+\delta_{\m_1}+\dots+\delta_{\m_{\kk}}} 
\tfrac{\prod_\m\gamma(\m)!}{\kk!} \, ,
\end{align}
where the sum is taken over $\m_1,\dots,\m_{\kk}$ and $\bar\n_1,\dots,\bar\n_{\kk}$ such that $\delta_{\bar\n_1}+\dots+\delta_{\bar\n_{\kk}}=\delta_{\n_1}+\dots+\delta_{\n_{\kk}}$.
The last item of \eqref{model_pop} is a general consequence of Lemma~\ref{lem_sets}~Item~2, and the estimate \eqref{est_Gamma_main} immediately follows from the explicit expressions of $(\Gamma_{xy})_\beta^\gamma$ given above.

This establishes all statements of Theorem~\ref{prop_main_reg} concerning multiindices $[\beta]<0$. 
However, in the inductive proof of the upcoming Sections~\ref{sec:regular}--\ref{sec:singular} treating multiindices $[\beta]\geq0$, further auxiliary statements are used as an induction hypothesis for all populated multiindices $\beta'\prec\beta$. 
Therefore, we establish the induction hypothesis for multiindices of the form $\beta=\delta_\n$ and $\beta=\delta_{\kk}+\delta_{\n_1}+\dots+\delta_{\n_{\kk}}$ 
(the induction hypothesis is not used for non-populated multiindices).
The necessary statements 
in Sections~\ref{sec:regular} and \ref{sec:intermediate} are:
\begin{itemize}
\item 
estimate \eqref{est_pin} of $\pi^{(\n)}_{xy}$,
\item 
divergent bounds \eqref{est_Pi_bounded} and \eqref{est_Pi_taylorremainder} of $\Pi_x$ 
and \eqref{est_Pi-_bounded}, \eqref{est_Pi-_taylorremainder} of $\Pi^-_x$,
\item 
covariances 
\eqref{Pi_shift},
\eqref{Pi_parity},
\eqref{Pi_reflection}
of $\Pi_x$ and 
\eqref{Pi-_shift},
\eqref{Pi-_parity},
\eqref{Pi-_reflection}
of $\Pi^-_x$,
\item 
estimate \eqref{est_Pi-} of $\Pi^-_x$,
\end{itemize}
and additionally in Section~\ref{sec:singular}:
\begin{itemize}
\item 
estimate \eqref{est_deltapin} of $\delta\pi^{(\n)}_{xy}$ and \eqref{est_deltaGamma} of $\delta\Gamma_{xy}$,
\item 
estimate \eqref{est_dGamma_inc} of $d\Gamma_{xy}-d\Gamma_{xz}\Gamma_{zy}$,
\item 
estimate \eqref{est_dpin} of $d\pi^{(\n)}_{xy}$ and \eqref{est_dGamma} of $d\Gamma_{xy}$,
\item 
divergent bounds \eqref{est_deltaPi_bounded} and \eqref{est_deltaPi_taylorremainder} of $\delta\Pi_x$ 
and \eqref{est_deltaPi-_bounded}, \eqref{est_deltaPi-_taylorremainder} of $\delta\Pi^-_x$,
\item 
estimate 
\eqref{est_deltaPi_inc} of $\delta\Pi_x-d\Gamma_{xz}\Pi_z$,
\item 
estimate \eqref{est_deltaPi-} of $\delta\Pi^-_x$ and \eqref{est_deltaPi} of $\delta\Pi_x$.
\end{itemize}
All statements listed as induction hypothesis for Sections~\ref{sec:regular} and \ref{sec:intermediate} 
are immediate consequences of the respective definitions of $\Pi^-_{x\beta}$, $\Pi_{x\beta}$, and $\pi^{(\n)}_{xy\beta}$ given above. 
All statements listed as induction hypothesis for Section~\ref{sec:singular} 
concerning $\delta\Pi^-_{x\beta}$, $\delta\Pi_{x\beta}$, $\delta\pi^{(\n)}_{xy\beta}$, 
and $(\delta\Gamma_{xy})_\beta^\gamma$ for populated $\gamma$ are trivially satisfied: 
the respective definitions of $\Pi^-_{x\beta}$, $\Pi_{x\beta}$, $\pi^{(\n)}_{xy\beta}$, and $(\Gamma_{xy})_\beta^\gamma$ 
given above are all deterministic so that their Malliavin derivatives vanish. 
Also $d\pi^{(\n)}_{xy\beta}$, $(d\Gamma_{xy})_\beta^\gamma$, and $(d\Gamma_{xy}-d\Gamma_{xz}\Gamma_{zy})_\beta^\gamma$ 
trivially satisfy all listed estimates as they vanish as well: 
for such multiindices we simply define $d\pi^{(\n)}_{xy\beta}\coloneqq0$ (in-line with \eqref{eq_pop_dpi}),
while for $(d\Gamma_{xy})_\beta^\gamma$ and $(d\Gamma_{xy}-d\Gamma_{xz}\Gamma_{zy})_\beta^\gamma$ 
this follows from the mapping properties established in Lemma~\ref{lem_sets} Items~2 and 3.

\section{Regular indices}\label{sec:regular}

In this section, we consider a regular multiindex $\beta$, i.e.~$|\beta|>2$  with
$[\beta]\geq0$, which are the least singular ones.
We assume
that all relevant objects associated with smaller multiindices $\beta' \prec \beta$ (both
regular and singular) have already been constructed and estimated. Our aim is to derive
estimates for the model component $\Pi_{x\beta}$, as well as 
defining $ \Gamma_{xy}$ such that 
\begin{equation}
	\begin{aligned}
		\Pi_{x\beta} = (\Gamma_{xy}\Pi_y)_\beta
		=
		\sum_\gamma (\Gamma_{xy})_\beta^\gamma \Pi_{y\gamma}\,,
	\end{aligned}
\end{equation}
which is a consequence of obtaining sufficient control of $\Pi^-_{x\beta}$.

As a first step, we establish stochastic estimates for $\Gamma_\beta^\gamma$, whenever
the corresponding expression factorises into terms depending only on smaller multiindices
(Proposition~\ref{prop_1}). 
This occurs for $\gamma$ that are not purely polynomial. Fortunately, this is sufficient for
our purposes, since $\Pi^-_{x\gamma}$ 
vanishes for purely polynomial $\gamma$ by \eqref{model_poly}.
We therefore postpone the treatment of $ \Gamma_{ \beta}^{\delta_{\n}}$, which arise only in later induction steps, until the end of the argument.
We then estimate $\Pi^-_{x\beta}$ by combining the hierarchy \eqref{hierarchy} with a
reconstruction argument (Proposition~\ref{prop:regrec}). Once $\Pi^-_{x\beta}$ has been constructed and estimated, the corresponding bounds for
	$\Pi_{x\beta}=(\partial_{x_0}-\Delta)^{-1}\Pi^-_{x\beta}$
follow by integration. It then remains to verify that $\Gamma$ indeed recentres the
integrated component 
$\Pi_\beta$, which is achieved by choosing the missing purely polynomial components
$\Gamma_\beta^{\delta_{\n}}=\pi^{(\n)}_\beta$ appropriately (Proposition~\ref{prop_3pt}).

Finally, we reestablish at the present level the probabilistic symmetries and continuity
properties of the model components on which the argument relies, see
Propositions~\ref{prop_covariance_Pi} and~\ref{prop_qualitative_pi} below.
This provides the structural input required for the next stage of the induction.

\medskip

The steps that we outlined above are summarised in the following table, borrowed from
\cite{LOTT24}, which the reader may use as a guiding aid throughout this section:

\begin{figure}[H]
		\small
		\begin{center}
			\begin{tabular}{|c|c|c|}
					\hline 
					{\bf Step} & {\bf Description} & {\bf Statement}\\  
					\hhline{|===|}
					1 & Algebraic argument I & Proposition
					\ref{prop_1} \\ 
					\hline
					2.~a) & Boundedness and continuity of $
					\Pi^{-}$  & Proposition \ref{prop_2} \\
					\hline
					2.~b) & Symmetries of $ \Pi^{-}_{ \beta }$ & Proposition
					\ref{prop_covariance_Pi-} \\
					\hline
					2.~c) & Recentering of $ \Pi^{-}_{\beta}$ & Proposition
					\ref{prop_recentre_Pi-} \\
					\hline
					3 & Regular reconstruction & Proposition
					\ref{prop:regrec} \\
					\hline
					4.~a) & Integration & Proposition
					\ref{prop_integration} \\
					\hline
					4.~b) & Symmetries of $ \Pi_{ \beta}$ & Proposition
					\ref{prop_covariance_Pi} \\
					\hline 
					5 & Three-point argument & Proposition
					\ref{prop_3pt} \\
					\hline
					6 & Boundedness and continuity of $
					\Pi_{ \beta}$ & Proposition \ref{prop_qualitative_pi} \\
					\hline
			\end{tabular} 
		\end{center}
		\caption{Summary of arguments for regular multiindices.}\label{fig:regular}
	\end{figure}

In the remainder of this section, we state and prove each of the steps listed in the
table, starting with the estimates of $ \Gamma_{\beta}^{\gamma}$ in the case when $
\gamma$ is not purely polynomial.

Throughout this section all necessary conditions and resulting estimates hold for all $x,y\in\RR^{1+d}$, $t\in(0,\infty)$, and $p\in[1,\infty)$. Therefore we drop the corresponding quantifiers for notational simplicity.
Moreover, all implicit constants only depend on $ \alpha$, $\beta$, $\kappa$, and $ p$.

\begin{proposition}[Algebraic argument I, Step 1]\label{prop_1} 
Let $[\beta]\geq0$ and assume that for all populated $ \beta' \prec \beta$ 
\begin{equation}\label{est_pin}
\begin{aligned}
\EE^{\frac{1}{p}} \left|
\pi^{(\n)}_{xy\, \beta'}
\right|^{ p} \lesssim \ve^{|\beta'|-\langle\beta'\rangle} |x-y|^{\langle \beta' \rangle - |\n|}\,.
\end{aligned}
\end{equation}
Then for all \(\gamma\) populated and not purely polynomial, we have
\begin{equation}\label{est_Gamma}
\begin{aligned}
\EE^{\frac{1}{p}} \left|
(\Gamma_{xy})_{ \beta}^{\gamma}
\right|^{ p} \lesssim \ve^{|\beta|-\langle\beta\rangle-|\gamma|+\langle\gamma\rangle} |x-y|^{\langle\beta \rangle - \langle\gamma\rangle}\,,
\end{aligned}
\end{equation}
with the implicit understanding that the exponents on the right-hand side in
\eqref{est_Gamma} are non-negative, unless the left-hand side is zero.
\end{proposition}

\begin{proof}
We recall the representation \eqref{eq_gamma_prod_rep} of $
(\Gamma_{xy})_{\beta}^{\gamma}$ together with \eqref{eq_tri_supp2}, which
yields 
\begin{equation}\label{e:supp1_regStep1}
\begin{aligned}
\EE^{\frac{1}{p}}
\left|
( \Gamma_{xy})_{ \beta}^{\gamma}
\right|^{p} 
\lesssim
\sum_{m =0}^{ \infty}
\sum_{\n_{1}, \ldots, \n_{m}}
\sum_{ \beta_{1}+ \cdots + \beta_{m+1}= \beta}
\EE^{\frac{1}{p}}
\left|
\pi^{( \n_{1 })}_{ \beta_{1}}\cdots \pi^{ ( \n_{m})}_{ \beta_{m}}\right|^{p} 
\delta_{ \beta_{m+1}}^{\gamma - \delta_{\n_{1}}- \cdots - \delta_{\n_{m}}}\,,
\end{aligned}
\end{equation}
where we used the crude estimate $\gamma ( \n_{1}) ( \gamma - \delta_{\n_{1}}) (\n_{2}) 
\cdots ( \gamma - \delta_{\n_{1}} - \cdots - \delta_{\n_{m-1}} )(\n_{m})
\lesssim 1$ with a constant only depending on $\beta$ 
(since there are only finitely many $\gamma\preceq\beta$, see Lemma~\ref{lem_def_gamma} Item~3 and Lemma~\ref{lem_ind_legal}). 
The sums over $m$ and $\n_1,\dots,\n_m$ are finite because the multiindex $ \beta$, and
hence $ \gamma$, has only finitely many nonzero components among the $\delta_{\n}$-variables.

Now we apply H\"olders inequality and use assumption \eqref{est_pin}, to
see that 
\begin{equation}\label{eq_prop1_supp1}
\begin{aligned}
\EE^{\frac{1}{p}}
\left|
( \Gamma_{xy})_{ \beta}^{\gamma}
\right|^{p} 
\lesssim 
\sum_{m =0}^{ \infty}
\sum_{\n_{1}, \ldots, \n_{m}}
\sum_{ \beta_{1}+ \cdots + \beta_{m+1}= \beta}
\ve^{\sum_{i=1}^{m} (|\beta_i|-\langle \beta_{i}\rangle)}
| x-y |^{\sum_{i=1}^{m} \langle \beta_{i}\rangle -|\n_{i}|}
\delta_{ \beta_{m+1}}^{\gamma - \delta_{\n_{1}}- \cdots - \delta_{\n_{m}}}\,,
\end{aligned}
\end{equation}
where the implicit constant depends on $\alpha$, $ \beta$, and $p$.
Because $ \beta_{m+1} = \gamma - \delta_{\n_{1}}- \cdots - \delta_{\n_{m}}$, we
can write
\begin{equation}
\begin{aligned}
| \beta| &= | \beta_{1}| + \cdots +| \beta_{m+1}| - m | 0 |
=| \beta_{1}| + \cdots +| \beta_{m}| + | \gamma - \delta_{\n_{1}}- \cdots -
\delta_{\n_{m}}| - m | 0 |\\
& = 
| \beta_{1}| + \cdots +| \beta_{m}| + | \gamma| - |\n_{1}| - \cdots - |\n_{m}|\,,
\end{aligned}
\end{equation}
where we used that $| \delta_{ \n}|= |\n|$, and analogously
\begin{align}
\langle \beta\rangle 
&= \langle \beta_{1}\rangle + \cdots +\langle \beta_{m+1}\rangle - m | 0 |
=\langle \beta_{1}\rangle + \cdots +\langle \beta_{m}\rangle + \langle \gamma - \delta_{\n_{1}}- \cdots -
\delta_{\n_{m}}\rangle - m | 0 |\\
& = 
\langle \beta_{1}\rangle + \cdots +\langle \beta_{m}\rangle + \langle \gamma\rangle - |\n_{1}| - \cdots - |\n_{m}|\,.
\end{align}
Thus, \eqref{eq_prop1_supp1} reads 
\begin{equation}
\begin{aligned}
\EE^{\frac{1}{p}}
\left|
( \Gamma_{xy})_{ \beta}^{\gamma}
\right|^{p} 
\lesssim 
\left\{\sum_{m =0}^{ \infty}
\sum_{\n_{1}, \ldots, \n_{m}}
\sum_{\substack{ \beta_{1}+ \cdots + \beta_{m+1}= \beta\\ \beta_{m +1}= \gamma - \delta_{\n_{1}}- \cdots - \delta_{\n_{m}}}}
\right\}
\ve^{|\beta|-\langle\beta\rangle-|\gamma|+\langle\gamma\rangle}
| x-y |^{\langle\beta\rangle - \langle\gamma \rangle }\,.
\end{aligned}
\end{equation}
Recall that the sum in front is actually finite, thus, it can be absorbed into
a $ \beta$ dependent constant, which
concludes the proof of \eqref{est_Gamma}.

We finally argue that the exponents in \eqref{est_Gamma} are non-negative unless the left-hand side vanishes. 
For the second exponent $\langle\beta\rangle-\langle\gamma\rangle$ this is an immediate consequence of Lemma~\ref{lem_def_gamma} (Item~2).
For the first exponent we note that $|\beta|-\langle\beta\rangle-|\gamma|+\langle\gamma\rangle=\kappa\sum_{k>\kk}(\beta(k)-\gamma(k))$, 
and that $\gamma(k)\leq\beta(k)$ as can be seen e.g.~from \eqref{e:supp1_regStep1}:
any $ \delta_{k}$ contained in $ \gamma$ must also be part of $
\beta_{m +1}$, and hence $ \beta$.
\end{proof}

In the two subsequent propositions, we prove symmetry and continuity properties of $
\Pi^{-}_{\beta}$ that are necessary in the reconstruction argument. 

\begin{proposition}[Boundedness and continuity of $ \Pi^{-}$, Step 2.a)]\label{prop_2} 
	Let $ [\beta]\geq0 $ and assume that 
	\begin{enumerate}
\item for all $ \beta'$ and $m$ such that $ \beta' + m\delta_{\0}\prec \beta$
\begin{equation}\label{est_d_bounded}
\begin{aligned}
| c^{(m)}_{\beta'}| \lesssim \ve^{| \beta'+ m\delta_{\0}|- 2}=
\ve^{| \beta' | - m\alpha - 2}\,,
\end{aligned}
\end{equation}

\item for all populated $ \beta' \prec \beta$ and $\n$
\begin{equation}\label{est_Pi_bounded}
\begin{aligned}
\EE^{\frac{1}{p}} \left|
\partial^\n \Pi_{x \beta'}(y)
\right|^{ p} \lesssim 
\ve^{\alpha-|\n|}
( \ve + |x-y|)^{| \beta' |- \alpha}\,.
\end{aligned}
\end{equation}
\end{enumerate}
Then for every $k$ and $\n$ 
\begin{equation}\label{est_Pi-_bounded_preliminary}
\begin{aligned}
	\ve^{(\kk-k)\alpha}
\EE^{\frac{1}{p}} \left| \partial^\n \big(z_k W_{k, \varepsilon}(\Pi_x)\big)_\beta(y) 
\right|^{ p} 
\lesssim 
\ve^{\alpha-2-|\n|}
( \ve+ | x-y|)^{| \beta| - \alpha} \, ,
\end{aligned}
\end{equation}
\begin{equation}\label{est_Pi-_bounded}
\begin{aligned}
\EE^{\frac{1}{p}} \left| 
\partial^\n \Pi^-_{x\beta}(y) 
\right|^{ p} 
\lesssim 
\ve^{\alpha-2-|\n|}
( \ve+ | x-y|)^{| \beta| - \alpha} \, .
\end{aligned}
\end{equation}
As well as, for every $l\geq 0$
\begin{equation}\label{est_Pi-_taylorremainder_preliminary}
\begin{aligned}
\EE^{\frac{1}{p}} \! \left| \ve^{(\kk-k)\alpha}
(\mathrm{id}-T_z^{\leq l}) \partial^\n \big(z_k W_{k, \varepsilon}(\Pi_x)\big)_\beta(y) 
\right|^{ p} 
\! \lesssim 
\ve^{\alpha-2-|\n|-1-l}
|y-z|^{1+l} 
( \ve+ | x-y| + |x-z|)^{| \beta| - \alpha} \, ,
\end{aligned}
\end{equation}
and
\begin{equation}\label{est_Pi-_taylorremainder}
\begin{aligned}
\EE^{\frac{1}{p}} \left|
(\mathrm{id}-T_z^{\leq l}) \partial^\n \Pi_{x \beta}^{-}(y) 
\right|^{ p} 
\lesssim 
\ve^{\alpha-2-|\n|-1-l}
|y-z|^{1+l}
( \ve+ | x-y| + |x-z|)^{| \beta| - \alpha} \,.
\end{aligned}
\end{equation}
\end{proposition}

\begin{proof}
\textbf{Step 1 \textnormal{(control of $ W_{k, \varepsilon}( \Pi_{x0})$)}.}
First, we establish that for every $ \n \in \NN^{1+d}$
\begin{equation}\label{eq_est_W}
	\begin{aligned}
		\EE^{\frac{1}{p}}\big| \partial^{\n} W_{k , \varepsilon}(
		\Pi_{x0}(y))\big|^{ p}
		\lesssim \ve^{k \alpha-|\n|} \,.
	\end{aligned}
\end{equation}
To this end, we note that by Fa\`a di Bruno (Lemma~\ref{lem:bruno}) we can write
\begin{equation}
	\begin{aligned}
		\frac{1}{\n !} {\partial}^{\n} W_{k , \varepsilon}(
		\Pi_{x0}(y))
		= \sum_{l = 1}^{|n|} 
		\frac{1}{l!} 
		W^{( l)}_{k , \varepsilon} (\Pi_{x0}(y))
		\sum_{\n_{1}+ \cdots + \n_{l}= \n }
		\frac{1}{\n_{1} !} {\partial}^{\n_{1}} \Pi_{x 0}(y)
		\cdots 
		\frac{1}{\n_{l} !} {\partial}^{\n_{l}} \Pi_{x 0}(y)\,.
	\end{aligned}
\end{equation}
In particular, we can reduce the proof of \eqref{eq_est_W} to the case $\n =0$, since by H\"older's inequality 
and \eqref{est_Pi_bounded} (for $ \beta' = 0 $), we have 
\begin{equation}
	\begin{aligned}
		\EE^{\frac{1}{p}}\big| \partial^{\n} W_{k , \varepsilon}(
		\Pi_{x0}(y))\big|^{ p}
		\lesssim 
		\sum_{l =1}^{k}
		\varepsilon^{l \alpha - |\n|}		
		\EE^{\frac{1}{p'}}\big|  W_{k -l , \varepsilon}(
		\Pi_{x0}(y))\big|^{ p'}\,,
	\end{aligned}
\end{equation}
for some $p' \in (1, \infty) $,
where we also used that $ W_{k , \varepsilon}$ forms an Appell sequence.

Now, by definition of $ W_{k- l , \varepsilon}$ 
\begin{equation}
	\begin{aligned}
		\EE^{\frac{1}{p}}\big| W_{k-l , \varepsilon}(
		\Pi_{x0}(y))\big|^{ p}
		= 
		\varepsilon^{ (k-l)\alpha }
		\EE^{\frac{1}{p}}\big| W_{k -l }(
		 \varepsilon^{- \alpha}\Pi_{x0}(y))\big|^{ p}\,.
	\end{aligned}
\end{equation}
Moreover, recalling that $ W_{k-l}$ is a polynomial of degree $k-l$ with finite coefficients (which are
independent of $ \varepsilon$), the expectation in the previous display is
uniformly bounded using~\eqref{est_Pi_bounded}. 

\textbf{Step 2 \textnormal{(proof of \eqref{est_Pi-_bounded_preliminary})}.}
First, by the Fa\`a  di Bruno formula \eqref{eq_faadibruno_W} and Leibniz rule, we can
write
\begin{equation}\label{tmp09}
	\begin{aligned}
		&\ve^{(\kk-k)\alpha} \tfrac{1}{\n!}\partial^\n\big(z_k W_{k, \varepsilon}(\Pi_x)\big)_\beta \\
&= 
\ve^{(\kk-k)\alpha} 
\tfrac{1}{\n!}\partial^\n
\bigg(
\sum_{l\geq0}	
\sum_{\substack{\delta_k+\beta_1+\dots+\beta_{l}=\beta\\ \beta_i\neq0\forall i}} 
\tbinom{k}{l}
W_{k-l, \varepsilon}(\Pi_{x0  })
\Pi_{x \beta_{1}} \cdots
\Pi_{x \beta_{l} }
\bigg)\\
&= \ve^{(\kk-k)\alpha} \sum_{l\geq0}
\sum_{\substack{\delta_k+\beta_1+\dots+\beta_{l}=\beta\\ \beta_i\neq0\forall i}} 
\sum_{\n_0+\dots+\n_{l}=\n}
\tbinom{k}{l}
\bar\partial^{\n_0}W_{k-l, \varepsilon}(\Pi_{x 0})
\bar\partial^{\n_1}\Pi_{x \beta_1} \cdots
\bar\partial^{\n_{l}}\Pi_{x \beta_{l}} \, , 
	\end{aligned}
\end{equation}
where we abbreviated
$\bar\partial^\n\coloneqq\frac{1}{\n!}\partial^\n$.
Hence, by the assumption \eqref{est_Pi_bounded} and \eqref{eq_est_W}, the left-hand side of \eqref{est_Pi-_bounded_preliminary} is (up to a constant) estimated by a finite sum of terms of the form 
\begin{equation}\label{eq_supp1_reg2a}
\ve^{(\kk-k)\alpha} 
\ve^{(k-l)\alpha-|\n_0|}
\ve^{l\alpha-|\n_1|-\dots-|\n_{l}|} (\ve+|x-y|)^{|\beta_1|+\dots+|\beta_{l}|-l\alpha}
\, ,
\end{equation}
with $l,\beta_1,\dots,\beta_{l},\n_0,\dots,\n_{l}$ as in the sum above. 
Using $\n_0+\dots+\n_{l}=\n$ 
and
that $|\cdot|-\alpha$ is additive, we see that \eqref{eq_supp1_reg2a} 
equals
\begin{equation}\label{eq_supp3_reg2a}
\ve^{\kk\alpha-|\n|} 
(\ve+|x-y|)^{|\beta_1+\dots+\beta_{l}|-\alpha}
 = 
\ve^{\kk\alpha-|\n|} 
(\ve+|x-y|)^{| \beta| -2 - \kk \alpha}
\,,
\end{equation}
where we used in the last step that  $\delta_k+\beta_1+\dots+\beta_{l}=\beta$ implies
$|\beta_1+\dots+\beta_{l}|-\alpha
=|\beta|-2-\kk\alpha$. 
Finally, \eqref{est_Pi-_bounded_preliminary} follows because
$\ve^{\kk\alpha}\leq\ve^{\alpha-2}(\ve+|x-y|)^{2+(\kk-1)\alpha}$, which is
as a consequence of $2+(\kk-1)\alpha=a>0$.

\textbf{Step 3 \textnormal{(proof of \eqref{est_Pi-_bounded})}.}
By the definition \eqref{eq_piminus} of $\Pi^-$ and the fact 
$\partial^\n T_x^{<|\beta|-2}=T_x^{<|\beta|-2-|\n|}\partial^\n$, we have
\begin{align}
\partial^\n\Pi^-_{x\beta} (y)
&= \sum_{k \geqslant \kk} \ve^{(\kk-k)\alpha} (\mathrm{id}-T_x^{<|\beta|-2-|\n|})\partial^\n \big(z_k
W_{k, \varepsilon}(\Pi_x)\big)_\beta (y) \\
&- \sum_{k<\kk} \sum_{\beta_0+\beta_1=\beta} c^{(k)}_{\beta_0} \partial^\n \big(W_{k,
\varepsilon} (\Pi_x)\big)_{\beta_1}(y)
+ \partial^\n\xi_\ve(y) 1_{\beta=0} \, .
\end{align}
We note that $c^{(k)}_{\beta_0}$ is only required for $\beta_0+k\delta_{\0}\prec\beta$
although Lemma~\ref{lem_dep_pi} only yields $\beta_0+k\delta_{\0}\preceq\beta$: whenever $\beta_0+k\delta_{\0}=\beta$ we have $c^{(k)}_{\beta_0}=0$ since we consider $|\beta|>2$, cf.~Remark~\ref{rem:bphz_inductive}.
However, this is not the case when Proposition~\ref{prop_2} is applied again in
Section~\ref{sec:block8}, and the estimate \eqref{est_d_bounded} of $c^{(k)}_{\beta_0}$
for $\beta_0+k\delta_{\0}$ has to be included in the assumption.
We estimate the three terms on the right-hand side separately: 
\begin{itemize}
	\item In the first term, we break up the Taylor remainder using the triangle
		inequality and apply~\eqref{est_Pi-_bounded_preliminary}.
		For the term including $ T^{< | \beta| - 2 - |\n|}_{x}$ we use the
		 bound
		(implied by 
		\eqref{est_Pi-_bounded_preliminary})
		\begin{equation}
			\begin{aligned}
\ve^{(\kk-k)\alpha}
\EE^{\frac{1}{p}} \left| \partial^{\n+\m} \big(z_k W_{k, \varepsilon}(\Pi_x)\big)_\beta(x) 
\right|^{ p} 
\lesssim 
\ve^{\alpha-2-|\n|- | \m| +| \beta| - \alpha}\,,
		\end{aligned}
		\end{equation}
		since we can leverage powers of $x-y$ from the application of $ T^{< |
		\beta| - 2 - |\n|}$, see \eqref{eq_taylor}.

	\item For the second term,  we use \eqref{est_d_bounded} and
$\pi_\beta = (z_k\pi)_{\beta+\delta_k}$, such that
\begin{equation}
	\begin{aligned}
		\EE^{\frac{1}{ p}} \big|
	\partial^\n\big(z_k W_{k,
	\varepsilon}(\Pi_x)\big)_{\beta_{1}+ \delta_{k}}\big|^{p}
	\lesssim 
\ve^{-(\kk-k)\alpha} \ve^{\kk\alpha-|\n|} 
(\ve+|x-y|)^{| \beta_{1}+ \delta_{k}| -2 - \kk \alpha}\,,
	\end{aligned}
\end{equation}
which is a consequence of the bound established in \eqref{eq_supp3_reg2a}. Notice that
the derivation of this bound for $ \beta_{1}+ \delta_{k}$ only depends on
\eqref{est_Pi_bounded} for multiindices $ \beta'
\prec \beta$, 
because \eqref{tmp09} contains the sum over multiindices $ \beta'_{i}$ satisfying 
\begin{equation}
\begin{aligned}
	\delta_{k}+ \beta_{1}'+ \cdots + \beta_{l}' = \beta_{1}+ \delta_{k}\,,
\end{aligned}
\end{equation}
i.e.~$ \beta_{1} = \beta_{1}'+ \cdots +\beta_{l}' $. Hence, for every $i =1, \ldots, l $
\begin{equation}\label{eq_supp10_regStep2}
	\begin{aligned}
		|\beta |_{\prec}
		 = | \beta_{0}|_{\prec} + | \beta_{1}' |_{\prec} + \cdots + |
		\beta_{l}' |_{\prec} - l \alpha
		\geqslant (2+ \kk \alpha- \alpha ) + | \beta_{i}' |_{\prec} \,,
	\end{aligned}
\end{equation}
where we used that $ \beta_{ 0}$ does not contain any polynomial components since
otherwise $ c_{\beta_{0}}^{(k)}$ vanishes, see Remark~\ref{rem:bphz_inductive}, and $ |
\cdot |_{\prec} \geqslant \alpha$. Notice that the parenthesis on the right-hand side in
\eqref{eq_supp10_regStep2} are strictly positive due to subcriticality 
\eqref{subcritical}. Hence, $ |
\beta|_{\prec} > | \beta'_{i}|_{\prec} $.

	\item For the third term, we use
$\EE^\frac1p|\partial^\n\xi_\ve(y)|^p\lesssim\ve^{\alpha-2-|\n|}$ (a consequence of
$\xi_\ve=\ve^{\alpha-2}\zeta(\cdot/\ve^2,\cdot/\ve)$ and the assumption that $\zeta$ is
stationary and all its derivatives have all moments).
\end{itemize}
Overall, we derive the bound 
\begin{equation}\label{eq_supp2_reg2a}
	\begin{aligned}
			\EE^{\frac{1}{p}}\big|\partial^\n\Pi^-_{x\beta} (y)\big|^{p}
	&\lesssim 
\ve^{\alpha-2-|\n|}(\ve+|x-y|)^{|\beta|-\alpha} 
+ \sum_{|\m|<|\beta|-2-|\n|} \ve^{\alpha-2-|\n|-|\m|+|\beta|-\alpha} |x-y|^{|\m|} \\
&+\sum_{k<\kk} \sum_{\beta_0+\beta_1=\beta} 
\ve^{|\beta_0|-k\alpha-2} 
\ve^{k \alpha-|\n|} (\ve+|x-y|)^{|\beta_1+\delta_k|-2-\kk\alpha} 
+ \ve^{\alpha-2-|\n|} 1_{\beta=0} \, .
	\end{aligned}
\end{equation}%
Since the restriction $|\m|<|\beta|-2-|\n|$ implies in particular
$|\beta|-|\m|-\alpha\geq0$ (used for the second term), 
and because $|\cdot|-\alpha$ is non-negative and additive (used for the third term), the
right-hand side \eqref{eq_supp2_reg2a}
implies the desired bound \eqref{est_Pi-_bounded}.

\textbf{Step 4 \textnormal{(proof of \eqref{est_Pi-_taylorremainder_preliminary} and \eqref{est_Pi-_taylorremainder})}.}
Both \eqref{est_Pi-_taylorremainder_preliminary} and \eqref{est_Pi-_taylorremainder} are
an immediate consequence of Lemma~\ref{lem_interpolation} and the already established
bounds \eqref{est_Pi-_bounded_preliminary} and \eqref{est_Pi-_bounded}, respectively. 
\end{proof}

\begin{proposition}[Covariances of $\Pi^-$, Step 2.b]\label{prop_covariance_Pi-}
Let $[\beta]\geq0$ and assume that for all populated $\beta'\prec\beta$
\begin{align}
\Pi_{x\beta'}[\xi(\cdot+y)]
&=
\Pi_{x+y \beta'}[\xi](\cdot+y) 
\quad \textnormal{ for $y \in \RR^{1+ d} $}\, , \label{Pi_shift} \\
\Pi_{x\beta'}[-\xi]
&=
(-1)^{1+[\beta']} \Pi_{x\beta'}[\xi] \, , \label{Pi_parity} \\
\Pi_{x\beta'}[\xi(R\cdot)] 
&= (-1)^{\sum_\n |\n|\beta'(\n)} \Pi_{Rx \beta'}[\xi](R\cdot)
\quad\textnormal{ where $R x\coloneqq ( x_{0}, - x_{1, \ldots, d}) $}\, , \label{Pi_reflection}
\end{align}
where we abbreviated  $ \xi= \xi_{\varepsilon}$.
Then for all $c^{(k')}\in\RR[\![z_k]\!]$ 
\begin{align}
\Pi^-_{x\beta}[\xi(\cdot+y)]
&= \Pi^-_{x+y \beta}[\xi](\cdot+y)
\quad  \textnormal{ for $y\in \RR^{1+d}$}\, , \label{Pi-_shift} \\
\Pi^-_{x\beta}[-\xi]
&= (-1)^{1+[\beta]} \Pi^-_{x\beta}[\xi] \, , \label{Pi-_parity} \\
\Pi^-_{x\beta}[\xi(R\cdot)] 
&= (-1)^{\sum_\n |\n|\beta(\n)} \Pi^-_{Rx \beta}[\xi](R\cdot)
\, . \label{Pi-_reflection}
\end{align}
\end{proposition}

\begin{proof}
	The proof is a consequence of \eqref{hierarchy} and the symmetries of smaller
	multiindices that hold by assumption. 
Similar calculations have been carried out in \cite[Section~1.12]{BOT} and
\cite[Section~5.1]{LOTT24}.
Thus, we only illustrate the argument by 
focusing on summands arising from the Taylor
polynomial in
\eqref{hierarchy}, which are of the form 
\begin{equation}
	\begin{aligned}
		T_x^{<|\beta|-2}\big(W_{k-l, \varepsilon}(\Pi_{x 0})
\Pi_{x \beta_1}\cdots\Pi_{x \beta_{l}}\big) \,,
	\end{aligned}
\end{equation}
with $ \beta = \delta_{k}+ \beta_{1}+ \ldots + \beta_{l}$.
The corresponding identities in \eqref{Pi-_shift}, \eqref{Pi-_parity}, and \eqref{Pi-_reflection} then hold
respectively:
\begin{itemize}
	\item \eqref{Pi-_shift}. Since \eqref{Pi_shift} holds for all indices $ \beta' \prec \beta$, we can
		write 
		\begin{equation}
			\begin{aligned}
				&T_x^{<|\beta|-2}\big(W_{k-l, \varepsilon}(\Pi_{x
					0}[\xi( \cdot +y)])
\Pi_{x \beta_1}[\xi( \cdot +y)]\cdots\Pi_{x \beta_{l}}[\xi( \cdot +y)]\big)\\
				&=				\sum_{ \substack{  | \m| <| \beta| -2 }}
								\frac{1}{\m !} 
				\Big(\partial^{\m}
\big(W_{k-l, \varepsilon}(\Pi_{x+y 0})
\Pi_{x+y \beta_1}\cdots\Pi_{x+y \beta_{l}}\big) (x +y)\Big)\, ( \cdot - x +( y-y) )^{\m}\\
				& = 
				T_{x +y}^{< | \beta| -2}
				\big(W_{k-l, \varepsilon}(\Pi_{x+y 0})
\Pi_{x+y \beta_1}\cdots\Pi_{x+y \beta_{l}}\big) ( \cdot +y)\,,
			\end{aligned}
		\end{equation}
		as desired.

	\item \eqref{Pi-_parity}. Here the Taylor polynomial has no effect. The argument is verbatim to the term
		that corresponds to $ \mathrm{id}$ in \eqref{hierarchy}. 
		Let us point out the importance of $F_{k} = 0 $ for all even $k $, since
otherwise \eqref{Pi-_parity} would not hold true.

	\item \eqref{Pi-_reflection}. Again by evaluating the Taylor polynomial
		explicitly and using the reflection property for smaller multiindices
		\eqref{Pi_reflection}, we see that 
				\begin{equation}
			\begin{aligned}
				&T_x^{<|\beta|-2}\big(W_{k-l, \varepsilon}(\Pi_{x
					0}[\xi(R \cdot)])
\Pi_{x \beta_1}[\xi(R  \cdot)]\cdots\Pi_{x \beta_{l}}[\xi( R \cdot)]\big)\\
& =
( -1)^{ \sum_{i =1}^{l} \sum_{\n} |\n | \beta_{i}( \n )}
\!\!\!\!\!\!\sum_{ \substack{  | \m| <| \beta| -2 }}
\tfrac{1}{\m !} 
				\partial^{\m}
\Big(				\big(W_{k-l, \varepsilon}(\Pi_{Rx0})
					\Pi_{ Rx \beta_1}\cdots\Pi_{
				Rx \beta_{l}}\big) ( R \cdot\,) \Big)(x) (\cdot -
				x)^{\m}\\
			& = (-1)^{ \sum_{\n} | \n | \beta( \n )}
			\!\!\!\!\!\!\sum_{ \substack{  | \m| <| \beta| -2 }}
				\tfrac{1}{\m !} 
				(-1)^{ \sum_{i=1}^{d} m_{i}}
				\partial^{\m}
				\big(W_{k-l, \varepsilon}(\Pi_{Rx0})
					\Pi_{ Rx \beta_1}\cdots\Pi_{
			Rx \beta_{l}}\big) ( R x) (R(R\cdot - Rx))^{\m}\\
			& = 
(-1)^{ \sum_{\n} | \n | \beta( \n )}
T_{Rx}^{< | \beta| -2}
\big(W_{k-l, \varepsilon}(\Pi_{Rx0})
\Pi_{Rx \beta_1}\cdots\Pi_{R x \beta_{l}}\big)
				(R \cdot)\,,
			\end{aligned}
		\end{equation}
		where the $-1$ factors from the chain rule cancel precisely with the ones
		arising from $ ( - ( (R \cdot)_{i} - (R x)_{i})^{m_{i}} $.
\end{itemize}
This finishes the proof.
\end{proof}

The last ingredient for the reconstruction argument is that $\Gamma_{xy}$ recentres $\Pi^-_y$ up to a polynomial.

\begin{proposition}[Recentering of $\Pi^-$, Step 2.c]\label{prop_recentre_Pi-}
Let $[\beta]\geq0$ and assume that $(\Gamma_{xy}\Pi_y)_{\beta'}=\Pi_{x\beta'}$ for all populated $\beta'\prec\beta$.
Then for every $c^{(k')}\in\RR[\![z_k]\!]$, we have
\begin{align}
(\Gamma_{xy}\Pi_y^-)_\beta 
&= \Pi_{x\beta}^- 
+ \Big( 
\sum_{k \geqslant \kk} \ve^{(\kk-k)\alpha} T_x^{<|\cdot|-2}\big(z_k W_{k, \varepsilon}(\Pi_x)\big)\Big)_\beta \label{Pi-_recentre} \\
&\hphantom{= \Pi_{x\beta}^-\ }
- \Big(\Gamma_{xy} 
\sum_{ k \geqslant \kk} \ve^{(\kk-k)\alpha} T_y^{<|\cdot|-2}\big(z_k W_{k, \varepsilon}(\Pi_y)\big)\Big)_\beta \\ 
&=\Pi_{x\beta}^- 
+ \Big(
\sum_{ k \geqslant \kk} \ve^{(\kk-k)\alpha} (T_x^{<|\cdot|-2}-T_y^{<|\cdot|-2}) \big(z_k W_{k,
\varepsilon}(\Pi_x)\big)\Big)_\beta \\
&\hphantom{=\Pi_{x\beta}^-\ }
+ \Big( \Gamma_{xy} 
\sum_{ k \geqslant \kk} \ve^{(\kk-k)\alpha} (T_y^{<|\beta|-2}-T_y^{<|\cdot|-2})
\big(z_k W_{k, \varepsilon}(\Pi_y)\big) \Big)_\beta \, . 
\end{align}
In particular, $(\Gamma_{xy}\Pi^-_y-\Pi^-_x)_\beta$ is a (random) polynomial of degree $<|\beta|-2$. 
\end{proposition}

\begin{proof}
\textbf{Step 1.}
The first equality follows from the definition \eqref{eq_piminus} of $\Pi^-_x$ once we establish 
\begin{equation}\label{tmp04}
\big(\Gamma_{xy} 
z_k W_{k, \varepsilon}(\Pi_y) \big)_\beta
= \big(
	z_k W_{k, \varepsilon}(\Pi_x) 
\big)_\beta 
\quad\textnormal{and}\quad
\big(\Gamma_{xy} c^{(k)} W_{k, \varepsilon}(\Pi_y) \big)_\beta
= \big( c^{(k)} W_{k, \varepsilon} (\Pi_x) \big)_\beta 
\, .
\end{equation}
To see these identities we use that $W_{k, \varepsilon}(\Pi_y)$ is a polynomial in $\Pi_y$,
and that $\Gamma_{xy}$ is multiplicative and coincides with the identity on $\RR[\![z_k]\!]$ 
(in particular $\Gamma_{xy}c^{(k)}=c^{(k)}$), see~Lemma~\ref{lem_def_gamma}. 
Thus, \eqref{tmp04} follows
by the triangular dependence established in Lemma~\ref{lem_dep_pi}~(1) and the assumption.

\textbf{Step 2.} To prove the second equality we spell out 
\begin{align}
&\Big(\Gamma_{xy} 
T_y^{<|\cdot|-2}\big(z_k W_{k, \varepsilon}(\Pi_y)\big)\Big)_\beta \\
&= \sum_\gamma (\Gamma_{xy})_\beta^\gamma 
\sum_{|\n|<|\gamma|-2} \tfrac{1}{\n!} \partial^\n \big(z_k W_{k, \varepsilon}(\Pi_y)\big)_\gamma(y) (\cdot-y)^\n \\
&= \sum_\gamma (\Gamma_{xy})_\beta^\gamma 
\sum_{|\n|<|\beta|-2} \tfrac{1}{\n!} \partial^\n \big(z_k W_{k, \varepsilon}(\Pi_y)\big)_\gamma(y) (\cdot-y)^\n \\
&\,-\sum_\gamma (\Gamma_{xy})_\beta^\gamma 
\sum_{|\gamma|-2\leq|\n|<|\beta|-2} \tfrac{1}{\n!} \partial^\n \big(z_k W_{k,
\varepsilon}(\Pi_y)\big)_\gamma(y) (\cdot-y)^\n \, ,
\end{align}
where in the second inequality we used that $(\Gamma_{xy})_\beta^\gamma$ 
vanishes unless $|\gamma| \leqslant |\beta|$ by Lemma~\ref{lem_def_gamma}~(2). 
The first term on the right-hand side coincides with 
\begin{equation}
\Big(T_y^{<|\cdot|-2} \big(\Gamma_{xy} 
z_k W_{k, \varepsilon}(\Pi_y)\big) \Big)_\beta 
\stackrel{\eqref{tmp04}}{=} 
\Big( T_y^{<|\cdot|-2} 
\big(
z_k W_{k, \varepsilon}(\Pi_x) 
\big) \Big)_\beta 
\, ,
\end{equation}
while the second term coincides with 
\begin{equation}
- \Big( \Gamma_{xy} 
(T_y^{<|\beta|-2}-T_y^{<|\cdot|-2})\big(z_k W_{k, \varepsilon}(\Pi_y)\big) \Big)_\beta \, ,
\end{equation}
and the claim follows. 

\textbf{Step 3.} 
We finally argue that $(\Gamma_{xy}\Pi^-_y-\Pi^-_x)_\beta$ is a polynomial of degree $<|\beta|-2$. 
This is an immediate consequence of the triangularity of $\Gamma_{xy}$ with respect to the homogeneity $|\cdot|$, see Lemma~\ref{lem_def_gamma}~(2), 
and the presence of $T_x^{<|\cdot|-2}$ respectively $T_y^{<|\cdot|-2}$.
\end{proof}

Now, we can finally present the reconstruction argument for regular indices. 

\begin{proposition}[Regular reconstruction, Step 3]\label{prop:regrec}
	Let $ [\beta]\geq0$ be a regular multiindex, i.e.~$| \beta| > 2 $,
	such that
	$\eqref{est_Pi-_bounded_preliminary}_{ \beta}$,
	$\eqref{est_Pi-_taylorremainder_preliminary}_{\beta}$,
	$\eqref{est_Pi-_taylorremainder}_{ \beta}$, and $\eqref{est_Gamma}_{ \beta }^{\gamma}$ hold for all $\gamma$ populated and not purely polynomial. 
	If all populated $ \beta' \prec \beta$ satisfy, $\eqref{est_Pi}_{ \beta' }$ and $\eqref{est_Pi-}_{ \beta' }$,
then 
\begin{equation}\label{est_Pi-}
\EE^\frac{1}{p} \left|
\Pi_{x\beta t}^-(y) \right|^p
\lesssim 
\ve^{|\beta|-\langle\beta\rangle} 
\big(\sqrt[4]{t}\big)^{\alpha-2 } (\sqrt[4]{t}+|x-y|)^{\langle\beta\rangle-\alpha} \,.
\end{equation}
\end{proposition}

Before turning to the proof, recall that reconstruction is, in general, the task of
constructing a distribution that is well approximated by a family of distributions $\{F_x\}_{x\in\RR^{1+d}}$ when evaluated near $ x $. Since $ F_x=\Pi_{x\beta}^{-} $ vanishes on the
diagonal, see \eqref{bphz} below, estimate \eqref{est_Pi-} is precisely the
``reconstruction bound'' of~\cite[Display (3.2)]{BrouxCaravennaZambotti} for reconstructing the zero function.

\begin{proof}
As is typical for a reconstruction argument, 
the proof is based on ``\emph{vanishing at the base point}'' in the form of 
\begin{equation}\label{bphz}
\lim_{t\searrow0}
\EE^\frac{1}{p} \left|
\Pi_{x\beta t}^- (x) \right|^p
=0 \, ,
\end{equation}
which we establish in Step~1 below, 
and ``\emph{continuity in the base point}'' in the form of 
\begin{equation}\label{eq_supp2_regRec}
\EE^\frac{1}{p} \left|
\big(\Pi_{x\beta}^- -\Pi_{y\beta}^-\big)_t(y) \right|^p
\lesssim 
\ve^{|\beta|-\langle\beta\rangle}
\big(\sqrt[4]{t}\big)^{\alpha-2} (\sqrt[4]{t}+|x-y|)^{\langle\beta\rangle-\alpha} \, ,
\end{equation}
which we establish in Step~2 below. 
By a standard reconstruction argument which requires $\langle\beta\rangle-2>0$, see e.g.~\cite[Lemma~4.8]{LO22}, 
\eqref{bphz} and \eqref{eq_supp2_regRec} imply the desired bound \eqref{est_Pi-}. 

\textbf{Step 1 \textnormal{(vanishing at the base point)}.}
We prove that $  \Pi^-_{x\beta}(x)=0$. 
Together with the continuity of $\Pi^-_{x\beta}$ established in
\eqref{est_Pi-_taylorremainder}, this implies \eqref{bphz}.
To see that $\Pi^-_{x\beta}(x)$ vanishes, we recall definition \eqref{hierarchy} and
first note that
the Taylor remainder $
(\mathrm{id}-T_x^{<|\beta|-2}) $ clearly vanishes on
the diagonal (since $ | \beta| >2$).
Moreover, for $\beta_0+\dots+\beta_l=\beta$ and\footnote{The proof only requires $ \n=0$,
however, in view of Remark~\ref{rem_pos_recentering} we prove \eqref{eq_supp3_regRec} more generally.} $\n_0+\dots+\n_l=\n$
\begin{equation}\label{eq_supp3_regRec}
	\begin{aligned}
		c^{(k)}_{\beta_0}\partial^{\n_0} W_{k-l,\ve}(\Pi_{x0})(x)
\partial^{\n_1}\Pi_{x\beta_1}(x)\cdots\partial^{\n_l}\Pi_{x\beta_l}(x)
		= 0\,,
	\end{aligned}
\end{equation}
because 
Remark~\ref{rem:bphz_inductive} implies that $c^{(k)}_{\beta_0}\neq0\implies|\beta_0+k\delta_\0|<2$, 
where the latter condition is equivalent to $|\beta_0|<2+k\alpha$. 
Now, $|\n|<|\beta|-2$ implies 
\begin{equation}
	\begin{aligned}
		|\n|
<\langle\beta\rangle-2
&=\langle\beta_0\rangle+\dots+\langle\beta_l\rangle-l\alpha-2\\
&<\langle\beta_1\rangle+\dots+\langle\beta_l\rangle+(k-l)\alpha
\leq\langle\beta_1\rangle+\dots+\langle\beta_l\rangle\,.
	\end{aligned}
\end{equation}
Thus, for $\partial^{\n_1}\Pi_{x\beta_1}(x)\cdots\partial^{\n_l}\Pi_{x\beta_l}(x)$ not to vanish, 
we need $\langle\beta_i\rangle\leq|\n_i|$ for all $i=1,\dots,l$ 
as a consequence of continuity \eqref{est_Pi_taylorremainder} and the model estimate \eqref{est_Pi},
which yields the contradiction
$$|\n|<\langle\beta_1\rangle+\dots+\langle\beta_l\rangle
\leq|\n_1|+\dots+|\n_l|
\leq|\n|.$$
Hence, the sum over counterterms in \eqref{hierarchy} of the form \eqref{eq_supp3_regRec}
also vanishes.

\textbf{Step 2} (continuity in the base point)\textbf{.}
Proposition~\ref{prop_recentre_Pi-}
implies that 
\begin{align}
	\Pi^-_{x \beta}-\Pi^-_{y \beta} 
&=\big((\Gamma_{xy}-\mathrm{id})\Pi^-_y \big)_{ \beta}
- 
\Big(
\sum_{ k \geqslant \kk} \ve^{(\kk-k)\alpha} T_x^{<|\cdot|-2} \big(z_k W_{k, \varepsilon} (\Pi_x)\big)\Big)_{ \beta}
 \notag \\
&\hphantom{=(\Gamma_{xy}-\mathrm{id})\Pi^-_y\ }
+ \Big(\Gamma_{xy} 
\sum_{ k \geqslant \kk} \ve^{(\kk-k)\alpha} T_y^{<|\cdot|-2}
\big(z_k W_{k, \varepsilon}(\Pi_y)\big)\Big)_{ \beta}
 \label{tmp05} \\
&= \big((\Gamma_{xy}-\mathrm{id})\Pi^-_y \big)_{\beta}
- 
\Big(\sum_{ k \geqslant \kk} \ve^{(\kk-k)\alpha} (T_x^{<|\cdot|-2}-T_y^{<|\cdot|-2}) \big(z_k W_{k,
\varepsilon}(\Pi_x)\big)\Big)_{ \beta}
 \notag \\
&\hphantom{=(\Gamma_{xy}-\mathrm{id})\Pi^-_y\ }
- \Big(\Gamma_{xy} 
\sum_{ k \geqslant \kk} \ve^{(\kk-k)\alpha} (T_y^{<|\beta|-2}-T_y^{<|\cdot|-2})
\big(z_k W_{k, \varepsilon}(\Pi_y)\big)\Big)_{ \beta}
 \label{tmp06}\,.
\end{align}
We proceed by estimating each term separately, starting with
$(\Gamma_{xy}-\mathrm{id})\Pi^-_y$ in Step~2a below.
Thereafter, we distinguish two regimes $\ve\leq\sqrt[4]{t}+|x-y|$ and $\ve>\sqrt[4]{t}+|x-y|$:
in the first case, we estimate \eqref{tmp05} (Step~2b below), while in the latter case we
estimate \eqref{tmp06} (Step~2c below). 

\textbf{Step 2a.} 
Using H\"older's inequality, we deduce
\begin{equation}\label{eq_supp4_regRec}
\EE^\frac{1}{p}|\big((\Gamma_{xy}-\mathrm{id})\Pi_y^-\big)_{\beta t}(y) |^p 
\leq \sum_\gamma 
\EE^\frac{1}{q}|(\Gamma_{xy}-\mathrm{id})_\beta^\gamma|^q \, 
\EE^\frac{1}{r}|\Pi^-_{y\gamma t}(y) |^r\,,
\end{equation}
for $1/p=1/q+1/r$. 
Now, the strict triangularity of $\Gamma_{xy}-\mathrm{id}$ with respect to~$\prec$ (Lemma~\ref{lem_def_gamma}~(3)) implies
that the sum over $\gamma$ is effectively finite
(Lemma~\ref{lem_ind_legal}).
Hence, we may apply \eqref{est_Pi-} for every $ \gamma \prec \beta$ in the sum. 
Since $\Pi^-\in\widetilde T$ by definition, we have that 
$[\gamma]\geq0$ which allows us to use \eqref{est_Gamma}.
Altogether, the right-hand side of \eqref{eq_supp4_regRec} is bounded (up to a constant) by 
\begin{equation}\label{eq_supp5_regRec}
\sum_{\gamma\prec\beta}
\ve^{|\beta|-\langle\beta\rangle-|\gamma|+\langle\gamma\rangle} 
|x-y|^{\langle\beta\rangle-\langle\gamma\rangle} 
\ve^{|\gamma|-\langle\gamma\rangle} (\sqrt[4]{t})^{\langle\gamma\rangle-2} 
\lesssim 
\ve^{|\beta|-\langle\beta\rangle} (\sqrt[4]{t})^{\alpha-2}
(\sqrt[4]{t}+|x-y|)^{\langle\beta\rangle-\alpha}\, ,
\end{equation}
where we used that the sum is actually also restricted to
$\langle\beta\rangle-\langle\gamma\rangle\geq0$ (Lemma~\ref{lem_def_gamma}~(2)), 
and that $\langle\cdot\rangle-\alpha\geq0$ (Lemma~\ref{lem_hom_props}).
This yields the desired bound. 

\textbf{Step 2b.} 
We assume $\ve\leq\sqrt[4]{t}+|x-y|$ and estimate the second and third contributions of \eqref{tmp05}. 
For the second contribution we use the bound
\begin{equation}\label{tmp07}
	\EE^\frac1p|\ve^{(\kk-k)\alpha} \partial^\n\big(z_k W_{k, \varepsilon}(\Pi_x)\big)_\beta(x) \big((\cdot-x)^\n\big)_t(y)|^p
\lesssim \ve^{|\beta|-2-|\n|}
(\sqrt[4]{t}+|x-y|)^{|\n|} 
\end{equation}
for $|\n|<|\beta|-2$, 
which in turn is a direct consequence of \eqref{est_Pi-_bounded_preliminary} and the semigroup $\Psi_t$. 
Since $\langle\beta\rangle-2-|\n|\geq0$, 
the right-hand side is bounded by $\ve^{|\beta|-\langle\beta\rangle}(\sqrt[4]{t}+|x-y|)^{\langle\beta\rangle-2}$, 
which in turn can again be estimated by 
$\ve^{|\beta|-\langle\beta\rangle}(\sqrt[4]{t})^{\alpha-2}(\sqrt[4]{t}+|x-y|)^{\langle\beta\rangle-\alpha}$
(using that $\langle\cdot\rangle-\alpha\geq0$). 

To estimate the third contribution of \eqref{tmp05} we proceed similarly. 
Using H\"older's inequality in probability, together with $ \eqref{est_Gamma}$
and \eqref{tmp07}, yields
\begin{equation}\label{tmp08}
	\begin{aligned}
&\EE^\frac1p| (\Gamma_{xy})_\beta^\gamma \ve^{(\kk-k)\alpha} \partial^\n\big(z_k W_{k,
\varepsilon}(\Pi_y)\big)_\gamma(y) \big((\cdot-y)^\n\big)_t(y)|^p \\
&\lesssim 
\ve^{|\beta|-\langle\beta\rangle-|\gamma|+\langle\gamma\rangle}
|x-y|^{\langle\beta\rangle-\langle\gamma\rangle} \ve^{|\gamma|-2-|\n|}
(\sqrt[4]{t})^{|\n|} 
\leqslant
\ve^{|\beta|-\langle\beta\rangle} (\sqrt[4]{t}+|x-y|)^{\langle\beta\rangle-2}\, .
	\end{aligned}
\end{equation}
In the second inequality we used that $\langle\beta\rangle-\langle\gamma\rangle\geq0$ 
(Lemma~\ref{lem_def_gamma}~(2)) and that $|\n|\leq\langle\gamma\rangle-2$ (which follows
from the restriction $|\n|<|\gamma|-2$ in \eqref{tmp05}), together with the fact that we
restricted ourselves to
$\ve\leq\sqrt[4]{t}+|x-y|$.
Finally, the right-hand side of \eqref{tmp08} can be trivially bounded by
$\ve^{|\beta|-\langle\beta\rangle}
(\sqrt[4]{t})^{\alpha-2}(\sqrt[4]{t}+|x-y|)^{\langle\beta\rangle-\alpha}$.

\textbf{Step 2c.}
We turn to the case $\ve>\sqrt[4]{t}+|x-y|$ and estimate the second and third contributions of \eqref{tmp06}.
For the third contribution we use the first inequality of \eqref{tmp08}, 
which we now only need for $\n$ with $|\gamma|-2\leq|\n|<|\beta|-2$. 
Then, using $\ve>\sqrt[4]{t}+|x-y|$, we have  
\begin{equation}
	\begin{aligned}
		\EE^\frac1p| (\Gamma_{xy})_\beta^\gamma \ve^{(\kk-k)\alpha} \partial^\n\big(z_k W_{k,
\varepsilon}(\Pi_y)\big)_\gamma(y) \big((\cdot-y)^\n\big)_t(y)|^p
&\lesssim 
\ve^{|\beta|-\langle\beta\rangle}
(\sqrt[4]{t}+|x-y|)^{\langle\beta\rangle-2}\,,
	\end{aligned}
\end{equation}
where now we used that 
$\ve^{|\gamma|-2-|\n|}
\leq \ve^{|\gamma|-\langle\gamma\rangle} 
(\sqrt[4]{t}+|x-y|)^{\langle\gamma\rangle-2-|\n|}$. 

For the second contribution of \eqref{tmp06} we combine 
$T_x^{<|\beta|-2}-T_y^{<|\beta|-2}
=T_x^{<|\beta|-2}(\mathrm{id}-T_y^{<|\beta|-2})$ and
$\partial^\n T_y^{<|\beta|-2} = T_y^{<|\beta|-2-|\n|}\partial^\n$ 
to deduce
\begin{equation}
(T_x^{<|\beta|-2}-T_y^{<|\beta|-2})f
= \sum_{|\n|<|\beta|-2}\tfrac{1}{\n!} (\mathrm{id}-T_y^{<|\beta|-2-|\n|}) \partial^\n f(x) (\cdot-x)^\n \, .
\end{equation}
Thus, the claim follows as a direct consequence of
\eqref{est_Pi-_taylorremainder_preliminary}: 
\begin{align}
&\EE^\frac1p|\ve^{(\kk-k)\alpha} (\mathrm{id}-T_y^{<|\beta|-2-|\n|}) \partial^\n \big(z_k
W_{k, \varepsilon}(\Pi_x)\big)_\beta(x) \big((\cdot-x)^\n\big)_t(y) |^p \\
&\lesssim \ve^{\alpha-2-|\n|-\lceil|\beta|-2-|\n|\rceil} |x-y|^{\lceil|\beta|-2-|\n|\rceil} (\ve+|x-y|)^{|\beta|-\alpha} (\sqrt[4]{t}+|x-y|)^{|\n|} \, .
\end{align}
Since $|x-y|\leq\sqrt[4]{t}+|x-y|\leq\ve$ and $|\beta|-\alpha\geq0$ this right-hand side is estimated by 
\begin{equation}
\ve^{|\beta|-\langle\beta\rangle}
\ve^{\langle\beta\rangle-2-|\n|-\lceil|\beta|-2-|\n|\rceil} 
|x-y|^{\lceil|\beta|-2-|\n|\rceil} 
(\sqrt[4]{t}+|x-y|)^{|\n|} \, .
\end{equation}
Since 
$\langle\beta\rangle-2-|\n|-\lceil|\beta|-2-|\n|\rceil\leq0$, this is further bounded by 
$\ve^{|\beta|-\langle\beta\rangle}(\sqrt[4]{t}+|x-y|)^{\langle\beta\rangle-2}$
and thus again by 
$\ve^{|\beta|-\langle\beta\rangle}
(\sqrt[4]{t})^{\alpha-2}(\sqrt[4]{t}+|x-y|)^{\langle\beta\rangle-\alpha}$.
\end{proof}

\begin{remark}[Upgrading to bounds with $ \partial^{\n}$]\label{rem_improve_n}
	We notice that the bound $\eqref{est_Pi-}_{ \beta}$ immediately implies 
	\begin{equation}\label{est_Pi-n}
		\begin{aligned}
			\EE^\frac{1}{p} \left|
\partial^{\n}
\Pi_{x\beta t}^-(y) \right|^p
\lesssim 
\ve^{|\beta|-\langle\beta\rangle} 
\big(\sqrt[4]{t}\big)^{\alpha-2 - |\n|} (\sqrt[4]{t}+|x-y|)^{\langle\beta\rangle-\alpha} \, ,
		\end{aligned}
	\end{equation}
	for any $ \n \in \NN^{1+d}$.
	To see this, we write 
\begin{equation}
	\begin{aligned}
		\EE^{\frac{1}{p}}\big| \partial^{\n} \Pi_{x\beta t}^-(y)
		\big|^{p}
		& \leqslant 
		\int dz \, \big|\partial^{\n}  \Psi_{t/2}(y-z)\big| \EE^{\frac{1}{p}}\big|  \Pi_{x\beta
		\frac{t}{2} }^-(z) \big|^{p} \\
		& \lesssim 
		\ve^{|\beta|-\langle\beta\rangle} \big(\sqrt[4]{t}\big)^{\alpha-2}
		\int dz\, \big| \partial^{\n}  \Psi_{t/2}(y-z)\big|
		(\sqrt[4]{t}+|x-z|)^{\langle\beta\rangle-\alpha} \,,
	\end{aligned}
\end{equation}
where we used \eqref{est_Pi-}. Then, using the moment bound
\eqref{eq_momentbound_semigroup} with $ \varepsilon = 0 $, we conclude
that 
\begin{equation}
	\begin{aligned}
			\EE^{\frac{1}{p}}\big| \partial^{\n} \Pi_{x\beta t}^-(y)
		\big|^{p}
\lesssim \ve^{|\beta|-\langle\beta\rangle} \big(\sqrt[4]{t}\big)^{\alpha-2- | \n|}
 (\sqrt[4]{t}+|x-y|)^{\langle\beta\rangle-\alpha}\,.
	\end{aligned}
\end{equation}
	The above argument applies more generally. 
	Therefore, throughout the paper, we state only the bounds for $ \n = \0 $ and
	use without further proof their corresponding extensions to general $ \n $
	whenever this argument applies.
\end{remark}

Finally, we can construct and estimate the model component $ \Pi_{\beta}$ using an
integration argument.

\begin{proposition}[Integration, Step 4.a]\label{prop_integration}
Let $[\beta]\geq0$ and assume that \eqref{est_Pi-} holds for $\Pi^-_{x\beta}$. Then
\begin{equation}\label{eq_supp1_regInt}
\Pi_{x\beta} \coloneqq \int_0^\infty ds\, (\mathrm{id}-T_x^{<|\beta|}) ( -\partial_{x_{0}}- \Delta) \Pi^-_{x\beta s}
\end{equation}
is the unique solution of $(\partial_{x_{0}}-\Delta)\Pi_{x\beta}=\Pi^-_{x\beta}$ satisfying 
\begin{equation}\label{est_Pi}
\EE^\frac1p\left| \Pi_{x\beta t}(y) \right|^p
\lesssim 
\ve^{|\beta|-\langle\beta\rangle}
(\sqrt[4]{t})^\alpha (\sqrt[4]{t}+|x-y|)^{\langle\beta\rangle-\alpha} \,.
\end{equation}
\end{proposition}

Let us recall once more that we implicitly understand the statement of
Proposition~\ref{prop_integration} only for suitable populated multiindices that satisfy  $ [ \beta] \geqslant 0 $. 
In general, we have $(\partial_{x_{0}}-\Delta)\Pi_{x \delta_{\n}}\neq \Pi^-_{x
\delta_{\n}}=0$ for purely polynomial indices, and 
$ 0 = (\partial_{x_{0}}-\Delta)\Pi_{x \beta} \neq \Pi_{x \beta}^{-}$ for multiindices 
$ \beta = \delta_{\kk} + \delta_{\n_{1}}+ \cdots \delta_{\n_{\kk}}$,
see \eqref{model_poly}.
\\

The proof of Proposition~\ref{prop_integration} is almost verbatim to
\cite[Section~3.7]{BOT}. 
We only note that $T_x^{<|\beta|}$ in \eqref{eq_supp1_regInt} can be replaced by
$T_x^{<\langle\beta\rangle}$, since 
our choice of $\langle\cdot\rangle$ implies that 
$|\n|<|\beta|$ if and only if $|\n|<\langle\beta\rangle$. 
Hence, the order of the Taylor remainder is compatible with the order of growth/vanishing
of the input \eqref{est_Pi-} and the output \eqref{est_Pi},
and furthermore $\langle\beta\rangle$ is not an integer for $[\beta]\geq0$.  

\medskip

It is only left to reestablish the necessary symmetries and continuity properties of $
\Pi_{x \beta}$, so it can be used in the next induction step. Moreover, we have to define
$ \Gamma_{\beta}^{\gamma}$ appropriately for all $ \gamma$ that are purely polynomial,
which we do in a three-point argument below. 

\begin{proposition}[Covariances of $\Pi$, Step 4.b]\label{prop_covariance_Pi}
Let $[\beta]\geq0$ and assume that the covariances of Proposition~\ref{prop_covariance_Pi-} hold for
$\Pi^-_{x\beta}$, i.e.~we have \eqref{Pi-_shift}, \eqref{Pi-_parity}, and \eqref{Pi-_reflection}.
Then the covariances \eqref{Pi_shift}, \eqref{Pi_parity}, and \eqref{Pi_reflection} hold for $\Pi_{x\beta}$.
\end{proposition}

\begin{proof}
This is an immediate consequence of the integral representation of $\Pi_{x\beta}$ in
\eqref{eq_supp1_regInt}.
\end{proof}
 
\begin{proposition}[Three-point argument, Steps 5.a--c)]\label{prop_3pt}
Let $[\beta]\geq0$ and assume 
that \eqref{est_Pi} holds for $\Pi_{x\beta'}$ for all populated $\beta'\preceq\beta$, 
that $(\Gamma_{xy}\Pi^-_y-\Pi^-_x)_\beta$ is a (random) polynomial of degree $<|\beta|-2$, 
and that \eqref{est_Gamma} holds for $(\Gamma_{xy})_\beta^\gamma$ for all $\gamma$ populated and not purely polynomial.
Then
\begin{enumerate}
\item there exist $\pi^{(\n)}_{xy\beta}$ for $|\n|<|\beta|$ such that 
\begin{equation}\label{eq_recentre_Pi}
(\Gamma_{xy}\Pi_y)_\beta = \Pi_{x\beta} \, ,
\end{equation}
\item $(\Gamma_{xy}\Gamma_{yz})_\beta^\gamma = (\Gamma_{xz})_\beta^\gamma$ 
and $(\Gamma_{xx})_\beta^\gamma = (\mathrm{id})_\beta^\gamma$ for all $\gamma$, 
provided both hold also for all populated $\beta'\prec\beta$, 
and provided $(\Gamma_{xy}\Pi_y)_{\beta'}=\Pi_{x\beta'}$ holds for all populated $\beta'\prec\beta$,
\item \eqref{est_pin} holds for $\pi^{(\n)}_{xy\beta}$ and thus 
\eqref{est_Gamma} holds for $(\Gamma_{xy})_\beta^\gamma$ for all $\gamma$. 
\end{enumerate}
\end{proposition}

\begin{proof}
By the choice of $\langle\cdot\rangle$, 
the assumption yields that $(\Pi^-_x-\Gamma_{xy}\Pi^-_y)_\beta$ is a random polynomial of degree $<\langle\beta\rangle-2$.
Since also $\langle\beta\rangle\not\in\NN$ for $[\beta]\geq0$, 
we can follow the argument of \cite[Proposition~5.3]{LOTT24}, 
just replacing the homogeneity by $\langle\cdot\rangle$, 
to deduce that $(\Pi_x-\Gamma_{xy}P_{\widetilde{T}}\Pi_y)_\beta$ is a random polynomial of degree $<\langle\beta\rangle\leq|\beta|$.
Here $P_{\widetilde{T}}$ denotes the canonical projection from $\RR[\![z_k,z_\n]\!]$ to $\widetilde{T}$ defined just before Lemma~\ref{lem_sets}. 
One can therefore define $\pi^{(\n)}_{xy\beta}$ by 
\begin{equation}\label{eq_supp1_regstep5}
\sum_{|\n|<|\beta|} \pi^{(\n)}_{xy\beta} (\cdot-y)^\n 
\coloneqq (\Pi_x - \Gamma_{xy}P_{\widetilde{T}}\Pi_y)_\beta \, .
\end{equation}
Now, because $\Pi_{y\delta_\n}=(\cdot-y)^\n$ and 
$(\Gamma_{xy})_\beta^{\delta_\n}=\pi^{(\n)}_{xy\beta}$ for $\beta$ not purely polynomial
(see \eqref{e_def_gamma}), we can write the left-hand side of \eqref{eq_supp1_regstep5} in terms of 
$(\Gamma_{xy}(\mathrm{id}-P_{\widetilde{T}})\Pi_y)_\beta$, and the desired recentering of $\Pi$ follows.
We emphasize that the choice of $\pi^{(\n)}_{xy}$ via \eqref{eq_supp1_regstep5} satisfies \eqref{eq_pop_pi}. 

The second and third items of the proposition follow along the lines of \cite[Propostion~5.4]{LOTT24} 
and \cite[Proposition~4.4]{LOTT24}, respectively. 
\end{proof}

\begin{proposition}[Boundedness and continuity of $\Pi$, Step 6]\label{prop_qualitative_pi}
Let $[\beta]\geq0$ and assume that \eqref{est_Pi-_bounded} 
and \eqref{est_Pi-}
hold for $\Pi^-_{x\beta}$. 
Then \eqref{est_Pi_bounded} holds for $\Pi_{x\beta}$ as well as 
\begin{equation}\label{est_Pi_taylorremainder}
\begin{aligned}
\EE^{\frac{1}{p}} \left|
(\mathrm{id}-T_z^{\leq l})\partial^\n\Pi_{x \beta}(y) 
\right|^{ p} 
\lesssim
\ve^{\alpha-|\n|-1-l} |y-z|^{1+l}
( \ve + |x-y| +|x-z|)^{| \beta | - \alpha} \, ,
\end{aligned}
\end{equation}
for all $\n$ and $l\geq0$.
\end{proposition}

\begin{proof}
Again, \eqref{est_Pi_taylorremainder} is a simple consequence of \eqref{est_Pi_bounded}
and Lemma~\ref{lem_interpolation}, so it is only left to prove \eqref{est_Pi_bounded}.
To this end, we use the representation \eqref{eq_supp1_regInt} (which is a consequence of
Proposition~\ref{prop_integration}) in the form of 
\begin{equation}\label{eq_supp1_regStep6}
\partial^\n \Pi_{x\beta}(y) 
= \int_0^\infty dt \, (\mathrm{id}-T_x^{<|\beta|-|\n|})\partial^\n ( -\partial_{x_{0}}- \Delta) \Pi^-_{x\beta t}(y) \, .
\end{equation}
Note that it is part of the following proof that this expression is indeed well-defined. 
In the following three steps we estimate the contributions of the integral over the
regions $0\leq t\leq\ve^4$, $\ve^4\leq t\leq(\ve+|x-y|)^4$, and $ (\ve+|x-y|)^4\leq t$, respectively. 

\textbf{Step 1 \textnormal{($0\leq t\leq\ve^4$)}.}
First, we notice that
\begin{align}
\EE^\frac1p| \partial^\n ( -\partial_{x_{0}}- \Delta)  \Pi^-_{x\beta t}(y) |^p
&= \EE^\frac1p\Big| \int dz\, ( -\partial_{x_{0}}- \Delta) \Psi_t(y-z)\partial^\n\Pi^-_{x\beta}(z)\Big|^p \\
&\lesssim \int dz\, |( -\partial_{x_{0}}- \Delta) \Psi_t(y-z)| \ve^{\alpha-2-|\n|}(\ve+|x-z|)^{|\beta|-\alpha} \\
&\lesssim \ve^{\alpha-2-|\n|} (\ve+\sqrt[4]{t}+|x-y|)^{|\beta|-\alpha} (\sqrt[4]{t})^{-2} \, ,
\end{align}
where in the first inequality we used \eqref{est_Pi-_bounded} and in the second
inequality we used the moment bound \eqref{eq_momentbound_semigroup}. 
Thus, estimating the Taylor remainder in \eqref{eq_supp1_regStep6}  by the triangle inequality yields 
\begin{equation}\label{eq_supp2_regStep6}
\begin{aligned}
&\EE^\frac1p\Big|
\int_0^{\ve^4} dt \, (\mathrm{id}-T_x^{<|\beta|-|\n|})\partial^\n ( -\partial_{x_{0}}- \Delta) \Pi^-_{x\beta t}(y) \Big|^p \\
&\lesssim 
\int_0^{\ve^4} dt\, \ve^{\alpha-2-|\n|}(\ve+\sqrt[4]{t}+|x-y|)^{|\beta|-\alpha} (\sqrt[4]{t})^{-2} \\
&\,+\int_0^{\ve^4} dt\, \sum_{|\m|<|\beta|-|\n|} \ve^{\alpha-2-|\n|-|\m|}(\ve+\sqrt[4]{t})^{|\beta|-\alpha} (\sqrt[4]{t})^{-2} |x-y|^{|\m|} \, .
\end{aligned}
\end{equation}
Now, using the restriction $\sqrt[4]{t}\leq\ve$, which in particular implies
$\int_0^{\ve^4}dt\,(\sqrt[4]{t})^{-2}\lesssim\ve^2$, the right-hand side of
\eqref{eq_supp2_regStep6} is (up to a constant) bounded by 
\begin{equation}
\ve^{\alpha-|\n|} (\ve+|x-y|)^{|\beta|-\alpha} 
+ \sum_{|\m|<|\beta|-|\n|} \ve^{|\beta|-|\n|-|\m|} |x-y|^{|\m|} 
\lesssim \ve^{\alpha-|\n|} (\ve+|x-y|)^{|\beta|-\alpha} \, ,
\end{equation}
which is the estimate we desired. 

\textbf{Step 2 \textnormal{($\ve^4\leq t\leq(\ve+|x-y|)^4$)}.}
We still break up the Taylor remainder in \eqref{eq_supp1_regStep6}, but use
\eqref{est_Pi-n} to see
\begin{align}
&\EE^\frac1p\Big|
\int_{\ve^4}^{(\ve+|x-y|)^4} dt \, (\mathrm{id}-T_x^{<|\beta|-|\n|})\partial^\n ( -\partial_{x_{0}}- \Delta) \Pi^-_{x\beta t}(y) \Big|^p \\
&\lesssim 
\int_{\ve^4}^{(\ve+|x-y|)^4} dt\, 
\ve^{|\beta|-\langle\beta\rangle}
(\sqrt[4]{t})^{\alpha-4-|\n|}(\sqrt[4]{t}+|x-y|)^{\langle\beta\rangle-\alpha} \\
&\,+\int_{\ve^4}^{(\ve+|x-y|)^4} dt\, \sum_{|\m|<|\beta|-|\n|} 
\ve^{|\beta|-\langle\beta\rangle}
(\sqrt[4]{t})^{\alpha-4-|\n|-|\m|+\langle\beta\rangle-\alpha} |x-y|^{|\m|} \, .
\end{align}
Using $\sqrt[4]{t}\leq\ve+|x-y|$ and $\int_{\ve^4}^{(\ve+|x-y|)^4} dt\,(\sqrt[4]{t})^{\alpha-4-|\n|}\lesssim (\ve+|x-y|)^{\alpha-|\n|}+\ve^{\alpha-|\n|} $ this is (up to a constant) estimated by 
\begin{equation}
\ve^{|\beta|-\langle\beta\rangle}
\big((\ve+|x-y|)^{\alpha-|\n|}+\ve^{\alpha-|\n|}\big) (\ve+|x-y|)^{\langle\beta\rangle-\alpha} 
\lesssim \ve^{\alpha-|\n|} (\ve+|x-y|)^{|\beta|-\alpha} \, ,
\end{equation}
as desired. 

\textbf{Step 3 \textnormal{($(\ve+|x-y|)^4\leq t$)}.}
We use the Taylor remainder as in the proof of Lemma~\ref{lem_interpolation} and appeal
again to \eqref{est_Pi-n} to deduce
\begin{equation}\label{eq_supp01_regStep6}
\begin{aligned}
&\EE^\frac1p\Big|
\int_{(\ve+|x-y|)^4}^\infty dt \, (\mathrm{id}-T_x^{<|\beta|-|\n|})\partial^\n ( -\partial_{x_{0}}- \Delta) \Pi^-_{x\beta t}(y) \Big|^p \\
&\lesssim \int_{(\ve+|x-y|)^4}^\infty dt \sum_{\substack{|\m|\geq|\beta|-|\n|\colon\\ m_0+\dots+m_d\leq |\beta|-|\n|+1}} 
\ve^{|\beta|-\langle\beta\rangle}
(\sqrt[4]{t})^{\alpha-4-|\n|-|\m|} (\sqrt[4]{t}+|x-y|)^{\langle\beta\rangle-\alpha} |x-y|^{|\m|} \, .
\end{aligned}
\end{equation}
Using $|x-y|\leq\sqrt[4]{t}$ and 
$\int_{(\ve+|x-y|)^4}^\infty dt \, (\sqrt[4]{t})^{\langle\beta\rangle-4-|\n|-|\m|} 
\lesssim (\ve+|x-y|)^{\langle\beta\rangle-|\n|-|\m|}$ since $\langle\beta\rangle-|\n|-|\m|<0$, \eqref{eq_supp01_regStep6} is bounded by (up to a constant) 
\begin{equation}
\sum_{|\m|\geq|\beta|-|\n|\colon m_0+\dots+m_d\leq |\beta|-|\n|+1} 
\ve^{|\beta|-\langle\beta\rangle} (\ve+|x-y|)^{\langle\beta\rangle-|\n|-|\m|} |x-y|^{|\m|}
\lesssim (\ve+|x-y|)^{|\beta|-|\n|} \, ,
\end{equation}
which is again the desired estimate. 
\end{proof}

\section{Intermediate indices}\label{sec:intermediate}

In this section we consider multiindices $\beta$ with homogeneity $|\beta| = 2$, which
we call \emph{intermediate multiindices}. 
Almost all arguments from Section~\ref{sec:regular} for regular multiindices carry over directly to this setting. 
The only exception is the reconstruction argument in Proposition~\ref{prop:regrec}, 
which relies essentially on the assumption that $|\beta| > 2$.
We therefore require a separate argument for the estimate of $\Pi^-_{x\beta}$. 

First, we notice the following classification of multiindices with homogeneity $2$. 

\begin{lemma}\label{lem:classification}
If $\beta$ is populated and $|\beta|=2$, then
\begin{enumerate}
	\item either $\beta$ is purely polynomial, 
i.e.~$\beta=\delta_\n$ for some $|\n|=2$,
	\item or $\beta=\delta_k+\kk\delta_{\0}$ for some $k\geq\kk$.		
\end{enumerate}
\end{lemma}

\begin{proof}
Since $ \beta$ is populated, we need to consider three cases. 
First, if $\beta=\delta_\n$ then the claim is immediate because $|\beta|=|\n|$. 
Likewise, if $\beta=\delta_{\kk}+\delta_{\n_1}+\dots+\delta_{\n_{\kk}}$ then
$|\beta|=2+|\n_1|+\dots+|\n_{\kk}|$, which implies that $ \n_{i} = \0 $ necessarily.

Next, we turn to $[\beta]\geq0$.
Since $\alpha$ is irrational, the homogeneity defined in \eqref{e_def_homo} can only be an integer provided $1+(\kk-1)\sum_k\beta(k)-\sum_\n\beta(\n)=0$, 
in particular $\sum_\n\beta(\n)\geq1$.
On the other hand, $0\leq[\beta]=\sum_{k \geqslant \kk} (k-1)\beta(k)-\sum_\n\beta(\n)$ then implies $\sum_k\beta(k)\geq1$. 
In turn, this implies (since $|\beta|=2$) that $\sum_k\beta(k)=1$ and $\sum_\n|\n|\beta(\n)=0$. 
Putting this back into $1+(\kk-1)\sum_k\beta(k)-\sum_\n\beta(\n)=0$ yields $\beta=\delta_k+\kk\delta_\0$, 
and by $[\beta]\geq0$ we must have $k > \kk$.
\end{proof}

Lemma~\ref{lem:classification} greatly simplifies our task in extending the
reconstruction argument, because $ \Pi^{-}_{ \delta_{\n}} =0$ by definition. Therefore,
we only need to provide a reconstruction argument for indices $ \beta = \delta_{k}  + \kk
\delta_{\bf 0}$:

\begin{lemma}[Intermediate reconstruction]
Assume that $[\beta]\geq0$ and $|\beta|=2$, and that $\eqref{est_Pi-}_{\beta'}$ and $\eqref{est_Pi-_bounded}_{\beta'}$ hold for all populated $\beta'\prec\beta$. 
Then $\eqref{est_Pi-}_\beta$ holds.
\end{lemma}

\begin{proof}
	By Lemma~\ref{lem:classification} 
	it suffices to only consider $\beta=\delta_k+\kk\delta_\0$ for some $k\geq\kk$,
	since for purely polynomial components $ \Pi^{-}_{\delta_{\n}} = 0 $ by definition. 
Thus, by the definition of  $\Pi^-$ via the hierarchy \eqref{hierarchy}, we have 
\begin{equation}
\Pi^-_{x\beta} 
= \ve^{(\kk-k)\alpha} \tbinom{k}{\kk} W_{k-\kk, \varepsilon}(\Pi_{x0}) \, .
\end{equation}
The case $k=\kk$ is clear, we thus focus on odd $k>\kk$ and write
\begin{equation}\label{eq_supp10_intermed}
\Pi^-_{x\beta} 
= \ve^{-2\alpha} \ve^{(\kk-(k-2))\alpha} \tbinom{k}{\kk} W_{k-2-(\kk-2), \varepsilon}(\Pi_{x0}) 
= \ve^{-2\alpha} \tbinom{k}{\kk}/\tbinom{k-2}{\kk-2} \Pi^-_{x \delta_{k-2}+(\kk-2)\delta_\0} \, .
\end{equation}
Now, because
\begin{equation}
|\delta_{k-2}+(\kk-2)\delta_\0|_\prec
= 2+2\alpha+\tfrac{D}{2}(k-\kk)
< 2+\tfrac{D}{2}(k-\kk)
= |\delta_k+\kk\delta_\0|_\prec \, ,
\end{equation}
we may appeal to the assumption in the form of
$\eqref{est_Pi-}_{\delta_{k-2}+(\kk-2)\delta_\0}$ and
$\eqref{est_Pi-_bounded}_{\delta_{k-2}+(\kk-2)\delta_\0}$ in order to
 estimate \eqref{eq_supp10_intermed}. Once more, we split the estimate into two cases. 

Whenever $\ve\leq \sqrt[4]{t}+|x-y|$, we use
$\eqref{est_Pi-}_{\delta_{k-2}+(\kk-2)\delta_\0}$  and obtain 
\begin{align}
\EE^\frac{1}{p}|\Pi^-_{x\beta t}(y)|^p 
&\lesssim \ve^{-2\alpha} \ve^{|\delta_{k-2}+(\kk-2)\delta_\0|-\langle\delta_{k-2}+(\kk-2)\delta_\0\rangle}
(\sqrt[4]{t})^{\alpha-2}(\sqrt[4]{t}+|x-y|)^{\langle\delta_{k-2}+(\kk-2)\delta_\0\rangle-\alpha} \\
&\leq \ve^{\kappa} (\sqrt[4]{t})^{\alpha-2}(\sqrt[4]{t}+|x-y|)^{|\delta_{k-2}+(\kk-2)\delta_\0|-\alpha-2\alpha-\kappa} \, ,
\end{align}
where in the second inequality we used additionally that $|\cdot|-\langle\cdot\rangle\geq0$ and that $\kappa$ is chosen small enough so that $\kappa<-2\alpha$. 
This is the desired estimate since $\kappa=|\beta|-\langle\beta\rangle$ and $|\delta_{k-2}+(\kk-2)\delta_\0|-2\alpha=|\beta|$.

For the case $\ve\geq\sqrt[4]{t}+|x-y|$ we appeal to
$\eqref{est_Pi-_bounded}_{\delta_{k-2}+(\kk-2)\delta_\0}$ which yields 
\begin{equation}
	\begin{aligned}
\EE^\frac{1}{p}|\Pi^-_{x\beta t}(y)|^p 
& \lesssim 
\int dz\, |\Psi_{t}(y-z)| 
\ve^{-2\alpha} \ve^{\alpha-2}
( \ve+ | x-z|)^{|\delta_{k-2}+(\kk-2)\delta_\0| - \alpha}
\\
&\lesssim \ve^{-2\alpha} \ve^{\alpha-2}
(\ve+\sqrt[4]{t}+|x-y|)^{|\delta_{k-2}+(\kk-2)\delta_\0|-\alpha} 
\lesssim \ve^{|\delta_{k-2}+(\kk-2)\delta_\0|-2-2\alpha} 
=1 \, ,
	\end{aligned}
\end{equation}
where in the second inequality we used \eqref{eq_momentbound_semigroup}, and in the third inequality we used additionally that $|\cdot|-\alpha\geq0$. 
This turns into the desired estimate by noting that
\begin{equation}
1
\leq\ve^{|\beta|-\langle\beta\rangle} (\sqrt[4]{t}+|x-y|)^{\langle\beta\rangle-|\beta|}
\leq\ve^{|\beta|-\langle\beta\rangle} (\sqrt[4]{t})^{\alpha-2} (\sqrt[4]{t}+|x-y|)^{\langle\beta\rangle-\alpha} \, ,
\end{equation}
where in the first inequality we used again that $|\cdot|-\langle\cdot\rangle\geq0$, 
and in the second inequality we used that $|\beta|=2$.
\end{proof}

\section{Singular indices}\label{sec:singular}

We now repeat the strategy of Section~\ref{sec:regular} for singular multiindices $\beta$, i.e.~$|\beta|<2$, satisfying $[\beta]\geq0$. 
As in Section~\ref{sec:intermediate},
we must replace the reconstruction argument of Proposition~\ref{prop:regrec}.
In the singular setting, this step is substantially more involved and requires the
construction and estimation of auxiliary quantities and the use of Malliavin calculus.

The key ingredient of this section is the spectral gap inequality \eqref{eq_sg_L2}.
We apply it in the form of 
\begin{equation}\label{eq:sg}
\EE^\frac{1}{p}|F|^p
\lesssim_p |\EE F|
+\EE^\frac{1}{p}\Big\|\frac{\partial F}{\partial\zeta}\Big\|^p_{\dot{H}^{-s}}
=|\EE F|
+ \sup_{\delta\zeta\neq0} \frac{\EE\delta F}{\EE^\frac{1}{p^*}\|\delta\zeta\|^{p^*}_{\dot{H}^s}} 
\quad\textnormal{for}\quad
p\geq2
\quad\textnormal{and}\quad
\tfrac{1}{p}+\tfrac{1}{p^*}=1 \, .
\end{equation}
For $p>2$ the above inequality follows 
by inserting (a suitable approximation of) $|F|^{p/2}$ into \eqref{eq_sg_L2}, see \cite[Proposition~5.1]{IORT23} for details, 
and the equality follows by duality. 
Here $\delta F$ is a shorthand for 
\begin{equation}\label{gateaux}
\delta F(\zeta;\delta\zeta)
= \Big( \frac{\partial F}{\partial\zeta}[\zeta],\delta\zeta\Big)
= \int_{\RR^{1+d}} dx\, \frac{\partial F}{\partial\zeta}[\zeta](x)\delta\zeta(x) \, .
\end{equation}
Throughout this section we choose $2\leq p<\infty$ such that $1<p^*\leq2$, 
and to shorten the notation we introduce 
\begin{equation}\label{eq:weight}
\w\coloneqq\EE^\frac{1}{p^*}\|\delta\zeta\|_{\dot{H}^s}^{p^*}
\quad\textnormal{for}\quad\tfrac{1}{p}+\tfrac{1}{p^*}=1 \, .
\end{equation}
Therefore, in view of \eqref{eq:sg}, we always estimate quantities of the form 
\begin{equation}
\EE^\frac{1}{q}|\delta F|^q 
\quad\textnormal{for}\quad 1\leq q < p^*\leq2\leq p \, ,
\end{equation}
which allows to appeal in the induction recursively to estimates with exponents $q'\in(q,p^*)$ or $p'>p$ when applying Hölder's inequality.
Thus, all estimates acquire a dependence on $p$ and $q$. \\

Next, we outline the general strategy for this section. 
Our goal is to apply \eqref{eq:sg} to $F=\Pi^-_{x\beta t}(y)$, see
Subsection~\ref{sec:block5} below. Hence, it is necessary to establish estimates for both
\begin{itemize}
	\item $ \EE \Pi^{-}_{ x \beta t} (y) $, 
		which relies on a suitable choice of the counterterms $c^{(k')}$ for $k'<\kk$, 
and is carried out in Subsection~\ref{sec:block3}, 
	and

	\item  $\EE^\frac{1}{q}|\delta\Pi^-_{x\beta t}(y)|^q$, which is established by
		``higher order modelledness'' of $ \delta \Pi^{-}_{x\beta}$.
\end{itemize}
Let us be more precise on the latter. 
Informally, one may think of $\delta\Pi^-_{x\beta}$ as the multilinear (in $\xi_\ve$) $\Pi^-_{x\beta}$ where one instance of $\xi_\ve$ is replaced by $\delta\zeta_\ve\coloneqq\ve^{s-D/2}\delta\zeta(\cdot/\ve^2,\cdot/\ve)$.
While $\xi_\ve$ only has (annealed) Hölder regularity $s-D/2$, the field $\delta\zeta_\ve\in\dot
H^s(\RR^{1+d})$ has $D/2$ more derivatives (although only in $L^2(\RR^{1+d})$).
Consequently, for some modelled distribution $d\Gamma_{xz}$, the Malliavin derivative $\delta\Pi^-_{x\beta}$ is modelled around  $z\in\RR^{1+d}$ by
$(d\Gamma_{xz}\Pi^-_z)_\beta$ to order $|\beta|-2+D/2$:
in an annealed $L^2(\RR^{1+d})$--sense
\begin{equation}\label{eq_informal_singRec}
	\begin{aligned}
	\big|\big( \delta\Pi^-_x - d\Gamma_{xz}\Pi^-_z\big)_{\beta t}(y+z) \big|
	\lesssim \ve^{|\beta|-\langle\beta\rangle}
(\sqrt[4]{t})^{\alpha-2} (\sqrt[4]{t}+|y|)^{\langle\beta\rangle+D/2- \alpha} \, ,
	\end{aligned}
\end{equation}
which should be compared to \eqref{est_Pi-},
see Proposition~\ref{prop_recIII} below for the precise statement.
This point of view is  carried out rigorously in Subsection~\ref{sec:block4}.
Introducing this enhanced structure and establishing \eqref{eq_informal_singRec} is the core
replacement, when comparing to the regular reconstruction argument in Proposition~\ref{prop:regrec}.
Therefore, we require another set of algebraic arguments to be carried out for the
Malliavin derivative $\delta\Gamma_{xy}$, the new 
$d\Gamma_{xy}$, and their interplay via $d\Gamma_{xy}-d\Gamma_{xz}\Gamma_{zy}$ (the latter is convenient to track the continuity of $d\Gamma_{xy}\Pi_y$ in $y$), see Subsection~\ref{sec:block1} below. 

Finally, having estimated $\Pi^-_x$ via the spectral gap inequality, we can continue as in Section~\ref{sec:regular} with integration and three-point argument for $\Pi_x$ and $\Gamma_{xy}$, see Subsection~\ref{sec:block6} below.
However, several additional steps have to be carried out in order to provide the input for the next stage of the induction. 
As for the algebraic arguments, we do a second round of integration and three-point
argument for $\delta\Pi_x$ and $\delta\Gamma_{xy}$, which is summarized in
Subsection~\ref{sec:block7} below. 
Moreover, we perform a third round of integration and three-point argument for
$\delta\Pi_x-d\Gamma_{xy}\Pi_y$ and $d\Gamma_{xy}-d\Gamma_{xz}\Gamma_{zy}$.
Additionally, a fourth round of three-point argument for $d\Gamma_{xy}$ is necessary.
The latter two are established in Subsection~\ref{sec:block9} below. 
Lastly, in Subsection~\ref{sec:block8} we establish qualitative continuity of $\Pi^-_x$, $\Pi_x$,
and $\delta\Pi_x$.

\medskip

The steps that we outlined above are summarised in the following table, borrowed from
\cite{LOTT24}, which the reader may use as a guiding aid throughout this section:

\begin{figure}[H]
		\small
		\begin{center}
			\begin{tabular}{|c|c|c|}
					\hline 
					{\bf Step} & {\bf Description} & {\bf Statement}\\  
					\hhline{|===|}
					1.~a) & Algebraic argument I  & see Proposition \ref{prop_1} \\ 
					\hline
					1.~b) -- c) & Algebraic argument II  & Proposition
					\ref{prop1bc}\\
					\hline
					1.~d) & Algebraic argument IV  & Proposition
					\ref{prop:alg4} \\
					\hline
					1.~e) & Algebraic argument III & Proposition \ref{prop:alg3}\\
					\hline
					2.~a) & Recentering of $ \Pi^{-}_{\beta}$ & see Proposition
					\ref{prop_recentre_Pi-}\\
					\hline
					2.~b) & Symmetries of $ \Pi^{-}_{\beta}$ & see Proposition
					\ref{prop_covariance_Pi-} \\
					\hline
					2.~c) & Annealed space-time average exists  & Proposition \ref{prop:lim_exists}\\
					\hline
					3.~a) -- b) & Renormalisation (BPHZ)  & Proposition \ref{prop_bphz_choice}\\
					\hline
					3.~c) & Expectation & Proposition \ref{prop:expectation}\\
					\hline
					4.~a) -- b) &  Boundedness and continuity of $
					\delta\Pi^{-}_{\beta} $  & Proposition
					\ref{prop:bdd_cont_delta_Pi-} \\
					\hline
					4.~c) & Reconstruction III & Proposition \ref{prop_recIII}\\
					\hline 
					4.~d) & Sobolev & Proposition \ref{prop_sobolev}\\
					\hline
					5 &  SG inequality & Proposition~\ref{prop_SG} \\
					\hline
					6.~a) & Integration I  & see Proposition \ref{prop_integration}\\
					\hline
					6.~b) &  Symmetries of $ \Pi_{\beta}$ &  see Proposition \ref{prop_covariance_Pi}\\
					\hline
					6.~c) -- e) & Three-point argument I & see Proposition \ref{prop_3pt}\\
					\hline
					7.~a) -- b) & Integration II & Proposition \ref{prop_intII}\\
					\hline 
					7.~c) -- d) & Three-point argument II  & Proposition \ref{prop_3ptII}\\
					\hline
					8.~a) & Boundedness and continuity of $
					\Pi^{-}_{ \beta}$   & see Proposition \ref{prop_2}\\
					\hline
					8.~b) & Boundedness and continuity of $
					\Pi_{\beta}$ & see Proposition \ref{prop_qualitative_pi}\\
					\hline
					8.~c) & Boundedness and continuity of $\delta\Pi_{ \beta}$  & Proposition \ref{prop_qualitative_deltaPi}\\
					\hline
					9.~a) --b) & Integration III  & Proposition \ref{prop_intIII}\\
					\hline
					9.~c) & Three-point argument III & Proposition
					\ref{prop_3ptIII}\\
					\hline
					9.~d) & Three-point argument IV & Proposition \ref{prop_3ptIV}\\
					\hline
			\end{tabular} 
		\end{center}
		\caption{Summary of arguments for singular multiindices.}\label{fig:singular}
\end{figure}

In order to avoid unnecessary case distinctions throughout this section, we track only a gain of modelledness by
$D/\p$ in $L^{\p}(\RR^{1+d})$, for $\p>2$ chosen sufficiently large such that
\begin{equation}\label{p_large2}
\alpha-2+\frac{D}{\p}<0 \, .
\end{equation}
On the other hand, the reconstruction argument requires choosing $\p$ sufficiently small so that
\begin{equation}\label{p_small}
\kk\alpha+\frac{D}{\p}>0 \, ,
\end{equation}
which is compatible with $\p>2$ by \eqref{varianceblowup}, 
and compatible with \eqref{p_large2} by \eqref{subcritical}.

As already mentioned in \eqref{eq_informal_singRec} we have to establish control in an $L^{\p}$--sense. 
In fact, we introduce a
localisation length-scale $ r \geqslant \varepsilon$ and perform all relevant estimates 
with respect to 
\begin{equation}
	\begin{aligned}
	\bigg(\int_{B_{r}(x)} dz \, \big| \cdot \big|^{\p}\bigg)^{1/\p}\,.
	\end{aligned}
\end{equation}
As can be seen in \eqref{est_dGamma} below, this localisation is necessary since the
established estimate diverges as $r\to\infty$.
We also refer to \cite[Section~2.12]{BOT} for a detailed discussion.

Depending on $\p$ and the current multiindex $\beta$ of the induction step, we also choose $\kappa$ from the definition \eqref{eq_bracket} of the discounted homogeneity sufficiently small such that for all populated $\gamma\preceq\beta$ it holds
$\kk\alpha+D/\p+\langle\gamma\rangle-|\gamma|+\langle\beta\rangle-|\beta|>0$, 
and such that $|\gamma|>\alpha+D/\p$ implies $\langle\gamma\rangle>\alpha+D/\p$.
This is possible by \eqref{p_small} and since there are only finitely many such $\gamma$ by Lemma~\ref{lem_ind_legal}. 
We emphasize that \eqref{p_large2} is equivalent to $2+(\kk-1)\alpha>\kk\alpha+D/\p$, so that also $2+(\kk-1)\alpha+\langle\gamma\rangle-|\gamma|+\langle\beta\rangle-|\beta|>0$ for all populated $\gamma\preceq\beta$.

As in Section~\ref{sec:regular} all necessary conditions and resulting estimates hold for all $x,y\in\RR^{1+d}$, $t\in(0,\infty)$, and $p\in[1,\infty)$, and additionally on $q\in[1,2)$ and $\p>2$. Again we drop the corresponding quantifiers for notational simplicity.
Moreover, all implicit constants only depend on $ \alpha$, $\beta$, $\kappa$, $ p$, $\p$, and $q$.

\subsection{Block 1: Algebraic arguments}\label{sec:block1}

As in the case of regular indices, we begin by estimating $ \Gamma_{\beta}^{\gamma}$ 
	in the case where
$ \gamma$ is not purely polynomial.
Indeed, 
Step~1.a) is identical to Proposition~\ref{prop_1}, and refer to the statement there. 
Next, similar estimates are established for the Malliavin derivative $ \delta \Gamma$ as
well as $ d \Gamma$.

\begin{proposition}[Algebraic argument II, Steps 1.b) and 1.c)]\label{prop1bc}
Let $[\beta]\geq0$ and assume 
that for all populated $\beta'\prec\beta$ we have $\eqref{est_pin}_{\beta'}$ and 
\begin{equation}\label{est_deltapin}
\EE^{\frac{1}{q}} \left|
\delta\pi^{(\n)}_{xy\, \beta'}\right|^{ q} 
\lesssim \ve^{|\beta'|-\langle\beta'\rangle} |x-y|^{\langle \beta' \rangle - |\n|} \w \, .
\end{equation}
Then for all \(\gamma\) populated and not purely polynomial 
\begin{equation}\label{est_deltaGamma}
\EE^{\frac{1}{q}} \left|
(\delta\Gamma_{xy})_{ \beta}^{\gamma}
\right|^{ q} 
\lesssim \ve^{|\beta|-\langle\beta\rangle-|\gamma|+\langle\gamma\rangle} |x-y|^{\langle \beta\rangle - \langle\gamma\rangle} \w \,.
\end{equation}
\end{proposition}

Here and in the following we have the implicit understanding that Malliavin derivatives exist whenever we write them down, i.e.~the hypothesis of Proposition~\ref{prop1bc} contains that $\pi_{xy\beta'}^{(\n)}$ is Malliavin differentiable, 
and the conclusion contains that $(\Gamma_{xy})_\beta^\gamma$ is Malliavin differentiable. 
For algebraic-, reconstruction-, and three-point arguments this conclusion is a
simple consequence of the product rule for the Malliavin derivative.
For the integration arguments, we refer to \cite[Section~7]{LOTT24} for details. 

\begin{proof}
The proof is analogous to the proof of Proposition~\ref{prop_1}, applying additionally the Malliavin derivative to the representation \eqref{eq_gamma_prod_rep}.
\end{proof}

\begin{proposition}[Algebraic argument III, Step 1.e)]\label{prop:alg3}
Let $[\beta]\geq0$ and assume that for all populated $\beta'\prec\beta$ and all populated $\gamma$ we have $\eqref{est_Gamma}_{\beta'}^\gamma$ and 
$(\Gamma_{xy}\Gamma_{yz})_{\beta'}^\gamma 
= (\Gamma_{xz})_{\beta'}^\gamma$ , 
and that for all $\n$ and $\ve\leq r$
\begin{equation}\label{est_dGamma_pp_inc}
\Big(\int_{B_r(x)} dz \, \EE^\frac{\p}{q}\big| \big(d\Gamma_{x\,y+z}-d\Gamma_{xz}\Gamma_{z\,y+z}\big)_{\beta'}^{\delta_\n}\big|^q \Big)^\frac{1}{\p}
\lesssim \ve^{|\beta'|-\langle\beta'\rangle} 
\begin{cases} 
|y|^{\alpha+D/\p-|\n|} (|y|+r)^{\langle\beta'\rangle-\alpha} \w \, , \\
(|y|+r)^{\langle\beta'\rangle-|\n|+D/\p} \w \,.
 \end{cases} 
\end{equation}
Then for all $\gamma$ populated and not purely polynomial, $\bar
n\coloneqq\max\{|\n|\colon\gamma(\n)\neq0\}$, and $\ve\leq r$
\begin{align}
&\Big(\int_{B_r(x)} dz \, \EE^\frac{\p}{q}\big| \big(d\Gamma_{x\,y+z}-d\Gamma_{xz}\Gamma_{z\,y+z}\big)_{\beta}^{\gamma}\big|^q \Big)^\frac{1}{\p} \\
&\lesssim 
\ve^{|\beta|-\langle\beta\rangle-|\gamma|+\langle\gamma\rangle} 
\begin{cases} 
|y|^{\alpha+D/\p-\bar n} (|y|+r)^{\langle\beta\rangle-\langle\gamma\rangle+\bar n-\alpha} \w \, , \\
(|y|+r)^{\langle\beta\rangle-\langle\gamma\rangle+D/\p} \w \,.
 \end{cases} \label{est_dGamma_inc}
\end{align}
\end{proposition}

\begin{proof}
Since $\gamma$ is populated and not purely polynomial we can write
$z^\gamma=z_{k_1}\cdots z_{k_m} z_{\n_1}\cdots z_{\n_l}$
for some $k_1,\dots,k_m$ where $m\geq1$, and $\n_1,\dots,\n_l$ where $l\geq0$.
Applying iteratively \eqref{e_def_gamma} and \eqref{e_def_dgamma} yields
\begin{equation}
d\Gamma_{xz}\Gamma_{z\, y+z}z^\gamma
= z_{k_1}\cdots z_{k_m}
\sum_{i=1}^l d\Gamma_{xz}\Gamma_{z\, y+z}z_{\n_i}
\prod_{\substack{j=1\\ j\neq i}}^l
\Gamma_{xz}\Gamma_{z\, y+z}z_{\n_j} \, ,
\end{equation}
and 
\begin{equation}
d\Gamma_{x\,y+z} z^\gamma
= z_{k_1}\cdots z_{k_m}
\sum_{i=1}^l d\Gamma_{x\,y+z}z_{\n_i}
\prod_{\substack{j=1\\ j\neq i}}^l
\Gamma_{x\,y+z}z_{\n_j} \, ,
\end{equation}
and thus componentwise
\begin{align}
&(d\Gamma_{x\,y+z}-d\Gamma_{xz}\Gamma_{z\,y+z})_\beta^\gamma \\
&= \sum_{\delta_{k_1}+\dots+\delta_{k_m}+\beta_1+\dots+\beta_l=\beta}
\sum_{i=1}^l (d\Gamma_{x\,y+z}-d\Gamma_{xz}\Gamma_{z\,y+z})_{\beta_i}^{\delta_{\n_i}}
\prod_{\substack{j=1\\ j\neq i}}^l (\Gamma_{x\,y+z})_{\beta_j}^{\delta_{\n_j}} \, .
\end{align}
Here we directly used $(\Gamma_{xz}\Gamma_{z\,y+z})_{\beta_j} = (\Gamma_{x\,y+z})_{\beta_j}$, which follows by assumption since $\beta_1,\dots,\beta_l\prec\beta$ as we argue now.
From Lemma~\ref{lem_hom_props} and $m\geq1$ follows that 
\begin{equation}\label{eq_supp1_singStep1}
|\beta|-|0| 
= \sum_{i=1}^m(|\delta_{k_i}|-|0|) + \sum_{i=1}^l(|\beta_i|-|0|)
> |\beta_j|-|0| \quad \textnormal{for all }j=1,\dots,l \, .
\end{equation}
The condition $[\gamma]\geq-1$ (which holds since $\gamma$ is populated) translates into $k_1+\dots+k_m-m-l\geq-1$,
hence 
\begin{equation}\label{eq_supp2_singStep1}
[\beta]
=k_1+\dots+k_m-m+[\beta_1]+\dots+[\beta_l]
\geq k_1+\dots+k_m-m-(l-1)+[\beta_j] 
\geq [\beta_j] \quad\textnormal{for all }j=1,\dots,l \, ,
\end{equation}
where in the first inequality we used that $\beta_1,\dots,\beta_l$ are populated (see Lemma~\ref{lem_sets}) and thus $[\beta_1],\dots,[\beta_l]\geq-1$.
Altogether \eqref{eq_supp1_singStep1} and \eqref{eq_supp2_singStep1} yield the desired $|\beta|_{\prec}>|\beta_j|_\prec$ for all $j=1,\dots,l$.

Turning to the estimate \eqref{est_dGamma_inc}, 
we apply $(\int_{B_r(x)}dz \, \EE^\frac{\bar p}{q}|\cdot|^q)^\frac{1}{\bar p}$ to the above representation, 
use the triangle inequality and Hölder's inequality both
with respect to $\EE^\frac{1}{q}|\cdot|^q$.
We then use \eqref{est_Gamma} (together with $|x-y-z|\leq|y|+r$) and
\eqref{est_dGamma_pp_inc} to see that the left-hand side of \eqref{est_dGamma_inc} is up to a constant bounded by 
\begin{align}
&\sum_{i=1}^l \ve^{\sum_{j=1,j\neq i}^m (|\beta_j|-\langle\beta_j\rangle)}
(|y|+r)^{\sum_{j=1,j\neq i}^m (\langle\beta_j\rangle-|\n_j|)}
\ve^{|\beta_i|-\langle\beta_i\rangle} 
\begin{cases} 
|y|^{\alpha+D/\p-|\n_i|} (|y|+r)^{\langle\beta_i\rangle-\alpha} \w \\
(|y|+r)^{\langle\beta_i\rangle-|\n_i|+D/\p} \w 
\end{cases} \\
&\lesssim
\ve^{\sum_{j=1}^m (|\beta_j|-\langle\beta_j\rangle)}
\begin{cases} 
|y|^{\alpha+D/\p-\bar n}
(|y|+r)^{\sum_{j=1}^m (\langle\beta_j\rangle-|\n_j|)+\bar n-\alpha} \w \\
(|y|+r)^{\sum_{j=1}^m (\langle\beta_j\rangle-|\n_j|)+D/\p} \w \, ,
\end{cases} 
\end{align}
where we also used Lemma~\ref{lem_def_gamma}~(2) for $\langle\cdot\rangle$.
The desired estimate follows from additivity of $\langle\cdot\rangle-\langle0\rangle$ (see Lemma~\ref{lem_hom_props}) in the form of 
\begin{equation}
\langle\beta\rangle-\langle0\rangle 
= \sum_{i=1}^m(\langle\delta_{k_i}\rangle-\langle0\rangle) 
+ \sum_{i=1}^l(\langle\beta_i\rangle-\langle0\rangle)
= \langle\gamma\rangle-\langle0\rangle
-\sum_{i=1}^l (\langle\n_i\rangle-\langle0\rangle)
+ \sum_{i=1}^l(\langle\beta_i\rangle-\langle0\rangle)
\end{equation}
and the same identity with $\langle\cdot\rangle$ replaced by the homogeneity $|\cdot|$. 
\end{proof}

\begin{proposition}[Algebraic argument IV, Step 1.d)]\label{prop:alg4}
Let $[\beta]\geq0$ and assume that for all populated $\beta'\prec\beta$ and all 
$\gamma$ populated and not purely polynomial we have $\eqref{est_Gamma}_{\beta'}^\gamma$, 
and that for all populated $\beta'\prec\beta$ and $|\n|<\alpha+D/\p$  
\begin{equation}\label{est_dpin}
\Big(\int_{B_r(x)} dz \, \EE^\frac{\p}{q}\big| d\pi_{xz\beta'}^{(\n)}\big|^q \Big)^\frac{1}{\p}
\lesssim 
\ve^{|\beta'|-\langle\beta'\rangle}
r^{\langle\beta'\rangle-|\n|+D/\p} \w 
\quad\textnormal{for }\ve\leq r \, .
\end{equation}
Then for all $\gamma$ populated and not purely polynomial
\begin{equation}\label{est_dGamma}
\Big(\int_{B_r(x)} dz \, \EE^\frac{\p}{q}\big| \big(d\Gamma_{xz}\big)_\beta^\gamma \big|^q \Big)^\frac{1}{\p}
\lesssim 
\ve^{|\beta|-\langle\beta\rangle-|\gamma|+\langle\gamma\rangle}
r^{\langle\beta\rangle-\langle\gamma\rangle+D/\p} \w 
\quad\textnormal{for }\ve\leq r \,.
\end{equation}
\end{proposition}

We omit the proof of
Proposition~\ref{prop:alg4}, since it follows the same lines as the one of
Proposition~\ref{prop:alg3} using the representation, for 
 $ \gamma$ of the form $ \delta_{k_{1}}+ \cdots \delta_{k_{m}} +
\delta_{\n_{1}}+ \cdots + \delta_{\n_{l}} $, 
\begin{equation}
	\begin{aligned}
		d \Gamma_{xz} z^{\gamma} = 
		z_{k_{1}} \cdots z_{k_{m}}
		\sum_{i =1}^{l} d \pi^{(\n_{i})}_{xz} \prod_{j \neq i} \Gamma_{xz} z_{ \n_{j}}\,.
	\end{aligned}
\end{equation}

\subsection{Block 2: Properties of \texorpdfstring{$\Pi^-$}{Pi-}}\label{sec:block2}

Once more, we turn to estimates on $ \Pi^{-}$ at the current induction
level, beginning with several preliminary properties.
We notice that the proofs of Propositions~\ref{prop_covariance_Pi-}
and~\ref{prop_recentre_Pi-} apply generally, both for singular and regular indices $
\beta$. 
Indeed, 
\begin{itemize}
\item Step 2.a) is identical to Proposition~\ref{prop_recentre_Pi-} and yields $\eqref{Pi-_recentre}_\beta$, 
	i.e.~the recentering formula for $ \Pi^{-}$ holds.
\item Step 2.b) is identical to Proposition~\ref{prop_covariance_Pi-} and yields
	$\eqref{Pi-_shift}_\beta$, $\eqref{Pi-_parity}_\beta$, and
	$\eqref{Pi-_reflection}_\beta$, i.e.~the various covariances translate from $ \Pi
	_{ \beta' }$ for $\beta'\prec\beta$ to $ \Pi^{-}_{ \beta}$.
\end{itemize}

The goal of the next block (Block 3 below) is the choice of suitable counterterms, that allow $
\Pi_{x \beta t}^{-} $ to be uniformly (in $ \varepsilon$) well-defined.
To this end, we use the fundamental theorem of calculus in the form of 
$\lim_{t\to\infty} \EE\Pi^-_{x\beta t}(x) = \EE\Pi^-_{x\beta T}(x) + \int_T^\infty dt\,\frac{d}{dt} \EE\Pi^-_{x\beta t}(x)$ 
which is justified by the following estimate:

\begin{proposition}[Annealed space-time average exists, Step
	2.c)]\label{prop:lim_exists}
	Let $[\beta]\geq0$ and $ | \beta| < 2$.
	Assume that $\eqref{Pi-_shift}_\beta$ and $\eqref{Pi-_recentre}_\beta$ hold,
that $\eqref{est_Gamma}_\beta^\gamma$ holds for all $\gamma$ populated and not purely polynomial, 
and that $\eqref{est_Pi-}_{\beta'}$ holds for all populated $\beta'\prec\beta$. Then for all $T>0$
\begin{equation}\label{lim_exists}
\int_T^\infty dt\, \Big| \frac{d}{dt}\EE \Pi^-_{x\beta t}(y) \Big|
\lesssim \ve^{|\beta|-\langle\beta\rangle} (\sqrt[4]{T})^{\alpha-2} (\sqrt[4]{T}+|x-y|)^{\langle\beta\rangle-\alpha} \, .
\end{equation}
\end{proposition}

The proof to Proposition~\ref{prop:lim_exists} is identical to that of
\cite[Proposition~4.6]{LOTT24}.

\subsection{Block 3: Choice of the counterterm}\label{sec:block3}

As the terminology suggests, singular multiindices are less well-behaved than regular ones and require the renormalisation characteristic of the underlying singular SPDE. The following proposition establishes the existence of the corresponding counterterms:

\begin{proposition}[BPHZ renormalisation, Steps 3.a) and 3.b)]\label{prop_bphz_choice}
Assume that $[\beta]\geq0$ and $|\beta|<2$, that the assumption of Proposition~\ref{prop_2} holds, 
and that $\eqref{Pi-_shift}_\beta$, $\eqref{Pi-_parity}_\beta$, $\eqref{Pi-_reflection}_\beta$, and $\eqref{lim_exists}_\beta$ 
hold. 
Then there exists $c^{(\beta(\0))}_{\beta-\beta(\0)\delta_\0}\in\RR$ such that 
\begin{equation}\label{eq_bphz}
\lim_{t\to\infty} \EE \Pi^-_{x\beta t}(x)=0.
\end{equation}
For this choice, \eqref{est_d_bounded} also holds (with
$\beta'=\beta-\beta(\0)\delta_\0$ and $m=\beta(\0)$).
\end{proposition}

Before we proceed to the proof, and the choice\footnote{
	The choice of $ c$ is dictated since we demand the estimate \eqref{est_Pi-} of $\Pi^-_{x\beta t}(y)$ to hold for $t\in(0,\infty)$.
For $|\beta|<2$, \eqref{est_Pi-} implies \eqref{eq_bphz}, which in turn determines the counterterm $c$.} of $c$, we recall that $ c^{(k')} \in
\RR[\![ z_{k}]\!]$ does not contain any polynomial components.  
Since the choice of $c$ reflects the cancellation of expected values in the hierarchy,
we show that the symmetry in law of $ \xi$ yields that $ c^{(k')}_{ \beta}$ first
appears (and is defined) in the hierarchy representation of $ \Pi^{-}_{x \beta + k'
\delta_{ \bf{0}}}$, provided $ \beta + k' \delta_{\bf { 0}}$ is
singular.
Lastly, we notice that $ c^{(k')}_{\delta_{k }}= 0 $ for every $ k  $, since these
counterterms are already included in the definition of the $ W_{\bullet ,
	\varepsilon}$--polynomials:
\begin{equation}\label{eq_counterterm_deltak}
	\begin{aligned}
		\Pi^{-}_{x \delta_{k}+ k' \delta_{\bf 0}}
		= 
		 \ve^{(\kk-k)\alpha} \tbinom{k}{k'}W_{k-k',\ve}(\Pi_{x0}) 
		 -
		  \sum_{k' < k''<\kk} 
		  \tbinom{k''}{k'} c^{(k'')}_{\delta_{k }} W_{k''-k',\ve}(\Pi_{x0})
		  - c^{(k')}_{\delta_{k }}\,.	
	  \end{aligned}
\end{equation}
Hence, using that 
\begin{equation}
	\begin{aligned}
		\EE \int dz \  \Psi_{t} (x-z) W_{k - k', \varepsilon}( \Pi_{x0})(z)
		= \EE W_{k - k', \varepsilon}( \Pi_{x0})=0\,,
	\end{aligned}
\end{equation}
which is a consequence of the shift invariance \eqref{Pi-_shift},
centredness of the $ W_{\bullet ,
	\varepsilon}$--polynomials (see Lemma~\ref{lem_zeromean}),
	and the fact that $ \Pi_{x0}$ is independent of $ x$, \eqref{eq_bphz} is already satisfied with $ c^{(k')}_{\delta_{k }}= 0 $. 

\begin{proof}
First, we separate from the hierarchy \eqref{hierarchy} the term we want to
determine, and define 
\begin{equation}
\widetilde\Pi^-_{x\beta}\coloneqq \Pi^-_{x\beta} + c^{(\beta(\0))}_{\beta-\beta(\0)\delta_\0} \, ,
\end{equation}
with the understanding that $c^{(\beta(\0))}_{\beta-\beta(\0)\delta_\0}$ vanishes unless $\beta(\0)$ is odd and $0<\beta(\0)<\kk$.
Using Lemma~\ref{lem_dep_pi} 2., $\widetilde\Pi^-_{x\beta}$ only depends on
$\Pi_{x\beta'}$ for $\beta'\prec\beta$ and on $c^{(k')}_{\beta'}$ for
$\beta'+k'\delta_\0\prec\beta$. Hence, all the counterterms appearing int $
\widetilde{\Pi}_{x \beta}^{-}$ have been chosen
already at this stage.
Consequently, we can follow the same steps in the proof of \eqref{est_Pi-_bounded} in Proposition~\ref{prop_2} to see that 
\begin{equation}\label{est_Pi-tilde_bounded}
\begin{aligned}
\EE^{\frac{1}{p}} \left| 
\partial^\n \widetilde\Pi^-_{x\beta}(y) 
\right|^{ p} 
\lesssim 
\ve^{\alpha-2-|\n|}
( \ve+ | x-y|)^{| \beta| - \alpha} \, .
\end{aligned}
\end{equation}

Next, we turn to the choice of $c$. 
From \eqref{eq_supp1_lem_homProp} we read off that any $\beta$ satisfies either $\sum_\n\beta(\n)=0$ or $|\beta|\geq\sum_\n|\n|\beta(\n)$.
In particular, $|\beta|<2$ implies that $\sum_\n|\n|\beta(\n)\leq1$.
For $\beta$ with $\sum_{\n}|\n|\beta(\n)=1$ we deduce 
\begin{equation}
	\begin{aligned}
		\EE\Pi^-_{x\beta t}(x)
		&=
		\int dz \Psi_{t} (x-z) \EE\Pi^-_{x\beta }(z)
		=
		-
		\int dz \Psi_{t} (x-z) \EE\Pi^-_{Rx\beta }(Rz)\\
		&=
		-
		\int dz \Psi_{t} (Rx-z) \EE\Pi^-_{Rx\beta }(z)
		= - \EE\Pi^-_{Rx\beta t}(Rx)= -\EE\Pi^-_{x\beta t}(x)\,,
	\end{aligned}
\end{equation}
which implies $\EE\Pi^-_{x\beta t}(x)=0$. 
Here we used, \eqref{Pi-_reflection}, the
invariance under spatial reflections of $\Psi_t$ defined in \eqref{eq_Psi}, the
invariance of the
law of $\xi$
under spatial shifts and reflection, and \eqref{Pi-_shift}.

We turn to $\beta$ with $\sum_\n|\n|\beta(\n)=0$. 
Here we first note that $|\beta|<2$ and $\beta$ not purely polynomial imply by \eqref{eq_supp1_lem_homProp} that 
$2>\alpha+2+(\kk-1)\alpha-\alpha\beta(\0)$, and hence $\beta(\0)<\kk$.
Furthermore, because $\beta$ is a multiindex over odd $k$ and $\sum_\n|\n|\beta(\n)=0$ we have $(-1)^{1+[\beta]}=(-1)^{1-\beta(\0)}$, so that \eqref{Pi-_parity} implies by the invariance of the law of $\xi$ under sign change that $\EE\Pi^-_{x\beta t}(x)=0$ for any choice of $c$ unless $\beta(\0)$ is odd.

Finally, we are only left with $|\beta|<2$ that satisfy $\sum|\n|\beta(\n)=0$ with odd
$\beta(\0)< \kk $.
For such $\beta$ we have $c^{(\beta(\0))}_{\beta-\beta(\0)\delta_\0}$ at our disposal,
which we now choose:
\begin{equation}
c^{(\beta(\0))}_{\beta-\beta(\0)\delta_\0}
\coloneqq \lim_{t\to\infty} \EE\widetilde\Pi^-_{x\beta t}(x)
= \EE\widetilde\Pi^-_{x\beta T}(x)
+ \int_T^\infty dt\, \frac{d}{dt}\EE\widetilde\Pi^-_{x\beta t}(x) \, , 
\end{equation}
where the second equality holds for all $T>0$.
Note that the integral on the right-hand side is well-defined by \eqref{lim_exists}.
Moreover, we emphasize that $c^{(\beta(\0))}_{\beta-\beta(\0)\delta_\0}$ is independent of $x$ by \eqref{Pi-_shift} and the invariance of the law of $\xi$ under translations, and is indeed a constant as desired. 
Furthermore, \eqref{eq_bphz} holds. 

It remains to establish \eqref{est_d_bounded}. 
First, we notice that \eqref{est_Pi-tilde_bounded} and \eqref{eq_momentbound_semigroup} imply
\begin{equation}
|\EE\widetilde\Pi^-_{x\beta T}(x)|
\lesssim \ve^{\alpha-2} (\ve+\sqrt[4]{T})^{|\beta|-\alpha} \, .
\end{equation}
Combining the above upper bound with \eqref{lim_exists}, we see that the choice $T=\ve^4$
yields \eqref{est_d_bounded}. 
\end{proof}

Having chosen the counterterms appropriately, we can proceed with estimating the expectation of $\Pi^-_{x\beta}$ as is required for the application of the spectral gap inequality. 

\begin{proposition}[Expectation, Step 3.c)]\label{prop:expectation}
Let $[\beta]\geq0$ and assume that $\eqref{Pi-_recentre}_\beta$ and $\eqref{lim_exists}_\beta$ hold,
that $\eqref{est_Gamma}_\beta^\gamma$ holds for all $\gamma$ populated and not purely polynomial, 
and that $\eqref{est_Pi-}_{\beta'}$ holds for all populated $\beta'\prec\beta$. 
Then 
\begin{equation}\label{est_expectation}
|\EE \Pi^-_{x\beta t}(y)| 
\lesssim \ve^{|\beta|-\langle\beta\rangle} (\sqrt[4]{t})^{\alpha-2} (\sqrt[4]{t}+|x-y|)^{\langle\beta\rangle-\alpha} \, .
\end{equation}
\end{proposition}

The proof of Proposition~\ref{prop:expectation} makes use of $\EE\Pi^-_{x\beta t}(x) = -
\int_t^\infty ds\,\frac{d}{ds}\EE\Pi^-_{x\beta s}(x)$, which is a consequence of
\eqref{eq_bphz}.
The estimate obtained in \eqref{lim_exists}  can then be transferred to another active variable $y$ using reexpansion with $\Gamma_{xy}$. The details can be found e.g.~in \cite[Proposition~3.9]{GT}. 

\subsection{Block 4: Malliavin derivative of \texorpdfstring{$\Pi^-$}{Pi-}}\label{sec:block4}

The main goal of this section is to prove moment bounds of $ \delta \Pi^{-}$ which we
then lift to bounds in terms of $ \Pi^{-}$ in the subsequent section using the spectral
gap inequality. 
To this end, we establish a reconstruction-like estimate in Proposition~\ref{prop_recIII} below, which uses the following qualitative smoothness of $\delta\Pi^-$. 

\begin{proposition}[Boundedness and continuity of $\delta\Pi^-$, Steps 4.a) and 4.b)]\label{prop:bdd_cont_delta_Pi-}
Assume that the hypothesis of Proposition~\ref{prop_2} is satisfied, and that additionally 
for all populated $\beta'\prec\beta$ and all~$\n$ 
\begin{equation}\label{est_deltaPi_bounded}
\begin{aligned}
\EE^{\frac{1}{q}} \left|
\partial^\n \delta \Pi_{x \beta'}(y)
\right|^{ q} \lesssim 
\ve^{\alpha-|\n|}
( \ve + |x-y|)^{| \beta' |- \alpha}\w \,.
\end{aligned}
\end{equation}
Then for every $\n$ and $l\geq0$ 
\begin{equation}\label{est_deltaPi-_bounded}
\begin{aligned}
\EE^{\frac{1}{q}} \left| 
\partial^\n \delta\Pi^-_{x\beta}(y) 
\right|^{ q} 
\lesssim 
\ve^{\alpha-2-|\n|}
( \ve+ | x-y|)^{| \beta| - \alpha} \w \, ,
\end{aligned}
\end{equation}
and
\begin{equation}\label{est_deltaPi-_taylorremainder}
\begin{aligned}
\EE^{\frac{1}{q}} \left|
(\mathrm{id}-T_z^{\leq l}) \partial^\n \delta\Pi_{x \beta}^{-}(y) 
\right|^{ q} 
\lesssim 
\ve^{\alpha-2-|\n|-1-l}
|y-z|^{1+l}
( \ve+ | x-y| + |x-z|)^{| \beta| - \alpha} \w \,.
\end{aligned}
\end{equation}
\end{proposition}

The proof of Proposition~\ref{prop:bdd_cont_delta_Pi-} is analogous to the proof of \eqref{est_Pi-_bounded} and \eqref{est_Pi-_taylorremainder} in Proposition~\ref{prop_2}.
\begin{proposition}[Reconstruction III, Step 4.c)]\label{prop_recIII} 
Let $[\beta]\geq0$ and assume that $\eqref{est_deltaPi-_taylorremainder}_\beta$ holds, that $\eqref{est_dGamma_inc}_\beta^\gamma$ and 
$\eqref{est_dGamma}_\beta^\gamma$ hold for all $\gamma$ populated and not purely polynomial, 
and that $\eqref{est_Pi-_bounded_preliminary}_{\beta'}$, $\eqref{est_Pi-}_{\beta'}$, $\eqref{est_Pi}_{\beta'}$, $\eqref{eq_recentre_Pi}_{\beta'}$, 
$\eqref{est_Pi_taylorremainder}_{\beta'}$, and $\eqref{eq:intIII_qualitative}_{\beta'}$ hold for all populated $\beta'\prec\beta$. 
Then
\begin{align}
&\Big( \int_{B_r(x)} dz\, \EE^\frac{\p}{q}\big| \big( \delta\Pi^-_x - d\Gamma_{xz}\Pi^-_z\big)_{\beta t}(y+z) \big|^q\Big)^\frac{1}{\p} \\
&\lesssim \ve^{|\beta|-\langle\beta\rangle}
(\sqrt[4]{t})^{\alpha-2} (\sqrt[4]{t}+|y|)^{D/\p} 
(\sqrt[4]{t}+|y|+r)^{\langle\beta\rangle-\alpha} \w 
\quad\textnormal{for }\ve\leq r \, . \label{est_deltaPi-_inc}
\end{align}
\end{proposition}

As in the case of regular indices, the core ingredient to obtain \eqref{est_deltaPi-_inc} is  the reconstruction of the
zero function, see \eqref{eq_supp13_singRec} below. We also refer to the paragraph below
Proposition~\ref{prop:regrec}.

\begin{proof}
We define a family of distributions $F_{xz}$ by 
\begin{equation}\label{eq_germ}
F_{xz}
\coloneqq \delta\Pi^-_x - d\Gamma_{xz}\Pi^-_z - \delta\zeta_\ve \, 1 
- d\Gamma_{xz}\sum_{k \geq \kk} \ve^{(\kk-k)\alpha} T_z^{<|\cdot|-2}\big(z_k W_{k, \varepsilon}(\Pi_z)\big) \, ,
\end{equation}
where $\delta\zeta_\ve(x)\coloneqq\ve^{s-D/2}\delta\zeta(x_0/\ve^2,x_{1,\dots,d}/\ve)$.
In Step~1 we prove ``\emph{vanishing at the base point}'', i.e.~$F_{xz}(z)=0$.
In Step~2 we show ``\emph{continuity in the base point}'' of $z\mapsto F_{xz}$.
Both are combined in Step~3 by a general reconstruction argument, estimating $F_{xz}$ at $z$. 
In Step~4 we post-process the obtained estimate to deduce an off-diagonal estimate of $F_{xz}$ at $y$.
In Steps~5 and 6 we estimate $(\delta\Pi^-_x-d\Gamma_{xz}\Pi^-_z)-F_{xz}$, 
so that we can conclude in Step~7 with the triangle inequality.

\textbf{Step 1 \textnormal{(vanishing at the base point)}.} 
We show that
\begin{align}\label{eq_supp0_singRec}
\big(\delta\Pi^-_x-d\Gamma_{xz}\Pi^-_z\big)_\beta(z)
= 1_{\beta=0} \, \delta\zeta_\ve(z) 
+ \big( d\Gamma_{xz}\sum_k \ve^{(\kk-k)\alpha} T_z^{<|\cdot|-2}\big(z_k W_{k, \varepsilon}(\Pi_z)\big)(z)\big)_\beta \, .
\end{align}
For brevity we are going to drop the multiindex $\beta$ from the following identities, but emphasize that (some of) these identities only hold for their respective $\beta$-components with $|\beta|<2$. 
First note that by Leibniz' rule and the definition \eqref{eq_piminus} of $\Pi^-_{x}$ it holds for components of homogeneity less than $2$ 
\begin{equation}\label{eq_supp1_singRec}
\delta\Pi^-_x
= \sum_{k \geq \kk} \ve^{(\kk-k)\alpha} z_k k W_{k-1, \varepsilon}(\Pi_x)\delta\Pi_x 
- \sum_{k<\kk} c^{(k)} k W_{k-1, \varepsilon}(\Pi_x) \delta\Pi_x 
+ \delta\zeta_\ve \, 1 \, .
\end{equation}
Here we also used for $F(\zeta)=\xi_\ve(\varphi)=\zeta(\ve^{s+D/2}\varphi(\ve^2\cdot,\ve\cdot))$ for a test function $\varphi$ that $\partial F/\partial\zeta = \ve^{s+D/2}\varphi(\ve^2\cdot,\ve\cdot)$ by \eqref{frechet}, 
and hence $\delta F(\zeta,\delta\zeta) = (\ve^{s+D/2}\varphi(\ve^2\cdot,\ve\cdot),\delta\zeta)$ by \eqref{gateaux}, 
which coincides with $(\varphi,\delta\zeta_\ve)$, i.e.~we informally write in a connoisseur's notation $\delta(\xi_\ve) = \delta\zeta_\ve$.
Furthermore, $W_{k, \varepsilon}(\Pi_x)$ is a polynomial in $\Pi_x$ of degree $k$ with
coefficients $ ( a_{l , \varepsilon})_{l =0}^{k} $ that only depend on the law of $\xi_\ve$, so that 
\begin{equation}\label{eq_supp2_singRec}
	\begin{aligned}
		&d\Gamma_{xz} z_k W_{k, \varepsilon} (\Pi_z)
= d\Gamma_{xz} z_k \sum_{l=0}^k a_{l, \varepsilon} \Pi_z^l 
\stackrel{\eqref{e_def_dgamma}}{=} z_k \sum_{l=0}^k a_{l, \varepsilon} l(\Gamma_{xz}\Pi_z)^{l-1} d\Gamma_{xz}\Pi_z \\
&\stackrel{\eqref{eq_recentre_Pi}}{=} z_k \sum_{l=0}^k a_{l, \varepsilon} l(\Pi_x)^{l-1} d\Gamma_{xz}\Pi_z 
= z_k W_{k, \varepsilon}'(\Pi_x) d\Gamma_{xz}\Pi_z
\stackrel{\text{Lemma }\ref{lem_appell}}{=} z_k k W_{k-1, \varepsilon}(\Pi_x) d\Gamma_{xz}\Pi_z \, .
	\end{aligned}
\end{equation}
We point out that we only need the $\beta$-component of this identity, 
so that by Lemma~\ref{lem_dep_pi} only $\eqref{eq_recentre_Pi}_{\beta'}$ for $\beta'\prec\beta$ is required. 
By the same argumentation we obtain
\begin{equation}\label{eq_supp3_singRec}
	d\Gamma_{xz} c^{(k)} W_{k, \varepsilon}(\Pi_z)
= c^{(k)} k W_{k-1, \varepsilon}(\Pi_x) d\Gamma_{xz}\Pi_z \, .
\end{equation}
Hence, combining \eqref{eq_supp1_singRec} with \eqref{eq_supp2_singRec} and
\eqref{eq_supp3_singRec} inserted into  \eqref{eq_piminus} 
yields for components of homogeneity less than $2$ 
\begin{align}
\delta\Pi^-_x - d\Gamma_{xz}\Pi^-_z
&= \sum_{k \geqslant \kk} \ve^{(\kk-k)\alpha} z_k k W_{k-1, \varepsilon}(\Pi_x) (\delta\Pi_x - d\Gamma_{xz}\Pi_z) \\
&\,- \sum_{k<\kk} c^{(k)} k W_{k-1, \varepsilon}(\Pi_x) (\delta\Pi_x - d\Gamma_{xz}\Pi_z) \\
&\,+ \delta\zeta_\ve \, 1
+ d\Gamma_{xz}\sum_{k \geq \kk} \ve^{(\kk-k)\alpha} T_z^{<|\cdot|-2} \big(z_k W_{k, \varepsilon}(\Pi_z)\big) \, .
\end{align}
Evaluating the $\beta$-component of this identity at $z$ yields \eqref{eq_supp0_singRec}
as a consequence of \eqref{eq:intIII_qualitative},
provided only $\beta'$ components of the latter are required for $\beta'\prec\beta$ and $|\beta'|<2$.
This is indeed the case: $\beta'\prec\beta$ follows analogous to Lemma~\ref{lem_dep_pi} (Item~1), 
while for $|\beta'|<2$ we need to argue in addition that
$\beta_0+\dots+\beta_l=\beta$ for $\beta_0$ with $\sum_k\beta_0(k)>0$ and $|\beta|<2$ implies $|\beta_i|<2$ for all $i=1,\dots,l$.
This is a consequence of 
\begin{align}
2&>|\beta|
\stackrel{|\cdot|-\alpha\textnormal{ is additive}}{=} |\beta_0|+|\beta_1|+\dots+|\beta_l|-l\alpha
\stackrel{\eqref{eq_supp1_lem_homProp}}{\geq} 2+\kk\alpha+|\beta_1|+\dots+|\beta_l|-l\alpha \\
&\stackrel{|\cdot|\geq\alpha}{\geq} 2+\kk\alpha+|\beta_1|-\alpha
= a+|\beta_1|
>|\beta_1| \, ,
\end{align}
with the properties of the homogeneity $ |\cdot | $ being consequences of
Lemma~\ref{lem_hom_props}.

\textbf{Step 2 \textnormal{(continuity in the base point)}.}
For $F_{xz}$ defined in \eqref{eq_germ} we prove for all $ \varepsilon \leqslant r $ that
\begin{equation}\label{eq_supp6_singRec}
	\begin{aligned}
		&\Big(\int_{B_r(x)}dz\, \EE^\frac{\p}{q}\big| \big(F_{xz} - F_{x\,z+y}\big)_{\beta t}(z+y)\big|^q\Big)^\frac{1}{\p} \\
&\lesssim 
\sum_{\gamma\prec\beta}
\ve^{|\beta|-\langle\beta\rangle}
(\sqrt[4]{t})^{(\kk-1)\alpha+\langle\gamma\rangle-|\gamma|} 
(\sqrt[4]{t}+|y|)^{\alpha+D/\p+\langle\beta\rangle-|\beta|} 
(\sqrt[4]{t}+|y|+r)^{|\beta|-2-\kk\alpha+|\gamma|-\langle\gamma\rangle} \w 
\, .
	\end{aligned}
\end{equation}%
First, we notice that 
\begin{align}
F_{xz} - F_{x\,z+y}
&= (d\Gamma_{z\,z+y} - d\Gamma_{xz}\Gamma_{x\,z+y}) \Pi^-_{z+y} \label{tmp12} \\
&\,+(d\Gamma_{z\,z+y} - d\Gamma_{xz}\Gamma_{x\,z+y})
\sum_{ k \geq \kk} \ve^{(\kk-k)\alpha} T^{<|\cdot|-2}_{z+y} z_k W_{k, \varepsilon}(\Pi_{z+y}) \, . \label{tmp13}
\end{align}
To see this, we use the definition \eqref{eq_piminus} of $\Pi^-_z$ and rewrite the
difference in terms of
\begin{align}
F_{xz} - F_{x\,z+y}
&= d\Gamma_{x\,z+y} \Big( \sum_{k \geq \kk} \ve^{(\kk-k)\alpha} z_k W_{k, \varepsilon}(\Pi_{z+y}) - \sum_{k<\kk}
c^{(k)} W_{k, \varepsilon}(\Pi_{z+y}) +\xi_\ve 1\Big) \\
&\,- d\Gamma_{xz} \Big( \sum_{k \geq \kk} \ve^{(\kk-k)\alpha} z_k W_{k, \varepsilon}
(\Pi_{z}) - \sum_{k<\kk} c^{(k)} W_{k, \varepsilon}(\Pi_{z}) +\xi_\ve 1\Big) \, ,
\end{align}
which by \eqref{eq_recentre_Pi} and \eqref{e_def_gamma} can be further rewritten as (recall that $W_{k,\ve}(\Pi_{z+y})$ is a polynomial in $\Pi_{z+y}$, cf.~\eqref{eq_supp2_singRec}, and that $c\in\RR[\![z_k]\!]$)
\begin{align}
F_{xz} - F_{x\,z+y}
&= d\Gamma_{x\,z+y} \Big( \sum_{ k \geq \kk} \ve^{(\kk-k)\alpha} z_k W_{k, \varepsilon}(\Pi_{z+y}) -
\sum_{k<\kk} c^{(k)} W_{k, \varepsilon} (\Pi_{z+y}) +\xi_\ve 1\Big) \\
&\,- d\Gamma_{xz}\Gamma_{z\,z+y} \Big( \sum_{k \geq \kk} \ve^{(\kk-k)\alpha} z_k W_{k,
\varepsilon}(\Pi_{z+y}) - \sum_{k<\kk} c^{(k)} W_{k, \varepsilon}(\Pi_{z+y}) +\xi_\ve 1\Big) \, .
\end{align}
Since we are only interested in the $\beta$-component of this identity, \eqref{eq_recentre_Pi} is only required for multiindices $\beta'\prec\beta$ by Lemma~\ref{lem_dep_pi}.
Using once more the definition \eqref{eq_piminus} of $\Pi^-_{z+y}$ we obtain the desired
representation in terms of~\eqref{tmp12} and~\eqref{tmp13}.
We proceed estimating the contributions of~\eqref{tmp12} and~\eqref{tmp13} separately in the following Steps~2a and 2b.

\textbf{Step 2a.}
To estimate the $(\int_{B_r(x)}dz\,\EE^\frac{\p}{q}|(\cdot)_{\beta t}(z+y)|^q)^\frac{1}{\p}$-norm of \eqref{tmp12} we expand 
\begin{equation}
\big((d\Gamma_{x\,z+y} - d\Gamma_{xz}\Gamma_{z\,z+y})\Pi^-_{z+y}\big)_\beta
=\sum_{\gamma\prec\beta} (d\Gamma_{x\,z+y} - d\Gamma_{xz}\Gamma_{z\,z+y})_\beta^\gamma
\Pi^-_{z+y\gamma}\,,
\end{equation}
and distinguish two cases. 
When $2+\kk\alpha+D/\p-|\gamma|>|\beta|-\langle\beta\rangle$, we use
\eqref{est_Pi-} and
the first alternative of \eqref{est_dGamma_inc} (after an application of H\"older's
inequality) 
to estimate these summands for $\ve\leq r$ by 
\begin{equation}\label{eq_supp5_singRec}
	\begin{aligned}
	&	(\int_{B_r(x)}dz\,\EE^\frac{\p}{q}|
		(d\Gamma_{x\,z+y} - d\Gamma_{xz}\Gamma_{z\,z+y})_\beta^\gamma
\Pi^-_{z+y\gamma t}
(z+y)|^q)^\frac{1}{\p}\\
	&\lesssim
		\ve^{|\beta|-\langle\beta\rangle-|\gamma|+\langle\gamma\rangle}
|y|^{\alpha+D/\p-\bar n}(|y|+r)^{\langle\beta\rangle-\langle\gamma\rangle+\bar n-\alpha} \w \,
\ve^{|\gamma|-\langle\gamma\rangle} (\sqrt[4]{t})^{\langle\gamma\rangle-2} \\
&\leq \ve^{|\beta|-\langle\beta\rangle}
(\sqrt[4]{t})^{(\kk-1)\alpha+\langle\gamma\rangle-|\gamma|} 
(\sqrt[4]{t})^{|\gamma|-2-(\kk-1)\alpha} \\
&\quad \times
|y|^{2+\kk\alpha+D/\p-|\gamma|+\langle\beta\rangle-|\beta|}
|y|^{|\gamma|-2-(\kk-1)\alpha-\bar n}
(|y|+r)^{|\beta|-\langle\gamma\rangle+\bar n-\alpha} \w \, ,
	\end{aligned}
\end{equation}
where in the last inequality we used $\langle\beta\rangle-|\beta|\leq0$, and recall that $\bar n=\max\{|\n|\colon\gamma(\n)\neq0\}$ is introduced in Proposition~\ref{prop:alg3}.
Since $\gamma$ is not purely polynomial by \eqref{model_poly}, and $\sum_\n\gamma(\n)>0$ by
\eqref{e_def_dgamma},
we deduce from \eqref{eq_supp1_lem_homProp} that $|\gamma|\geq2+(\kk-1)\alpha+\bar n$, 
hence, \eqref{eq_supp5_singRec} is further bounded by 
\begin{equation}
\ve^{|\beta|-\langle\beta\rangle}
(\sqrt[4]{t})^{(\kk-1)\alpha+\langle\gamma\rangle-|\gamma|}
(\sqrt[4]{t}+|y|)^{\alpha+D/\p+\langle\beta\rangle-|\beta|} 
(|y|+r)^{|\beta|-2-\kk\alpha+|\gamma|-\langle\gamma\rangle} \w \, .
\end{equation}
Indeed, this yields the desired bound in \eqref{eq_supp6_singRec} since the last exponent
in the previous display is non-negative, because $|\gamma|-\langle\gamma\rangle\geq0$ 
and $|\beta|\geq2+\kk\alpha$. 
The latter can be argued as follows:
Lemma~\ref{lem_sets} (Items~2 and 3) shows that $\beta=0$ or $\sum_k\beta(k)>0$;
if $\beta=0$ then one can read off \eqref{eq_dgamma_prod_rep} 
and Lemma~\ref{lem_def_gamma} (Items~1 and 2) that 
$(d\Gamma_{x\,z+y}-d\Gamma_{xz}\Gamma_{z\,z+y})_0^\gamma$ is only non vanishing 
if $\gamma=\delta_\n$ is purely polynomial, 
in which case $\Pi^-_{z+y\gamma}$ vanishes;
hence $\sum_k\beta(k)>0$ and \eqref{eq_supp1_lem_homProp} shows that $|\beta|\geq2+\kk\alpha$.

In the case of $2+\kk\alpha+D/\p-|\gamma|\leq|\beta|-\langle\beta\rangle$
we use \eqref{est_Pi-} and the second alternative of \eqref{est_dGamma_inc} 
which yields a bound analogous to \eqref{eq_supp5_singRec},
for $\ve\leq r$, in terms of 
\begin{align}
&\ve^{|\beta|-\langle\beta\rangle-|\gamma|+\langle\gamma\rangle}
(|y|+r)^{\langle\beta\rangle-\langle\gamma\rangle+D/\p} \w \, 
\ve^{|\gamma|-\langle\gamma\rangle}
(\sqrt[4]{t})^{\langle\gamma\rangle-2} \\
&= \ve^{|\beta|-\langle\beta\rangle}
(\sqrt[4]{t})^{(\kk-1)\alpha+\langle\gamma\rangle-|\gamma|+\alpha+D/\p+\langle\beta\rangle-|\beta|}\\
&\quad \times
(\sqrt[4]{t})^{|\beta|-\langle\beta\rangle-D/\p-\alpha+|\gamma|-2-(\kk-1)\alpha}
(|y|+r)^{\langle\beta\rangle-\langle\gamma\rangle+D/\p} \w \, .
\end{align}
Using that $\langle\beta\rangle-\langle\gamma\rangle+D/\p\geq0$ by
Lemma~\ref{lem_dGamma} Item 2, 
and that $\alpha+D/\p+\langle\beta\rangle-|\beta|>0$, 
the previous display is further bounded by 
\begin{equation}
\ve^{|\beta|-\langle\beta\rangle}
(\sqrt[4]{t})^{(\kk-1)\alpha+\langle\gamma\rangle-|\gamma|}
(\sqrt[4]{t}+|y|)^{\alpha+D/\p+\langle\beta\rangle-|\beta|}
(\sqrt[4]{t}+|y|+r)^{|\beta|-2-\kk\alpha+|\gamma|-\langle\gamma\rangle} \w \, .
\end{equation}
Overall, we conclude that \eqref{eq_supp6_singRec} holds for $ \varepsilon \leqslant r $
as desired. 

\textbf{Step 2b.}
To estimate the $(\int_{B_r(x)}dz\,\EE^\frac{\p}{q}|(\cdot)_{\beta t}(z+y)|^q)^\frac{1}{\p}$-norm of \eqref{tmp13} we expand as in Step~2a
\begin{align}
&\big((d\Gamma_{z\,z+y} - d\Gamma_{xz}\Gamma_{x\,z+y})
\sum_{k \geq \kk} \ve^{(\kk-k)\alpha} T^{<|\cdot|-2}_{z+y} z_k W_{k, \varepsilon}(\Pi_{z+y})\big)_\beta \\
&= \sum_{\gamma\prec\beta}
\big(d\Gamma_{z\,z+y} - d\Gamma_{xz}\Gamma_{x\,z+y}\big)_\beta^\gamma
\sum_{k \geq \kk} \ve^{(\kk-k)\alpha}
\sum_{|\n|<|\gamma|-2}
\tfrac{1}{\n!}\partial^\n\big(z_k W_{k, \varepsilon}(\Pi_{z+y})\big)_\gamma(z+y) (\cdot-z-y)^\n \, . \label{tmp14}
\end{align}
Preparing for the estimates later, we first prove that the sum over $\n$ in the previous
display can be restricted to $\bar n\leq|\n|$, 
where $\bar n=\max\{|\m|\colon\gamma(\m)\neq0\}$ as in the previous step.
To see this, we expand
\begin{align}
&\partial^\n\big(z_k W_{k, \varepsilon}(\Pi_{z+y})\big)_\gamma(z+y) \\
&= \sum
\tbinom{\n}{\n_0,\dots,\n_l}
\partial^{\n_0} W_{k-l, \varepsilon}(\Pi_{z+y\,0})(z+y) \partial^{\n_1}\Pi_{z+y\gamma_1}(z+y)\cdots\partial^{\n_l}\Pi_{z+y\gamma_l}(z+y) \, ,
\end{align}
where the sum is taken over $l\geq0$, $\gamma_1,\dots,\gamma_l\neq0$ such that $\delta_k+\gamma_1+\cdots+\gamma_l=\gamma$, and $\n_0,\dots,\n_l$ such that $\n_0+\dots+\n_l=\n$.
Since $\gamma(\m)\neq0$ for some $\m$ with $|\m|=\bar n$, the same must hold true for at
least one of $\gamma_1,\dots,\gamma_l$, w.l.o.g.~say $\gamma_1$, hence, 
from \eqref{eq_supp1_lem_homProp} and $2+(\kk-1)\alpha-\kappa=a-\kappa\geq0$ 
we read off $\bar n\leq\langle\gamma_1\rangle$.
Furthermore, we notice that
\begin{equation}
	\begin{aligned}
		\partial^{\n_{1}}\Pi_{z+y \gamma_{1} } (z+y) = \lim_{ t \to 0} \partial
		^{\n_{1}}\Pi_{z+y \gamma_{1} t} (z+y)
		=0\,,
	\end{aligned}
\end{equation}
where we used qualitative smoothness
\eqref{est_Pi_taylorremainder} in the first identity. The second equality holds by
\eqref{est_Pi} (which we can upgrade again to include derivatives, see
Remark~\ref{rem_improve_n}) as long as $ |\n_{1}| < \langle \gamma_{1}\rangle $, hence the sum
effectively restricts to $| \n_{1}| \geqslant  \langle \gamma_{1} \rangle$.
On the other hand, we have $| \n_{1}| \leqslant | \n|$. Consequently, $\bar n\leq|\n|$ as desired.  

\textbf{Step 2b(i).}
We estimate the contribution of \eqref{tmp13} in the case $\ve>\sqrt[4]{t}+|y|$.
\begin{itemize}
	\item 
If $\alpha+D/\p-\bar n\geq0$, 
we use \eqref{est_Pi-_bounded_preliminary}, the first alternative of
\eqref{est_dGamma_inc}, and the moment bound \eqref{eq_momentbound_semigroup} to
estimate the summands of \eqref{tmp14} for $\ve\leq r$ by
\begin{equation}\label{tmp17}
	\begin{aligned}
		&	(\int_{B_r(x)}dz\,\EE^\frac{\p}{q}|
\big(d\Gamma_{z\,z+y} - d\Gamma_{xz}\Gamma_{x\,z+y}\big)_\beta^\gamma
\ve^{(\kk-k)\alpha}
\sum_{|\n|<|\gamma|-2}
\tfrac{1}{\n!}\partial^\n\big(z_k W_{k, \varepsilon}(\Pi_{z+y})\big)_\gamma(z+y)
(\cdot^{\n})_{t}(0) |^q)^\frac{1}{\p}\\
	&\lesssim
\ve^{|\beta|-\langle\beta\rangle-|\gamma|+\langle\gamma\rangle}
|y|^{\alpha+D/\p-\bar n}
(|y|+r)^{\langle\beta\rangle-\langle\gamma\rangle+\bar n-\alpha} \w \, 
\sum_{|\n|<|\gamma|-2} \ve^{|\gamma|-2-|\n|} (\sqrt[4]{t})^{|\n|} \, .
	\end{aligned}
\end{equation}
Next, we notice that 
\begin{equation}
	\begin{aligned}
		\ve^{|\gamma|-2-|\n|}&=\ve^{|\gamma|-2-(\kk-1)\alpha-\bar n+|\beta|-\langle\beta\rangle}
\ve^{\bar n-|\n|+(\kk-1)\alpha+\langle\beta\rangle-|\beta|}\\
&\leqslant 
(|y|+r)^{|\gamma|-2-(\kk-1)\alpha-\bar n+|\beta|-\langle\beta\rangle}
(\sqrt[4]{t})^{(\kk-1)\alpha} (\sqrt[4]{t}+|y|)^{\bar
n-|\n|+\langle\beta\rangle-|\beta|}\,.
	\end{aligned}
\end{equation}
Here we used $|\gamma|\geq2+(\kk-1)\alpha+\bar n$ (see Step~2a), 
$|\cdot|-\langle\cdot\rangle\geq0$, and $\ve\leq r\leq |y|+r$ to estimate the first
factor, as well as $\bar n-|\n|\leq0$ (see Step~2b), 
$(\kk-1)\alpha<0$, 
$\langle\cdot\rangle-|\cdot|\leq0$, and $\ve>\sqrt[4]{t}+|y|\geq \sqrt[4]{t}$ to estimate the second factor. 
Moreover, using $|y|^{\alpha+D/\p-\bar n}\leq(\sqrt[4]{t}+|y|)^{\alpha+D/\p-\bar n}$ and 
$(\sqrt[4]{t})^{|\n|}\leq(\sqrt[4]{t}+|y|)^{|\n|}$, we see that \eqref{tmp17} is further bounded by 
\begin{equation}
\ve^{|\beta|-\langle\beta\rangle-|\gamma|+\langle\gamma\rangle}
(\sqrt[4]{t})^{(\kk-1)\alpha}
(\sqrt[4]{t}+|y|)^{\alpha+D/\p+\langle\beta\rangle-|\beta|}
(|y|+r)^{|\beta|-2-\kk\alpha+|\gamma|-\langle\gamma\rangle} \w \, .
\end{equation}
Lastly, we establish the desired bound of \eqref{eq_supp6_singRec} using that $\langle\cdot\rangle-|\cdot|\leq0$, $\ve>\sqrt[4]{t}$, and $|\beta|-2-\kk\alpha\geq0$ (see Step~2a and note that also here purely polynomial $\gamma$ are not contributing). 

\item
Next, we turn to the case $\alpha+D/\p-\bar n<0$.
Here we use \eqref{est_Pi-_bounded_preliminary}, the second alternative of
\eqref{est_dGamma_inc}, and the moment bound \eqref{eq_momentbound_semigroup} to estimate
the summands of \eqref{tmp14} for $\ve\leq r$ by 
\begin{equation}\label{eq_supp8_singRec}
	\begin{aligned}
				&(\int_{B_r(x)}dz\,\EE^\frac{\p}{q}|
\big(d\Gamma_{z\,z+y} - d\Gamma_{xz}\Gamma_{x\,z+y}\big)_\beta^\gamma
\ve^{(\kk-k)\alpha}
\sum_{|\n|<|\gamma|-2}
\tfrac{1}{\n!}\partial^\n\big(z_k W_{k, \varepsilon}(\Pi_{z+y})\big)_\gamma(z+y)
(\cdot^{\n} )_{t}(0) |^q)^\frac{1}{\p}
	\\
& \lesssim \ve^{|\beta|-\langle\beta\rangle-|\gamma|+\langle\gamma\rangle}
(|y|+r)^{\langle\beta\rangle-\langle\gamma\rangle+D/\p} \w \,
\sum_{|\n|<|\gamma|-2}
\ve^{|\gamma|-2-|\n|}
(\sqrt[4]{t})^{|\n|} \\
&=\ve^{|\beta|-\langle\beta\rangle-|\gamma|+\langle\gamma\rangle}
(\sqrt[4]{t})^{(\kk-1)\alpha}
(\sqrt[4]{t}+|y|)^{\alpha+D/\p+\langle\beta\rangle-|\beta|} \\
&\quad\times
\sum_{|\n|<|\gamma|-2}
(\sqrt[4]{t})^{|\n|-(\kk-1)\alpha}
(\sqrt[4]{t}+|y|)^{-\alpha-D/\p-\langle\beta\rangle+|\beta|}
\ve^{|\gamma|-2-|\n|}
(|y|+r)^{\langle\beta\rangle-\langle\gamma\rangle+D/\p} \w \\
&\leq \ve^{|\beta|-\langle\beta\rangle-|\gamma|+\langle\gamma\rangle}
(\sqrt[4]{t})^{(\kk-1)\alpha}
(\sqrt[4]{t}+|y|)^{\alpha+D/\p+\langle\beta\rangle-|\beta|} \\
&\quad\times
\sum_{|\n|<|\gamma|-2}
(\sqrt[4]{t}+|y|)^{|\n|-\kk\alpha-D/\p-\langle\beta\rangle+|\beta|}
\ve^{|\gamma|-2-|\n|}
(|y|+r)^{\langle\beta\rangle-\langle\gamma\rangle+D/\p} \w \, .
	\end{aligned}
\end{equation}%
Since $\bar n-\alpha-D/\p>0$ by assumption,
$\alpha<0$,
and $|\n|-\bar n\geq0$ by Step~2b, we see that
$$|\n|-\kk\alpha-\tfrac{D}{\p}=\bar n -\alpha-\tfrac{D}{\p} - (\kk-1)\alpha+|\n|-\bar n\geq0\,.$$
Furthermore, $|\cdot|-\langle\cdot\rangle\geq0$, 
$|\gamma|-2-|\n|>0$, $\ve\leq r$,
and $\langle\beta\rangle-\langle\gamma\rangle+D/\p\geq0$ (see Lemma~\ref{lem_dGamma}), 
so that \eqref{eq_supp8_singRec} is further bounded by 
\begin{equation}
\ve^{|\beta|-\langle\beta\rangle-|\gamma|+\langle\gamma\rangle}
(\sqrt[4]{t})^{(\kk-1)\alpha}
(\sqrt[4]{t}+|y|)^{\alpha+D/\p+\langle\beta\rangle-|\beta|} 
(\sqrt[4]{t}+|y|+r)^{|\beta|-2-\kk\alpha+|\gamma|-\langle\gamma\rangle} \w \, .
\end{equation}
As before, this is the desired bound as in \eqref{eq_supp6_singRec}.
\end{itemize}

\textbf{Step 2b(ii).}
We estimate the contribution of \eqref{tmp13} in the case $\ve\leq\sqrt[4]{t}+|y|$. 
\begin{itemize}
	\item 
If $\alpha+D/\p-\bar n\geq0$
we use \eqref{est_Pi-_bounded_preliminary}, the first alternative of
\eqref{est_dGamma_inc}, and the moment estimate \eqref{eq_momentbound_semigroup} to
estimate the summands of \eqref{tmp14} for $\ve\leq r$ by 
\begin{align}
					&(\int_{B_r(x)}dz\,\EE^\frac{\p}{q}|
\big(d\Gamma_{z\,z+y} - d\Gamma_{xz}\Gamma_{x\,z+y}\big)_\beta^\gamma
\ve^{(\kk-k)\alpha}
\sum_{|\n|<|\gamma|-2}
\tfrac{1}{\n!}\partial^\n\big(z_k W_{k, \varepsilon}(\Pi_{z+y})\big)_\gamma(z+y)
(\cdot^{\n})_{t} (0) |^q)^\frac{1}{\p}
	\\
&\lesssim \ve^{|\beta|-\langle\beta\rangle-|\gamma|+\langle\gamma\rangle}
|y|^{\alpha+D/\p-\bar n}
(|y|+r)^{\langle\beta\rangle-\langle\gamma\rangle+\bar n-\alpha} \w \, 
\sum_{|\n|<|\gamma|-2} \ve^{|\gamma|-2-|\n|} (\sqrt[4]{t})^{|\n|} \\
&\lesssim \ve^{|\beta|-\langle\beta\rangle}
|y|^{\alpha+D/\p-\bar n}
(\sqrt[4]{t}+|y|)^{\langle\gamma\rangle-2}
(|y|+r)^{\langle\beta\rangle-\langle\gamma\rangle+\bar n-\alpha} \w \, , \label{tmp18}
\end{align}
where we directly used that $|\n|<|\gamma|-2$ implies $|\n|<\langle\gamma\rangle-2$. 
Moreover, using that $|\cdot|-\langle\cdot\rangle\geq0$, $\alpha+D/\p-\bar n\geq0$ by assumption, and $\alpha<0$, 
\eqref{tmp18} is further bounded by 
\begin{align}\label{eq_supp9_singRec}
\ve^{|\beta|-\langle\beta\rangle}
(\sqrt[4]{t})^{(\kk-1)\alpha+\langle\gamma\rangle-|\gamma|}
(\sqrt[4]{t}+|y|)^{|\gamma|-2-(\kk-1)\alpha-\bar n+\alpha+D/\p}
(|y|+r)^{\langle\beta\rangle-\langle\gamma\rangle+\bar n-\alpha} \w \, .
\end{align}
\begin{itemize}
	\item[($\bullet$)] If $\langle\beta\rangle-\langle\gamma\rangle+\bar n-\alpha\geq0$, then
\eqref{eq_supp9_singRec} is bounded as desired,
since also $|\gamma|-2-(\kk-1)\alpha-\bar n\geq0$ (see Step~2a) and $|\cdot|-\langle\cdot\rangle\geq0$.

\item[($\bullet$)] If $\langle\beta\rangle-\langle\gamma\rangle+\bar n-\alpha<0$, we bound \eqref{tmp18} instead by 
\begin{equation}
\ve^{|\beta|-\langle\beta\rangle}
|y|^{\langle\beta\rangle-\langle\gamma\rangle+D/\p}
(\sqrt[4]{t}+|y|)^{\langle\gamma\rangle-2} \w \, .
\end{equation}
Then, using that $\langle\beta\rangle-\langle\gamma\rangle+D/\p\geq0$, this is further bounded by 
\begin{align}
&\ve^{|\beta|-\langle\beta\rangle}
(\sqrt[4]{t}+|y|)^{\langle\beta\rangle-2+D/\p} \w \\
&= \ve^{|\beta|-\langle\beta\rangle}
(\sqrt[4]{t})^{(\kk-1)\alpha+\langle\gamma\rangle-|\gamma|}
(\sqrt[4]{t}+|y|)^{\alpha+D/\p+\langle\beta\rangle-|\beta|}
(\sqrt[4]{t})^{-(\kk-1)\alpha-\langle\gamma\rangle+|\gamma|}
(\sqrt[4]{t}+|y|)^{|\beta|-\alpha-2} \w \, .
\end{align}
Now, because $\alpha<0$ and $|\cdot|-\langle\cdot\rangle\geq0$, we can first bound 
$(\sqrt[4]{t})^{-(\kk-1)\alpha-\langle\gamma\rangle+|\gamma|}
\leq (\sqrt[4]{t}+|y|)^{-(\kk-1)\alpha-\langle\gamma\rangle+|\gamma|}$, 
and using $|\beta|\geq2+\kk\alpha$ (see Step~2a) and once more $|\cdot|-\langle\cdot\rangle\geq0$ we can then bound
$(\sqrt[4]{t}+|y|)^{|\beta|-2-\kk\alpha-\langle\gamma\rangle+|\gamma|}
\leq (\sqrt[4]{t}+|y|+r)^{|\beta|-2-\kk\alpha-\langle\gamma\rangle+|\gamma|}$, 
yielding the desired bound as in \eqref{eq_supp6_singRec}. 
\end{itemize}

\item
We turn to the case $\alpha+D/\p-\bar n<0$.
Here we use \eqref{est_Pi-_bounded_preliminary}, the second alternative of
\eqref{est_dGamma_inc}, and the moment estimate \eqref{eq_momentbound_semigroup} to
estimate the summands of \eqref{tmp14} for $\ve\leq r$ by 
\begin{align}
&\ve^{|\beta|-\langle\beta\rangle-|\gamma|+\langle\gamma\rangle}
(|y|+r)^{\langle\beta\rangle-\langle\gamma\rangle+D/\p} \w \, 
\sum_{|\n|<|\gamma|-2} \ve^{|\gamma|-2-|\n|} (\sqrt[4]{t})^{|\n|} \\
&\lesssim \ve^{|\beta|-\langle\beta\rangle}
(|y|+r)^{\langle\beta\rangle-\langle\gamma\rangle+D/\p} \w \, 
(\sqrt[4]{t}+|y|)^{\langle\gamma\rangle-2} \\
&= \ve^{|\beta|-\langle\beta\rangle}
(\sqrt[4]{t}+|y|)^{(\kk-1)\alpha+\langle\gamma\rangle-|\gamma|}
(\sqrt[4]{t}+|y|)^{\alpha+D/\p+\langle\beta\rangle-|\beta|}\\
& \quad \times
(\sqrt[4]{t}+|y|)^{|\gamma|-2-\kk\alpha-D/\p-\langle\beta\rangle+|\beta|}
(|y|+r)^{\langle\beta\rangle-\langle\gamma\rangle+D/\p} \w \, ,
\end{align}
where in the inequality we used $\langle\gamma\rangle-2-|\n|>0$ and $ \ve\leq\sqrt[4]{t}+|y| $.
By $\alpha<0$ and $|\cdot|-\langle\cdot\rangle\geq0$ we estimate
$(\sqrt[4]{t}+|y|)^{(\kk-1)\alpha+\langle\gamma\rangle-|\gamma|}
\leq (\sqrt[4]{t})^{(\kk-1)\alpha+\langle\gamma\rangle-|\gamma|}$.
Furthermore, by $|\gamma|\geq2+(\kk-1)\alpha+\bar n$ (see Step 2a) and hence 
$|\gamma|-2-\kk\alpha-D/\p\geq \bar n-\alpha-D/\p\geq0$, 
by $|\cdot|-\langle\cdot\rangle\geq0$, and by $\langle\beta\rangle-\langle\gamma\rangle+D/\p\geq0$ (see Lemma~\ref{lem_dGamma} (Item 2)), 
we have
\begin{equation}
(\sqrt[4]{t}+|y|)^{|\gamma|-2-\kk\alpha-D/\p-\langle\beta\rangle+|\beta|}
(|y|+r)^{\langle\beta\rangle-\langle\gamma\rangle+D/\p} 
\lesssim (\sqrt[4]{t}+|y|+r)^{|\beta|-2-\kk\alpha+|\gamma|-\langle\gamma\rangle} \, ,
\end{equation}
and the desired estimate follows. 
\end{itemize}

\textbf{Step 3 \textnormal{(reconstruction)}.}
We prove for $\ve\leq r$
\begin{equation}\label{eq_supp13_singRec}
	\begin{aligned}
&\Big(\int_{B_r(x)}dz \, 
\EE^\frac{\p}{q}\big|F_{x\,z+y\,\beta t}(z+y)\big|^q\Big)^\frac{1}{\p} \\
&\lesssim \sum_{\gamma\prec\beta}
\ve^{|\beta|-\langle\beta\rangle}
(\sqrt[4]{t})^{\kk\alpha+D/\p+\langle\gamma\rangle-|\gamma|+\langle\beta\rangle-|\beta|} 
(\sqrt[4]{t}+|y|+r)^{|\beta|-2-\kk\alpha+|\gamma|-\langle\gamma\rangle} \w \, .
	\end{aligned}
\end{equation}
This is a consequence of Steps~1 and 2 and reconstruction.
For completeness, we recall the argument from
\cite[Section~3.2]{BOT} in the 
$(\int_{B_r(x)}dz\,\EE^\frac{\p}{q}|(\cdot)_{\beta t}|^q)^\frac{1}{\p}$-setting.
It is based on the semigroup property, continuity \eqref{est_deltaPi-_taylorremainder}, and vanishing at the base-point established in
Step~1.
First, we introduce the representation through a telescoping sum:
\begin{equation}
	\begin{aligned}
	F_{x\,z+y\,t}(z+y) 
&= \sum_{s<t} \big( 
(F_{x\,\bullet\,2s}(\cdot))_{t-2s}(z+y)
- (F_{x\,\bullet\,s}(\cdot))_{t-s}(z+y) \big) \\
&= \sum_{s<t} 
\int d\bar z\, \Psi_{t-2s}(z+y-\bar z)
\int d\bar y\, \Psi_s(\bar z-\bar y) 
\big(F_{x\bar z s}(\bar y) - F_{x\bar y s}(\bar y)\big) \, ,
	\end{aligned}
\end{equation}
where the sum is restricted to dyadic fractions $s$ of $t$, i.e.~there exists a $k
\in \NN $ such that $ s= 2^{-k} t$.
Changing variables $\bar z\mapsto z+y-\bar z$ and then $\bar y\mapsto z+y-\bar z-\bar y$ shows 
\begin{equation}\label{eq_supp11_singRec}
F_{x\,z+y\,t}(z+y) 
= \sum_{s<t} 
\int d\bar z\, \Psi_{t-2s}(\bar z)
\int d\bar y\, \Psi_s(\bar y) 
\big(F_{x\,z+y-\bar z\, s}(z+y-\bar z-\bar y) - F_{x\,z+y-\bar z-\bar y\, s}(z+y-\bar
z-\bar y)\big) \, .
\end{equation}
Next, we use the continuity in the base-point \eqref{eq_supp6_singRec} (established in
Step~2) in combination with the identity \eqref{eq_supp11_singRec} and Minkowski's inequality (we change
the variables  according to $z\mapsto z-y+\bar z$ and also use that $ {B_r(x+y-\bar z)
	\subset B_{r+|y|+|\bar z|}(x)} $), to derive the bound
	\begin{equation}\label{eq_supp12_singRec}
\begin{aligned}
	&\bigg(\int_{B_r(x)}dz\,\EE^\frac{\p}{q}|F_{x\,z+y\, \beta t}(z+y)
			|^q\bigg)^\frac{1}{\p}
	\\
&\lesssim \sum_{s<t} \int d\bar z\,|\Psi_{t-2s}(\bar z)|
\int d\bar y\,|\Psi_s(\bar y)|
\sum_{\gamma\prec\beta}
\ve^{|\beta|-\langle\beta\rangle}
(\sqrt[4]{s})^{(\kk-1)\alpha+\langle\gamma\rangle-|\gamma|} 
(\sqrt[4]{s}+|\bar y|)^{\alpha+D/\p+\langle\beta\rangle-|\beta|} \\
& \qquad \qquad \times
(\sqrt[4]{s}+|\bar y|+r+|y|+|\bar z|)^{|\beta|-2-\kk\alpha+|\gamma|-\langle\gamma\rangle} \w \\
&\lesssim \sum_{\gamma\prec\beta}
\ve^{|\beta|-\langle\beta\rangle}
\sum_{s<t} 
(\sqrt[4]{s})^{\kk\alpha+D/\p+\langle\gamma\rangle-|\gamma|+\langle\beta\rangle-|\beta|} 
(\sqrt[4]{t}+|y|+r)^{|\beta|-2-\kk\alpha+|\gamma|-\langle\gamma\rangle} \w \, .
\end{aligned}
\end{equation}
Since the exponent
$\kk\alpha+D/\p+\langle\gamma\rangle-|\gamma|+\langle\beta\rangle-|\beta|$ on
$\sqrt[4]{s}$ is positive, 
the geometric series converges and
\eqref{eq_supp12_singRec} is further bounded by 
\begin{equation}
\sum_{\gamma\prec\beta}
\ve^{|\beta|-\langle\beta\rangle}
(\sqrt[4]{t})^{\kk\alpha+D/\p+\langle\gamma\rangle-|\gamma|+\langle\beta\rangle-|\beta|} 
(\sqrt[4]{t}+|y|+r)^{|\beta|-2-\kk\alpha+|\gamma|-\langle\gamma\rangle} \w \, .
\end{equation}
Thus, \eqref{eq_supp13_singRec} holds.

\textbf{Step 4.}
As an immediate consequence, we prove for $\ve\leq r$ 
\begin{equation}\label{eq_supp16_singRec}
\begin{aligned}
&\Big(\int_{B_r(x)}dz \, 
\EE^\frac{p}{q}\big|F_{xz\beta t}(z+y)\big|^q\Big)^\frac{1}{p} \\
&\lesssim \sum_{\gamma\prec\beta}
\ve^{|\beta|-\langle\beta\rangle}
(\sqrt[4]{t})^{(\kk-1)\alpha+\langle\gamma\rangle-|\gamma|}
(\sqrt[4]{t}+|y|)^{\alpha+D/\p+\langle\beta\rangle-|\beta|} 
(\sqrt[4]{t}+|y|+r)^{|\beta|-2-\kk\alpha+|\gamma|-\langle\gamma\rangle} \w \, .
\end{aligned}
\end{equation}
Indeed, writing
\begin{equation}
F_{xz\beta t}(z+y)
= (F_{xz}-F_{x\,z+y})_{\beta t}(z+y) + F_{x\,z+y\,\beta t}(z+y)\,,
\end{equation}
this is an immediate consequence of the triangle inequality and \eqref{eq_supp6_singRec} and
\eqref{eq_supp13_singRec} from Steps~2 and~3.

\textbf{Step 5.}
We prove
\begin{equation}\label{eq_supp14_singRec}
\Big(\int_{B_r(x)}dz \, 
\EE^\frac{\p}{q}\big|(\delta\zeta_\ve)_t(z+y)\big|^q\Big)^\frac{1}{\p}
\lesssim (\sqrt[4]{t})^{\alpha-2+D/\p} \w \, .
\end{equation}
Since we always consider $\p>2$ and $q<2$ (see the beginning of
Section~\ref{sec:singular}), we have by Minkowski's integral inequality
\begin{equation}
\Big(\int_{B_r(x)}dz \, 
\EE^\frac{\p}{q}\big|(\delta\zeta_\ve)_t(z+y)\big|^q\Big)^\frac{1}{\p}
\leq \EE^\frac{1}{q}\Big| \int_{B_r(x)}dz\, |(\delta\zeta_\ve)_t(z+y)|^{\p} \Big|^\frac{q}{\p}
\leq \EE^\frac{1}{q}\|\delta\zeta_\ve * \Psi_t\|_{L^{\p}(\RR^{1+d})}^q \, .
\end{equation}
Applying Young's convolution inequality with $1+1/\p=1/p_1+1/p_2$ and Sobolev's inequality (see \cite[Theorem~1.38]{BCD11}) for $p_2\geq2$ yields
\begin{equation}
\|\delta\zeta_\ve * \Psi_t\|_{L^{\p}(\RR^{1+d})}
\leq \|\Psi_t\|_{L^{p_1}} \|\delta\zeta_\ve\|_{L^{p_2}}
\leq \|\Psi_t\|_{L^{p_1}} \|\delta\zeta_\ve\|_{\dot H^{D/2-D/p_2}} \, .
\end{equation}
We choose 
\begin{equation}
p_1=\frac{D\p}{D\p+\p(\alpha-2+D/\p)}
\quad\textnormal{and}\quad
p_2=\frac{D}{2-\alpha} \,.
\end{equation}
Notice that $p_1>1$ because $\alpha-2+D/\p<0$ (see \eqref{p_large2}), 
and $p_2>2$ since $\alpha-2+D/2=s>0$ (by definition of $ \alpha$, see definition below~\eqref{rescaled_noise}).
Hence, 
\begin{equation}
	\begin{aligned}
		\EE^\frac{1}{q}\|\delta\zeta_\ve * \Psi_t\|_{L^{\p}(\RR^{1+d})}^q
		\lesssim  
		 (\sqrt[4]{t})^{D/p_1-D}
		 \EE^{ \frac{1}{q}} \|\delta\zeta_\ve\|_{\dot H^{D/2-D/p_2}}^{q}
		 =(\sqrt[4]{t})^{\alpha-2+D/\p}
		 \EE^{ \frac{1}{q}} \|\delta\zeta_\ve\|_{\dot H^{D/2-D/p_2}}^{q}\,,
	\end{aligned}
\end{equation}
where we used 
 $\|\Psi_t\|_{L^{p_1}} = (\sqrt[4]{t})^{D/p_1-D}\|\Psi_{t=1}\|_{L^{p_1}}$
(see~\eqref{eq_Psi_scale}) in the inequality,
and $D/p_1-D=\alpha-2+D/\p$ thereafter.
Since $D/2-D/p_2=s$ and $\|\delta\zeta_\ve\|_{\dot H^s} = \|\delta\zeta\|_{\dot H^s}$ we conclude \eqref{eq_supp14_singRec} by noting that
$\EE^\frac{1}{q}|\cdot|^q\leq\EE^\frac{1}{p^*}|\cdot|^{p^*}$ for $q<p^*$,
and recalling the definition of $\w$ in \eqref{eq:weight}.

\textbf{Step 6.}
We prove for $\ve\leq r$ 
\begin{equation}\label{eq_supp15_singRec}
\begin{aligned}
&\Big(\int_{B_r(x)}dz \, 
\EE^\frac{\p}{q}\big|\big(
d\Gamma_{xz}\sum_k \ve^{(\kk-k)\alpha} T_z^{\leq\langle\cdot\rangle-2}\big(z_k W_{k,
\varepsilon}(\Pi_z)\big) \big)_{\beta t}(z+y)\big|^q\Big)^\frac{1}{\p} \\
&\lesssim \ve^{|\beta|-\langle\beta\rangle} 
(\sqrt[4]{t})^{\alpha-2}
(\sqrt[4]{t}+|y|)^{D/\p} 
(\sqrt[4]{t}+|y|+r)^{\langle\beta\rangle-\alpha} \w \, .
\end{aligned}
\end{equation}
The argument is similar to the estimates of \eqref{tmp13} performed in Step~2b, we just
replace \eqref{est_dGamma_inc} by \eqref{est_dGamma}, which yields for $\ve\leq r$ that
the left-hand side of \eqref{eq_supp15_singRec} is bounded by
\begin{equation}
\sum_{\gamma\prec\beta}
\ve^{|\beta|-\langle\beta\rangle-|\gamma|+\langle\gamma\rangle} r^{\langle\beta\rangle-\langle\gamma\rangle+D/\p} \w \, 
\sum_{|\n|<|\gamma|-2} \ve^{|\gamma|-2-|\n|} (\sqrt[4]{t}+|y|)^{|\n|} \, .
\end{equation}
Using $\langle\beta\rangle-\langle\gamma\rangle+D/\p>0$, $\ve\leq r$, and that $|\n|<|\gamma|-2$ implies $|\n|<\langle\gamma\rangle-2$, this is bounded by 
\begin{align}
&\ve^{|\beta|-\langle\beta\rangle} 
(\sqrt[4]{t}+|y|+r)^{\langle\beta\rangle-2+D/\p} \w \\
&= \ve^{|\beta|-\langle\beta\rangle} 
(\sqrt[4]{t})^{\alpha-2+D/\p} (\sqrt[4]{t})^{-(\alpha-2+D/\p)} 
(\sqrt[4]{t}+|y|+r)^{\langle\beta\rangle-2+D/\p} \w \, , 
\end{align}
which in turn by $\alpha-2+D/\p<0$, see \eqref{p_large2}, is bounded by 
\begin{equation}
\ve^{|\beta|-\langle\beta\rangle} 
(\sqrt[4]{t})^{\alpha-2+D/\p} 
(\sqrt[4]{t}+|y|+r)^{\langle\beta\rangle-\alpha} \w \, .
\end{equation}
Hence, \eqref{eq_supp15_singRec} holds.

\textbf{Step 7 \textnormal{(conclusion)}.}
Using the representation \eqref{eq_germ} 
$$\delta\Pi^-_x-d\Gamma_{xz}\Pi^-_z 
= F_{xz} 
+ \delta\zeta_\ve \, 1
+ d\Gamma_{xz}\sum_{k \geq \kk}\ve^{(\kk-k)\alpha} T_z^{<|\cdot|-2}(z_kW_{k, \varepsilon}(\Pi_z))\,,$$
we obtain \eqref{est_deltaPi-_inc} 
by combining \eqref{eq_supp16_singRec}, \eqref{eq_supp14_singRec}, and \eqref{eq_supp15_singRec} from Steps~4, 5, and 6, respectively, and noting that 
\begin{align}
&(\sqrt[4]{t})^{(\kk-1)\alpha+\langle\gamma\rangle-|\gamma|}
(\sqrt[4]{t}+|y|)^{\alpha+D/\p+\langle\beta\rangle-|\beta|} 
(\sqrt[4]{t}+|y|+r)^{|\beta|-2-\kk\alpha+|\gamma|-\langle\gamma\rangle} \\
&= (\sqrt[4]{t})^{\alpha-2} (\sqrt[4]{t})^{-\alpha+|\beta|-\langle\beta\rangle} 
(\sqrt[4]{t}+|y|)^{\alpha+D/\p+\langle\beta\rangle-|\beta|} \\
&\quad
(\sqrt[4]{t})^{2+(\kk-1)\alpha+\langle\gamma\rangle-|\gamma|+\langle\beta\rangle-|\beta|}
(\sqrt[4]{t}+|y|+r)^{|\beta|-2-\kk\alpha+|\gamma|-\langle\gamma\rangle} \\
&\leq (\sqrt[4]{t})^{\alpha-2} 
(\sqrt[4]{t}+|y|)^{D/\p} 
(\sqrt[4]{t}+|y|+r)^{\langle\beta\rangle-\alpha} \, ,
\end{align}
where in the inequality we used $\alpha<0$, $|\cdot|-\langle\cdot\rangle\geq0$, and that
$2+(\kk-1)\alpha+\langle\gamma\rangle-|\gamma|+\langle\beta\rangle-|\beta|>0$. 
\end{proof}

Having established the reconstruction bound in Proposition~\ref{prop_recIII}, we can retrieve
a moment bound on the Malliavin derivative, which is the main input of the spectral gap
inequality.

\begin{proposition}[Sobolev, Step 4.d)]\label{prop_sobolev}
Let $[\beta]\geq0$ and assume that $\eqref{Pi-_recentre}_\beta$, $\eqref{est_deltaPi-_bounded}_\beta$ and $\eqref{est_deltaPi-_inc}_\beta$ hold,
that $\eqref{est_Gamma}_\beta^\gamma$, $\eqref{est_deltaGamma}_\beta^\gamma$, and $\eqref{est_dGamma}_\beta^\gamma$ hold for all $\gamma$ populated and not purely polynomial, 
and that $\eqref{est_Pi-}_{\beta'}$ and $\eqref{est_deltaPi-}_{\beta'}$ hold for all populated $\beta'\prec\beta$.
Then 
\begin{equation}\label{est_deltaPi-}
\EE^\frac{1}{q} \left|
\delta\Pi_{x\beta t}^-(y) \right|^q
\lesssim \ve^{|\beta|-\langle\beta\rangle}
\big(\sqrt[4]{t}\big)^{\alpha-2} (\sqrt[4]{t}+|x-y|)^{\langle\beta\rangle-\alpha} \w \, .
\end{equation}
\end{proposition}

\begin{proof}
We split the proof into two cases.

\textbf{Step 1.} Assume $\ve\geq \sqrt[4]{t}+|x-y|$.
In this case we appeal to \eqref{est_deltaPi-_bounded} and obtain 
\begin{align}
\EE^\frac{1}{q} \left|
\delta\Pi_{x\beta t}^-(y) \right|^q
&\leq \int dz\, |\Psi_t(y-z)| 
\EE^\frac{1}{q} \left|\delta\Pi_{x\beta}^-(z) \right|^q \\
&\lesssim \int dz\, |\Psi_t(y-z)| 
\ve^{\alpha-2}
( \ve+ | x-z|)^{| \beta| - \alpha} \w \\
&\lesssim \ve^{\alpha-2} (\ve+\sqrt[4]{t}+|x-y|)^{|\beta|-\alpha} \w \,,
\end{align}
where in the last inequality we used the moment bound \eqref{eq_momentbound_semigroup}.
Since $|\cdot|-\alpha\geq0$, 
this is bounded by $\ve^{|\beta|-2}$, which in turn  is bounded by
$\ve^{|\beta|-\langle\beta\rangle} (\sqrt[4]{t}+|x-y|)^{\langle\beta\rangle-2}$ using
$\langle\beta\rangle\leq|\beta|<2$.
Hence, the desired bound of \eqref{est_deltaPi-} holds in this case. 

\textbf{Step 2.} Assume $\ve<\sqrt[4]{t}+|x-y|$.
First, we apply Lemma~\ref{lem:sobolev} with $u=\delta\Pi^-_{x\beta t}$ which yields for $k>D/\p$
\begin{equation}\label{eq_supp2_sobolev}
	\begin{aligned}
\EE^\frac{1}{q}|\delta\Pi^-_{x\beta t}(y)|^q
&\lesssim \sum_{|\n|\leq 2k}
r^{|\n|-D/\p} \Big( \int_{B_r(0)} dz\, \EE^\frac{\p}{q}\big|\partial^\n
\delta\Pi^-_{x\beta t}(y+z)\big|^q\Big)^\frac{1}{\p} \\
&
=
\sum_{|\n|\leq 2k}
r^{|\n|-D/\p} \Big( \int_{B_r(x)} dz\, \EE^\frac{\p}{q}\big|\partial^\n
\delta\Pi^-_{x\beta t}(y-x+z)\big|^q\Big)^\frac{1}{\p}\, .
	\end{aligned}
\end{equation}
Next, using the bound 
(with $y$ replaced by $y-x$)
\begin{equation}\label{eq_supp1_sobolev}
\Big( \int_{B_r(x)} dz\, \EE^\frac{\p}{q}\big| \delta\Pi^-_{x\beta t}(y+z) \big|^q\Big)^\frac{1}{\p} 
\lesssim
\ve^{|\beta|-\langle\beta\rangle}
(\sqrt[4]{t})^{\alpha-2} 
(\sqrt[4]{t}+|y|+r)^{\langle\beta\rangle-\alpha+D/\p} \w 
\quad\textnormal{for }\ve\leq r \,,
\end{equation}
which we show at the end of the proof, \eqref{eq_supp2_sobolev} (improved to include $
\partial^{\n}$, see Remark~\ref{rem_improve_n}) implies for $\ve\leq r$
\begin{equation}
\EE^\frac{1}{q}|\delta\Pi^-_{x\beta t}(y)|^q
\lesssim \sum_{|\n|\leq 2k}
r^{|\n|-D/\p} 
\ve^{|\beta|-\langle\beta\rangle}
(\sqrt[4]{t})^{\alpha-2-|\n|} 
(\sqrt[4]{t}+|x-y|+r)^{\langle\beta\rangle-\alpha+D/\p} \w \, .
\end{equation}
Now, choosing $r=\sqrt[4]{t}+|x-y|>\ve$ yields
\begin{equation}
\EE^\frac{1}{q}|\delta\Pi^-_{x\beta t}(y)|^q
\lesssim \sum_{|\n|\leq 2k}
\ve^{|\beta|-\langle\beta\rangle}
(\sqrt[4]{t})^{\alpha-2-|\n|} 
(\sqrt[4]{t}+|x-y|)^{\langle\beta\rangle-\alpha+|\n|} \w \, .
\end{equation}
For $y=x$ this estimate collapses to 
\begin{equation}\label{eq_supp3_sobolev}
\EE^\frac{1}{q}|\delta\Pi^-_{y\beta t}(y)|^q
\lesssim 
\ve^{|\beta|-\langle\beta\rangle}
(\sqrt[4]{t})^{\langle\beta\rangle-2} \w \, .
\end{equation}
Moreover, because 
$\Pi^-_{x\beta} = (\Gamma_{xy}\Pi^-_y)_\beta$ (see \eqref{Pi-_recentre} since
$\langle\beta\rangle-2\leq|\beta|-2<0$), we have
\begin{equation}
	\begin{aligned}
		\delta\Pi^-_{x\beta} =
(\delta\Gamma_{xy}\Pi^-_y)_\beta+(\Gamma_{xy}\delta\Pi^-_y)_\beta\,.
	\end{aligned}
\end{equation}
Consequently, we can post-process \eqref{eq_supp3_sobolev} using additionally
$\eqref{est_deltaGamma}_\beta^{\gamma} $, $\eqref{est_Pi-}_{\prec\beta}$,
$\eqref{est_Gamma}_\beta^{\gamma}$, and $\eqref{est_deltaPi-}_{\prec\beta}$, which yields
\begin{align}
\EE^\frac{1}{q}|\delta\Pi^-_{x\beta t}(y)|^q
&\lesssim 
\sum_{\gamma\prec\beta}
\ve^{|\beta|-\langle\beta\rangle-|\gamma|+\langle\gamma\rangle}
|x-y|^{\langle\beta\rangle-\langle\gamma\rangle} \w \, 
\ve^{|\gamma|-\langle\gamma\rangle}
(\sqrt[4]{t})^{\langle\gamma\rangle-2} \\
&\,+ \sum_{\gamma\preceq\beta}
\ve^{|\beta|-\langle\beta\rangle-|\gamma|+\langle\gamma\rangle}
|x-y|^{\langle\beta\rangle-\langle\gamma\rangle} 
\ve^{|\gamma|-\langle\gamma\rangle}
(\sqrt[4]{t})^{\langle\gamma\rangle-2} \w \, .
\end{align}
Since $\langle\beta\rangle-\langle\gamma\rangle\geq0$ and
$\langle\cdot\rangle-\alpha\geq0$ the desired bound of \eqref{est_deltaPi-} follows.

It is only left to verify \eqref{eq_supp1_sobolev}. Indeed, 
writing $\delta\Pi^-_x = \delta\Pi^-_x-d\Gamma_{xz}\Pi^-_z + d\Gamma_{xz}\Pi^-_z$ and
using the triangle inequality together with \eqref{est_deltaPi-_inc}, \eqref{est_dGamma},
and \eqref{est_Pi-}, we obtain for $ \varepsilon \leqslant r $
\begin{align}
&\Big( \int_{B_r(x)} dz\, \EE^\frac{\p}{q}\big| \delta\Pi^-_{x\beta t}(y+z) \big|^q\Big)^\frac{1}{\p} \\
&\lesssim \ve^{|\beta|-\langle\beta\rangle}
(\sqrt[4]{t})^{\alpha-2} 
(\sqrt[4]{t}+|y|)^{D/\p} 
(\sqrt[4]{t}+|y|+r)^{\langle\beta\rangle-\alpha} \w \\
&\,+ \sum_{\gamma\prec\beta}
\ve^{|\beta|-\langle\beta\rangle-|\gamma|+\langle\gamma\rangle}
r^{\langle\beta\rangle-\langle\gamma\rangle+D/\p} \w \, 
\ve^{|\gamma|-\langle\gamma\rangle}
(\sqrt[4]{t})^{\alpha-2}
(\sqrt[4]{t}+|y|)^{\langle\gamma\rangle-\alpha}
\, .
\end{align}
Because $\langle\beta\rangle-\langle\gamma\rangle+D/\p>0$ and
$\langle\cdot\rangle-\alpha\geq0$ (see Lemmas~\ref{lem_dGamma} and~\ref{lem_hom_props}), this yields
\eqref{eq_supp1_sobolev}.
\end{proof}

\subsection{Block 5: Spectral gap}\label{sec:block5}

Now, we lift the bounds on $ \EE
\Pi^{-}_{ \beta} $ and $ \EE^{1/q}| \delta \Pi^{-}_{ \beta}|^{q}$ to the desired moment
estimate for
$ \Pi^{-}_{ \beta}$. 

\begin{proposition}[Spectral gap, Step 5]\label{prop_SG}
Let $[\beta]\geq0$ and assume that $\eqref{est_expectation}_\beta$ and $\eqref{est_deltaPi-}_\beta$ hold. 
Then $\eqref{est_Pi-}_\beta$ holds. 
\end{proposition}

The proof of this proposition is an immediate consequence of an application of the spectral gap inequality \eqref{eq:sg} to $F=\Pi^-_{x\beta t}(y)$, 
which is justified since the latter is in the completion of cylindrical functionals with respect to the norm 
$\EE^{1/p}(|F|^p+\|\partial F/\partial\zeta\|_{\dot H^{-s}}^p)$, see \cite[Section~7
and Appendix~A]{LOTT24}. 

\subsection{Block 6: Estimates of \texorpdfstring{$\Pi$ and $\Gamma$}{Pi and Gamma}}\label{sec:block6}

Next, we again construct $ \Pi_{ \beta}$ from $ \Pi_{ \beta}^{-}$ through an integration
argument, and prove that the same covariances hold as for $ \Pi_{\beta}^{-}$. 
Also, we once more choose the polynomials $ \pi^{(\n)}_{ \beta}$ such that $ \Gamma$
 recentres $ \Pi_{ \beta}$, with the desired stochastic estimates.  
Indeed, this entire block is analogous to the regular case:
\begin{itemize}
\item Step 6.a) is identical to Proposition~\ref{prop_integration} and yields the
	integration estimate $\eqref{est_Pi}_\beta$, 
\item Step 6.b) is identical to Proposition~\ref{prop_covariance_Pi} and yields the
	covariances $\eqref{Pi_shift}_\beta$, $\eqref{Pi_parity}_\beta$, and $\eqref{Pi_reflection}_\beta$, 
\item Steps 6.c)-e) are identical to Proposition~\ref{prop_3pt} and yield
	$\eqref{eq_recentre_Pi}_\beta$, $\eqref{est_pin}_\beta$, and
	$\eqref{est_Gamma}_\beta^\gamma$ for all $\gamma$ through the same three-point
	argument. 
\end{itemize}

\subsection{Block 7: Malliavin derivatives of \texorpdfstring{$\Pi$ and $\Gamma$}{Pi and Gamma}}\label{sec:block7}

In the previous block we constructed and estimated $ \Pi_{ \beta}$ and $ \Gamma$, the
main quantities of interest in the model.
To close the induction, we now establish analogous estimates for $ \delta
\Pi$ and $ \delta \Gamma$,  which provide the input for the next induction level.

\begin{proposition}[Integration II, Steps 7.a) and 7.b)]\label{prop_intII}
Let $[\beta]\geq0$ and assume that $\eqref{est_deltaPi-}_\beta$ holds. 
Then 
\begin{equation}\label{est_deltaPi}
\EE^\frac{1}{q} \left|
\delta\Pi_{x\beta}(y) \right|^q
\lesssim \ve^{|\beta|-\langle\beta\rangle}
(\sqrt[4]{t})^\alpha (\sqrt[4]{t}+|x-y|)^{\langle\beta\rangle-\alpha} \w \, .
\end{equation}
\end{proposition}

The proof of Proposition~\ref{prop_intII} is identical to the one of Proposition~\ref{prop_integration}.

\begin{proposition}[Three-point-argument II, Steps 7.c) and 7.d)]\label{prop_3ptII}
Let $[\beta]\geq0$ and assume that $\eqref{est_deltaPi}_{\beta'}$ holds for all populated $\beta'\preceq\beta$, 
that $\eqref{est_Gamma}_\beta^\gamma$ and $\eqref{est_deltaGamma}_\beta^\gamma$ hold for all $\gamma$ populated and not purely polynomial, 
and that $\eqref{eq_recentre_Pi}_\beta$ holds. 
Then $\eqref{est_deltapin}_\beta$ and $\eqref{est_deltaGamma}_\beta^\gamma$ hold for all $\gamma$.
\end{proposition}

Again, the proof of Proposition~\ref{prop_3ptII} is analogous to the one of Proposition~\ref{prop_3pt}.

\subsection{Block 8: Ultraviolet divergent continuity bounds}\label{sec:block8}

The first two steps in this block are identical to the regular case, namely
\begin{itemize}
\item Step 8.a) is identical to Proposition~\ref{prop_2} and yields
	$\eqref{est_Pi-_bounded}_\beta$ and $\eqref{est_Pi-_taylorremainder}_\beta$,
	which yields the required continuity estimate for $ \Pi^{-}_{ \beta}$.
\item Step 8.b) is identical to Proposition~\ref{prop_qualitative_pi} and yields
	$\eqref{est_Pi_bounded}_\beta$ and $\eqref{est_Pi_taylorremainder}_\beta$, which
	provides the continuity estimates for $ \Pi_{\beta}$.
\end{itemize}

It is only left to establish the analogous continuity estimate for $ \delta \Pi_{ \beta}$ as for $
\Pi_{ \beta}$ in
Proposition~\ref{prop_qualitative_pi}. We omit the proof, since it is identical to
the one of Proposition~\ref{prop_qualitative_pi}.

\begin{proposition}[Boundedness and continuity of $\delta\Pi$, Step 8.c)]\label{prop_qualitative_deltaPi}
Let $[\beta]\geq0$ and assume that $\eqref{est_deltaPi-_bounded}_\beta$ 
and $\eqref{est_deltaPi-}_\beta$
hold. 
Then $\eqref{est_deltaPi_bounded}_\beta$ holds as well as 
\begin{equation}\label{est_deltaPi_taylorremainder}
\begin{aligned}
\EE^{\frac{1}{q}} \left|
(\mathrm{id}-T_z^{\leq l}) \partial^\n \delta\Pi_{x \beta}(y) 
\right|^{ q} 
\lesssim 
\ve^{\alpha-|\n|-1-l}
|y-z|^{1+l}
( \ve+ | x-y| + |x-z|)^{| \beta| - \alpha} \w \,.
\end{aligned}
\end{equation}
for all $\n$ and $l\geq0$.
\end{proposition}

\subsection{Block 9: Modelledness}\label{sec:block9}

Finally, we arrive at the last block of arguments for singular multiindices. 
We present the last ingredients, which are continuity estimates and
bounds for $ d\Gamma $ that are necessary for the next induction level.
This requires choosing $ d \pi$, and one last iteration of integration and three-point arguments.

\begin{proposition}[Integration III, Steps 9.a) and 9.b)]\label{prop_intIII}
Let $[\beta]\geq0$ and $|\beta|<2$, and assume that $\eqref{est_deltaPi_taylorremainder}_\beta$ holds, 
and that $\eqref{est_Pi_taylorremainder}_{\beta'}$ holds for all populated $\beta'\prec\beta$. 
Then there exist $d\pi^{(\n)}_{xz\beta}$ for $|\n|<\alpha+D/\p$ such that 
\begin{equation}\label{eq:intIII_qualitative}
\partial^\n\big(\delta\Pi_x-d\Gamma_{xz}\Pi_z\big)_\beta(z) = 0 
\quad\textnormal{for }|\n|<\alpha+D/\p \, .
\end{equation}
Assume furthermore that 
$\eqref{est_Pi-_taylorremainder}_{\beta'}$,
$\eqref{est_Pi-}_{\beta'}$, and
$\eqref{est_Pi}_{\beta'}$ hold for all $\beta'\prec\beta$, that
$\eqref{est_deltaPi-_taylorremainder}_\beta$,
$\eqref{est_deltaPi-_inc}_\beta$, 
$\eqref{est_deltaPi-}_\beta$, and
$\eqref{est_deltaPi}_\beta$ hold, and that 
$\eqref{est_dGamma}_\beta^\gamma$ holds for all $\gamma$ populated and not purely polynomial. 
Then
\begin{align}
&\Big( \int_{B_r(x)} dz\, \EE^\frac{\p}{q}\big| \big( \delta\Pi_x -  d\Gamma_{xz}\Pi_z\big)_\beta(y+z) \big|^q\Big)^\frac{1}{\p} \\
&\lesssim \ve^{|\beta|-\langle\beta\rangle}
(\sqrt[4]{t})^\alpha (\sqrt[4]{t}+|y|)^{D/\p} 
(\sqrt[4]{t}+|y|+r)^{\langle\beta\rangle-\alpha} \w  
\quad\textnormal{for }\ve\leq r \, . \label{est_deltaPi_inc}
\end{align}
\end{proposition}

\begin{proof}
	We choose appropriate $ d \pi_{\beta}^{(\n)}$ in Step~1 and 
	pre-process the input \eqref{est_deltaPi-_inc} in Step~2.
	This allows to perform an integration argument which is carried out in Step~4, which in turn suffices to conclude \eqref{est_deltaPi_inc} in Step~3.

\textbf{Step 1 \textnormal{(proof of \eqref{eq:intIII_qualitative})}.}
The proof is analogous to \cite[Section~5.4]{LOTT24}:
splitting $(d\Gamma_{xz}\Pi_z)_\beta$ into 
$\sum_{[\gamma]\geq0}(d\Gamma_{xz})_\beta^\gamma\Pi_{z\gamma}$
and $\sum_{\m} (d\Gamma_{xz})_\beta^{\delta_\m} \Pi_{z\delta_\m}$, 
and recalling from Lemma~\ref{lem_dGamma} that $(d\Gamma_{xz})_\beta^{\delta_\m} =
d\pi^{(\m)}_{xz\beta}$ and from \eqref{model_poly} that $\Pi_{z\delta_\m}=(\cdot-z)^\m$, 
we see that \eqref{eq:intIII_qualitative} is equivalent to 
\begin{equation}\label{choice_dpin}
d\pi^{(\n)}_{xz\beta} 
= \tfrac{1}{\n!}\partial^\n\big(\delta\Pi_{x\beta}-\sum_{[\gamma]\geq0} (d\Gamma_{xz})_\beta^\gamma \Pi_{z\gamma}\big)(z) 
\quad\textnormal{for }|\n|<\alpha+D/\p \, .
\end{equation}
By Lemma~\ref{lem_triang} (Item~2) the right-hand side only depends on $d\pi^{(\n)}_{xz\beta'}$ for $\beta'\prec\beta$, hence this identity may serve as a definition of
$d\pi^{(\n)}_{xz\beta}$.
We emphasize that the choice of $d\pi^{(\n)}_{xz}$ via \eqref{choice_dpin} satisfies \eqref{eq_pop_dpi}.

\textbf{Step 2.}
First, we prove 
\begin{align}
&\Big( \int_{B_r(x)} dz\, \EE^\frac{\p}{q}\big| 
\big( \delta\Pi^-_{x\beta} -\sum_{|\gamma|<\alpha+D/\p} (d\Gamma_{xz})_\beta^\gamma \Pi^-_{z\gamma} \big)_t(y+z) \big|^q\Big)^\frac{1}{\p} \\
&\lesssim \ve^{|\beta|-\langle\beta\rangle}
(\sqrt[4]{t})^{\alpha-2} (\sqrt[4]{t}+|y|)^{D/\p} 
(\sqrt[4]{t}+|y|+r)^{\langle\beta\rangle-\alpha} \w 
\quad\textnormal{for }\ve\leq r \, , \label{tmp15}
\end{align}
which is a consequence of assumption \eqref{est_deltaPi-_inc} and 
\begin{align}
&\Big( \int_{B_r(x)} dz\, \EE^\frac{\p}{q}\big| 
\sum_{|\gamma|\geq\alpha+D/\p} (d\Gamma_{xz})_\beta^\gamma \Pi^-_{z\gamma t}(y+z) \big|^q\Big)^\frac{1}{\p} \\
&\lesssim \ve^{|\beta|-\langle\beta\rangle}
(\sqrt[4]{t})^{\alpha-2} (\sqrt[4]{t}+|y|)^{D/\p} 
(\sqrt[4]{t}+|y|+r)^{\langle\beta\rangle-\alpha} \w 
\quad\textnormal{for }\ve\leq r \, , \label{eq_supp1_singMod}
\end{align}
which we establish now. 
After applying the triangle inequality and Hölder's inequality in combination
with \eqref{est_dGamma} and \eqref{est_Pi-}, 
we see that the left-hand side of \eqref{eq_supp1_singMod} is bounded by, for $\ve\leq
r$,
\begin{align}
&\sum_{\gamma\prec\beta,\, |\gamma|\geq\alpha+D/\p}
\ve^{|\beta|-\langle\beta\rangle-|\gamma|+\langle\gamma\rangle}
r^{\langle\beta\rangle-\langle\gamma\rangle+D/\p} \w \, 
\ve^{|\gamma|-\langle\gamma\rangle}
(\sqrt[4]{t})^{\alpha-2}
(\sqrt[4]{t}+|y|)^{\langle\gamma\rangle-\alpha} \\
&= \sum_{\gamma\prec\beta,\, |\gamma|\geq\alpha+D/\p}
\ve^{|\beta|-\langle\beta\rangle}
(\sqrt[4]{t})^{\alpha-2}
(\sqrt[4]{t}+|y|)^{D/\p} 
(\sqrt[4]{t}+|y|)^{\langle\gamma\rangle-\alpha-D/\p} 
r^{\langle\beta\rangle-\langle\gamma\rangle+D/\p} \w \, .
\end{align}
Choosing $\p$ rational such that $D/\p$ is not an integer, 
we see that $|\gamma|\geq\alpha+D/\p$ strengthens to $|\gamma|>\alpha+D/\p$, which in turn implies $\langle\gamma\rangle>\alpha+D/\p$. 
Since also $\langle\beta\rangle-\langle\gamma\rangle+D/\p\geq0$ by Lemma~\ref{lem_dGamma}, the above is bounded as desired. 

\textbf{Step 3.}
Next, we claim that \eqref{est_deltaPi_inc} follows from 
\begin{align}
&\Big( \int_{B_r(x)} dz\, \EE^\frac{\p}{q}\big| 
\big( \delta\Pi_{x\beta} -\sum_{|\gamma|<\alpha+D/\p} (d\Gamma_{xz})_\beta^\gamma \Pi_{z\gamma} \big)_t(y+z) \big|^q\Big)^\frac{1}{\p} \\
&\lesssim \ve^{|\beta|-\langle\beta\rangle}
(\sqrt[4]{t})^\alpha (\sqrt[4]{t}+|y|)^{D/\p} 
(\sqrt[4]{t}+|y|+r)^{\langle\beta\rangle-\alpha} \w 
\quad\textnormal{for }\ve\leq r \, . \label{tmp16}
\end{align}
Indeed, this is a consequence of 
\begin{align}
&\Big( \int_{B_r(x)} dz\, \EE^\frac{\p}{q}\big| 
\sum_{|\gamma|\geq\alpha+D/\p} (d\Gamma_{xz})_\beta^\gamma \Pi_{z\gamma t}(y+z) \big|^q\Big)^\frac{1}{\p} \\
&\lesssim \ve^{|\beta|-\langle\beta\rangle}
(\sqrt[4]{t})^\alpha (\sqrt[4]{t}+|y|)^{D/\p} 
(\sqrt[4]{t}+|y|+r)^{\langle\beta\rangle-\alpha} \w 
\quad\textnormal{for }\ve\leq r \, ,
\end{align}
which we establish now. 
The argument is identical to Step~2, just replacing \eqref{est_Pi-} by \eqref{est_Pi}.
Note that purely polynomial $\gamma$ are not present, 
as for $\gamma=\delta_\n$ we have $(d\Gamma_{xz})_\beta^{\delta_\n}=d\pi^{(\n)}_{xz\beta}$ by Lemma~\ref{lem_dGamma}, 
which is only non vanishing for $|\gamma|=|\n|<\alpha+D/\p$.

\textbf{Step 4.}
It is left to argue that \eqref{tmp15} implies \eqref{tmp16}. 
This is based on the representation 
\begin{equation}
\delta\Pi_{x\beta}-\sum_{|\gamma|<\alpha+D/\p} (d\Gamma_{xz})_\beta^\gamma \Pi_{z\gamma}
= \int_0^\infty ds\, (\mathrm{id}-T_z^{<\alpha+D/\p}) ( -\partial_{x_{0}}- \Delta) 
\big(\delta\Pi^-_{x\beta s} - \sum_{|\gamma|<\alpha+D/\p} (d\Gamma_{xz})_\beta^\gamma \Pi^-_{z\gamma s}\big) \, ,
\end{equation}
and is similar to the proof of Proposition~\ref{prop_qualitative_pi} presented above. 
The details can be found in \cite[Section~3.8]{BOT}.
\end{proof}

Having chosen $ d \pi^{(\n)}_{ \beta}$, we can move on to establishing the continuity
estimate for $ d \Gamma $. 

\begin{proposition}[Three-point argument III, Step 9.c)]\label{prop_3ptIII}
Let $[\beta]\geq0$ and assume that $\eqref{est_deltaPi_inc}_\beta$ holds,
that $\eqref{est_dGamma_inc}_\beta^\gamma$ holds for all $\gamma$ populated and not purely polynomial, 
and that $\eqref{est_Pi}_{\beta'}$ and $\eqref{eq_recentre_Pi}_{\beta'}$ hold for all populated $\beta'\prec\beta$. 
Then $\eqref{est_dGamma_pp_inc}_\beta^{\delta_\n}$ holds for all $\n$, 
and thus $\eqref{est_dGamma_inc}_\beta^\gamma$ holds for all $\gamma$. 
\end{proposition}

\begin{proof}
The proof is similar to \cite[Section~2.16]{BOT}, and is based on the identity
\begin{equation}
(d\Gamma_{x\,z+y}-d\Gamma_{xz}\Gamma_{z\,z+y})\Pi_{z+y}
= (\delta\Pi_x - d\Gamma_{xz}\Pi_z)
- (\delta\Pi_x - d\Gamma_{x\,z+y}\Pi_{z+y}) \, ,
\end{equation}
which follows from \eqref{eq_recentre_Pi}. 
Considering the $\beta$-component of this identity, and splitting purely polynomial from other multiindices yields
\begin{align}
\sum_\n (d\Gamma_{x\,z+y}-d\Gamma_{xz}\Gamma_{z\,z+y})_\beta^{\delta_\n} (\cdot-z-y)^\n
&= (\delta\Pi_x - d\Gamma_{xz}\Pi_z)_\beta
- (\delta\Pi_x - d\Gamma_{x\,z+y}\Pi_{z+y})_\beta \\
&- \sum_{[\gamma]\geq0} (d\Gamma_{x\,z+y}-d\Gamma_{xz}\Gamma_{z\,z+y})_\beta^\gamma \Pi_{z+y\,\gamma} \, ,
\end{align}
where we directly used $\Pi_{z+y\,\delta_\n}=(\cdot-z-y)^\n$, see \eqref{model_poly}.
We apply $(\cdot)_t(z+y)$ and the $(\int_{B_r(x)}dz\,\EE^\frac{\p}{q}|\cdot|^q)^\frac{1}{\p}$-norm to this identity, 
and split the three right-hand side terms by the triangle inequality. 
\begin{itemize}
	\item Using \eqref{est_deltaPi_inc}, the first right-hand side term is (for
		$\ve\leq r$) bounded by 
\begin{equation}\label{tmp19}
\ve^{|\beta|-\langle\beta\rangle}
(\sqrt[4]{t})^\alpha (\sqrt[4]{t}+|y|)^{D/\p} 
(\sqrt[4]{t}+|y|+r)^{\langle\beta\rangle-\alpha} \w \, .
\end{equation}
\item Similarly, using \eqref{est_deltaPi_inc} (with $y$ replaced by $0$ and $r$ replaced
	by $r+|y|$), the second right-hand side term is (for $\ve\leq r$) bounded by 
\begin{equation}\label{tmp20}
\ve^{|\beta|-\langle\beta\rangle}
(\sqrt[4]{t})^{\alpha+D/\p} 
(\sqrt[4]{t}+|y|+r)^{\langle\beta\rangle-\alpha} \w \, .
\end{equation}
\item Lastly, using \eqref{est_Pi} and \eqref{est_dGamma_inc}, the third right-hand side
	term is (for $\ve\leq r$) bounded by 
\begin{equation}\label{tmp21}
\ve^{|\beta|-\langle\beta\rangle} (\sqrt[4]{t})^{\langle\gamma\rangle}
\begin{cases} 
|y|^{\alpha+D/\p-\bar n} (|y|+r)^{\langle\beta\rangle-\langle\gamma\rangle+\bar n-\alpha} \w \, , \\
(|y|+r)^{\langle\beta\rangle-\langle\gamma\rangle+D/\p} \w \,.
\end{cases} 
\end{equation}
\end{itemize}
For the first alternative of \eqref{est_dGamma_pp_inc} we choose $\sqrt[4]{t}=|y|$, so that \eqref{tmp19} and \eqref{tmp20} are both bounded by 
\begin{equation}
\ve^{|\beta|-\langle\beta\rangle}
|y|^{\alpha+D/\p} 
(|y|+r)^{\langle\beta\rangle-\alpha} \w \, , 
\end{equation}
and the first alternative in \eqref{tmp21} is bounded by 
\begin{equation}
\ve^{|\beta|-\langle\beta\rangle}
|y|^{\alpha+D/\p+\langle\gamma\rangle-\bar n} 
(|y|+r)^{\langle\beta\rangle-\langle\gamma\rangle+\bar n-\alpha} \w \, .
\end{equation}
Since here $[\gamma]\geq0$ and $\sum_\n\gamma(\n)>0$ (see \eqref{e_def_dgamma}), we obtain from \eqref{eq_supp1_lem_homProp} that $|\gamma|\geq2+(\kk-1)\alpha+\bar n$. 
By $2+(\kk-1)\alpha=a>0$ this implies $|\gamma|>\bar n$, and hence also $\langle\gamma\rangle>\bar n$. 
Altogether we obtain
\begin{align}
&\Big(\int_{B_r(x)}dz\,\EE^\frac{\p}{q}\Big|
\sum_\n (d\Gamma_{x\,z+y}-d\Gamma_{xz}\Gamma_{z\,z+y})_\beta^{\delta_\n} ((\cdot-z-y)^\n)_{t=|y|^4}(z+y)
\Big|^q\Big)^\frac{1}{\p} \\
&\lesssim \ve^{|\beta|-\langle\beta\rangle}
|y|^{\alpha+D/\p} 
(|y|+r)^{\langle\beta\rangle-\alpha} \w 
\quad\textnormal{for }\ve\leq r \, . 
\end{align}
Now, the first alternative of \eqref{est_dGamma_pp_inc} follows from a duality argument, the details of which can be found in \cite[Section~2.16]{BOT}.

To prove the second alternative of \eqref{est_dGamma_pp_inc}, we choose instead
$\sqrt[4]{t}=|y|+r$, so that \eqref{tmp19}, \eqref{tmp20}, and the second alternative of \eqref{tmp21} are all bounded by 
\begin{equation}
\ve^{|\beta|-\langle\beta\rangle}
(|y|+r)^{\langle\beta\rangle+D/\p} \w \, .
\end{equation}
The desired estimate follows again from the duality argument mentioned above. 
\end{proof}

It is only left to prove the corresponding boundedness estimate for $ d \Gamma$. 

\begin{proposition}[Three-point argument IV, Step 9.d)]\label{prop_3ptIV}
Let $[\beta]\geq0$ and assume that $\eqref{est_deltaPi}_\beta$ and $\eqref{est_deltaPi_inc}_\beta$ hold, 
that $\eqref{est_dGamma}_\beta^\gamma$ holds for all $\gamma$ populated and not purely polynomial, 
and that $\eqref{est_Pi}_{\beta'}$ holds for all populated $\beta'\prec\beta$. 
Then $\eqref{est_dpin}_\beta$ holds and thus $\eqref{est_dGamma}_\beta^\gamma$ holds for all $\gamma$. 
\end{proposition}

\begin{proof}
Similarly to the proof of Proposition~\ref{prop_3ptIII}, 
following \cite[Section~2.18]{BOT}, the proof is based on the identity 
\begin{equation}
\sum_\n (d\Gamma_{xz})_\beta^{\delta_\n} (\cdot-z)^\n
= \delta\Pi_x 
- (\delta\Pi_x-d\Gamma_{xz}\Pi_z)_\beta
- \sum_{[\gamma]\geq0} (d\Gamma_{xz})_\beta^\gamma \Pi_{z\gamma} \, .
\end{equation}
Applying $(\cdot)_t(z)$ and the
$(\int_{B_r(x)}dz\,\EE^\frac{\p}{q}|\cdot|^q)^\frac{1}{\p}$-norm to this identity and
splitting the right-hand side terms via the triangle inequality, 
its first right-hand side term is by bounded by 
\begin{equation}
\Big(\int_{B_r(x)}dz\,\big(
\ve^{|\beta|-\langle\beta\rangle}
(\sqrt[4]{t})^\alpha (\sqrt[4]{t}+|x-z|)^{\langle\beta\rangle-\alpha} \w )^{\p}\Big)^\frac{1}{\p}
\lesssim \ve^{|\beta|-\langle\beta\rangle}
(\sqrt[4]{t})^\alpha (\sqrt[4]{t}+r)^{\langle\beta\rangle-\alpha} \w r^{D/\p} \, ,
\end{equation}
as a consequence of \eqref{est_deltaPi}. 
On the other hand, using \eqref{est_deltaPi_inc}, \eqref{est_dGamma}, and \eqref{est_Pi},
the second and third right-hand side terms are (for $\ve\leq r$) bounded by 
\begin{equation}
\ve^{|\beta|-\langle\beta\rangle}
(\sqrt[4]{t})^{\alpha+D/\p} 
(\sqrt[4]{t}+r)^{\langle\beta\rangle-\alpha} \w 
+\sum_{\gamma\prec\beta}
\ve^{|\beta|-\langle\beta\rangle}
(\sqrt[4]{t})^{\langle\gamma\rangle} 
r^{\langle\beta\rangle-\langle\gamma\rangle+D/\p} \w \, .
\end{equation}
Choosing $\sqrt[4]{t}=r$, all three right-hand side terms are bounded by
$\ve^{|\beta|-\langle\beta\rangle} r^{\langle\beta\rangle+D/\p}\w$, 
and the claim follows from a duality argument, 
the details of which can be found in \cite[Section~2.16]{BOT}.
\end{proof}

\appendix

\section{Proofs of the algebraic framework}\label{sec:alg_proofs}

We collect here all proofs related to the algebraic skeleton introduced in Section~\ref{sec:setup}.
Most proofs are similar to existing arguments in the literature, 
see in particular \cite{OST23} for a similar point of view. 
However, since we are in a setting of an equation which would be formally supercritical without the damping of the nonlinearity with suitable powers of $\ve$, 
which affects the notion of homogeneity, we provide all details for completeness. 

In the following proofs we frequently appeal to the \emph{length} of a multiindex defined by 
\begin{equation}\label{eq_def_len}
\begin{aligned}
\ell( \beta)\coloneqq \sum_{k} \beta(k) + \sum_{\n} \beta (\n)\,.
\end{aligned}
\end{equation}

\begin{proof}[Proof of Lemma \ref{lem_hom_props}]
We state the proofs of the properties in the same order as stated above: 
\begin{itemize}
\item 
We have 
\begin{equation}\label{eq_supp1_lem_homProp}
 | \beta| -|0| 
= | \beta| - \alpha
=
\big( 2 + (\kk-1)\alpha \big)
\sum_{k} \beta (k) 
+ \sum_{\n} (|\n| - \alpha) \beta (\n)\,,
\end{equation}
which is clearly additive.

\item Because $2+(\kk-1)\alpha=a>0$ and $\alpha<0$, 
the display  \eqref{eq_supp1_lem_homProp} is non-negative 
since the coefficients in both sums are positive.

\item If $ | \beta| = |0| $, then by \eqref{eq_supp1_lem_homProp}
\begin{equation}
\begin{aligned}
0
=
\big(2+(\kk-1)\alpha\big)
\sum_{k} \beta (k) 
+ \sum_{\n} (|\n| - \alpha) \beta (\n)\,.
\end{aligned}
\end{equation}
However, as mentioned above, the coefficients in both sums are positive. Thus,
the equality can only hold if $ \beta =0$.
The reverse direction is immediate.
\end{itemize}
The proof can be repeated with $\langle\cdot\rangle$ replacing the homogeneity $|\cdot|$, 
just noting that $a-\kappa>0$ by assumption. 
\end{proof}

\begin{proof}[Proof of Lemma \ref{lem_ord_props}]
We prove the properties in the same order as stated in the lemma. 
\begin{itemize}
\item We have
\begin{equation}\label{eq_supp1_lem_ordProps}
\begin{aligned}
| \beta|_{\prec} - | 0 |_{\prec}
= | \beta| - | 0| + \tfrac{D}{2} [\beta]\,,
\end{aligned}
\end{equation}
where we used that $[0]=0$. We already proved in Lemma~\ref{lem_hom_props},
that $ \beta \mapsto | \beta| - | 0|$ is additive. Also $[\beta] $ is additive by
definition, which concludes additivity of $| \beta|_{\prec} - | 0 |_{\prec}$.

\item If $ [\beta ] \geqslant 0$, then 
\begin{equation}
\begin{aligned}
 | \beta|_{\prec}- | 0|_{\prec}
\geqslant  | \beta| -|0| \geqslant 0\,,
\end{aligned}
\end{equation}
where we used \eqref{eq_supp1_lem_ordProps} in the first inequality and
Lemma~\ref{lem_hom_props} in the second one.

\item 
If $ [ \beta] \geqslant 0$, then
\begin{equation}
\begin{aligned}
0=
| \beta|_{\prec} - | 0 |_{\prec}
=(| \beta| - | 0|) + \tfrac{D}{2} [\beta]  
\end{aligned}
\end{equation}
implies $| \beta| - | 0|=0$. Together with Lemma~\ref{lem_hom_props}, we see that $\beta =0$.
The reverse direction is immediate.

\item If $[\beta]\geq0$, then 
\begin{equation}
\begin{aligned}
| \beta|_{\prec} \geqslant  | \beta| + \tfrac{D}{2}  \geq \alpha+D/2 \,,
\end{aligned}
\end{equation}
where in the second inequality we used Lemma~\ref{lem_hom_props}.

\item In the remainder of the proof we assume that $ \beta$ is populated. 
 If $ \beta = \delta_{\n}$ or $\beta=\delta_{\kk}+\delta_{\n_1}+\dots+\delta_{\n_{\kk}}$, then 
$[\beta]=-1$ and thus $|\beta|_\prec=|\beta|$. 
On the other hand, if  $ [ \beta]\geqslant 0$ we have just shown in the previous bullet that $|\beta|_\prec\geq|\beta|+D/2$, which is larger than $|\beta|$.\qedhere
\end{itemize}
\end{proof}

\begin{proof}[Proof of Lemma \ref{lem_ind_legal}]
\textbf{Step 1 \textnormal{(proof of Item 1)}.}
Fix an arbitrary threshold $C>0$ and let $ \gamma$ be populated with $ |
\gamma|_{\prec}< C$. 
First, consider $ \gamma =\delta_\n$ purely polynomial. In this case
\begin{equation}
\begin{aligned}
|\n|
=
| \delta_{\n}|_{\prec}< C\,,
\end{aligned}
\end{equation}
which only holds for finitely many $\n$.

Similarly, if $\gamma=\delta_{\kk}+\delta_{\n_1}+\dots+\delta_{\n_{\kk}}$, then 
\begin{equation}
2+|\n_1|+\dots+|\n_{\kk}|=|\gamma|_\prec<C \, ,
\end{equation}
which holds for finitely many $\n_1,\dots,\n_{\kk}$. 

To conclude that there are only finitely many $ \gamma$ satisfying the given
order bound
with $[ \gamma] \geqslant 0$, we argue as follows:
\begin{itemize}
\item First, $[\gamma]\geq0$ implies by Lemma~\ref{lem_ord_props}
\begin{equation}
 \alpha 
+ \big(2+(\kk-1)\alpha\big)\sum_{k} \gamma (k) 
+ 
 \sum_{ {\n}}(|\n| - \alpha ) \gamma ({\n})
+ \tfrac{D}{2} 
=
| \gamma| + \tfrac{D}{2}
\leqslant 
| \gamma|_{\prec} 
 <C\,.
\end{equation}
Lower bounding the left-hand side using $2+(\kk-1)\alpha=a>0$ yields
\begin{equation}
\begin{aligned}
\alpha + \sum_{\n} (|\n| - \alpha)  \gamma (\n) 
\leqslant 
| \gamma|_{\prec}
< C\,.
\end{aligned}
\end{equation}
In particular, this bound can only hold for finitely many polynomial components
of
$ \gamma$.
\item On the other hand, for the $k$-components we use the lower bound
\begin{equation}
\begin{aligned}
\alpha + \tfrac{D}{2} (1+ [ \gamma])
\leqslant 
| \gamma|_{\prec} < C\,,
\end{aligned}
\end{equation}
where we took advantage of $ | \gamma | \geqslant \alpha$, cf.
Lemma~\ref{lem_hom_props}. 
Rearranging this inequality, we see that
\begin{equation}
\begin{aligned}
\tfrac{D}{2}\sum_{k}(k-1) \gamma(k) < C - \alpha -\tfrac{D}{2} \left( 1 - \sum_{\n}
\gamma(\n) \right)\,,
\end{aligned}
\end{equation}
which can only be satisfied for finitely many $ \gamma$ using that $
\sum_{\n} \gamma(\n)$ is bounded by a $C$ dependent constant, from the previous
bullet.
\end{itemize}

\textbf{Step 2 \textnormal{(proof of Item 2)}.}
Consider first the case of $ \beta= \delta_{\n}$ purely polynomial. Then 
\begin{equation}
\begin{aligned}
| \beta|_{\prec} 
=|\delta_{\n}|_{\prec} 
= |\n| 
\geq 0
=| \delta_{\0}|_{\prec}
\end{aligned}
\end{equation}
with equality if and only if $ \n =\0$.
On the other hand, if $\beta=\delta_{\kk}+\delta_{\n_1}+\dots+\delta_{\n_{\kk}}$ then 
\begin{equation}
|\beta|_\prec
=2 +|\n_1|+\dots+|\n_{\kk}|
\geq 2>0=|\delta_\0|_\prec \, .
\end{equation}
Finally, if $[ \beta] \geqslant 0 $, by Lemma~\ref{lem_ord_props}, 
\begin{equation}
\begin{aligned}
| \beta |_{\prec}
\geq |0 |_{\prec}
= \alpha + \tfrac{D}{2} >0 = | \delta_{\0} |_{\prec}\,.
\end{aligned}
\end{equation}
This concludes the proof.
\end{proof}

\begin{proof}[Proof of Lemma \ref{lem_def_gamma}]
~
\begin{enumerate}
\item Let $ \pi \in  \RR [\![ z_{k}, z_{\n} ]\!]$, then $ \Gamma$ acts
naturally by
extension on $ \pi = \sum_{\gamma} \pi_{\gamma} z^{ \gamma}$:
\begin{equation}\label{e_1}
\begin{aligned}
\Gamma \pi
= \sum_{ \gamma} \pi_{ \gamma} \Gamma z^{ \gamma} 
= \sum_{ \gamma} \pi_{ \gamma} \sum_{ \beta} (\Gamma z^{ \gamma} )_{\beta}\, 
z^{ \beta}
=
\sum_{ \beta} \left( \sum_{ \gamma} \pi_{ \gamma} \Gamma^{\gamma}_{\beta}
\right) z^{ \beta}\,,
\end{aligned}
\end{equation}
where again we used the notation $ \Gamma^{ \gamma}_{ \beta}= (\Gamma
z^{ \gamma})_{\beta}$.
Thus, for the map $ \Gamma$ to be well-defined on  $\RR [\![ z_{k}, z_{\n}
]\!]$, we have to show that the sum 
\begin{equation}\label{eq_supp1_wd_gamma}
\begin{aligned}
 \sum_{ \gamma} \pi_{ \gamma} \Gamma^{\gamma}_{\beta} 
\end{aligned}
\end{equation}
is finite for every $ \beta$. In fact, we prove that $\Gamma^{\gamma}_{\beta} \neq 0$ for
only finitely many $ \gamma$. 

For this purpose, let $ \beta$ be arbitrary and consider first multiindices of the
simplest form: 
\begin{itemize}
\item $ \gamma =0 $, for which $ \Gamma z^{\gamma}= \Gamma 1 =1$,

\item $ \gamma = \delta_{k}$, then $ (\Gamma z^{\gamma})_{ \beta}  =
( \Gamma z_{k})_{ \beta}= (  z_{k})_{ \beta}  = \delta_{\beta}^{\delta_{k}}$,

\item $ \gamma = \delta_{\n}$, then 
$$  (\Gamma z^{\gamma})_{ \beta}  =
( \Gamma z_{\n})_{ \beta}= (z_{\n})_{ \beta} + \pi^{(\n)}_{ \beta}=
\delta_{ \beta}^{\delta_{\n}} +\pi^{(\n)}_{ \beta}\,, $$ 
which  does not vanish if and only if  $ \gamma = \delta_{\n}= \beta $ or $ |\n|<
| \beta|$.
\end{itemize}
Hence, there are only finitely many $ \gamma$ of this simple form such that $
\Gamma^{ \gamma}_{ \beta}\neq 0 $, for any $ \beta$.
Now proceed by induction over the lengths $\ell (\gamma) \leqslant \ell ( \beta)$, recall
\eqref{eq_def_len}, and assume that 
\begin{equation}
\begin{aligned}
| \{ \gamma \,: \, \Gamma^{ \gamma}_{ \beta}\neq 0 \ \text{and } \ \ell(\gamma
)=l \}| < \infty \qquad \forall \beta\,, \ l \leqslant \overline{l}\,.
\end{aligned}
\end{equation}
We split $ \gamma = \gamma_{1} + \gamma_{2}$, with $
\ell ( \gamma_{1}) , \ell( \gamma_{2}) < \ell ( \gamma) = \overline{l}+1$. Using the multiplicity of $ \Gamma$ we
have 
\begin{equation}
\begin{aligned}
\Gamma_{\beta}^{\gamma} 
= 
\big(  (\Gamma z^{\gamma_{1}}) (  \Gamma z^{\gamma_{2}}) \big)_{\beta}
= \sum_{\beta_{1}+ \beta_{2}= \beta}
\Gamma^{ \gamma_{1}}_{ \beta_{1}} \Gamma^{ \gamma_{2}}_{ \beta_{2}}\,,
\end{aligned}
\end{equation}
with the sum on the right-hand side being over finitely many multiindices $
\beta_{1}, \beta_{2}$.
Consequently, if 
$\Gamma_{\beta}^{\gamma} \neq 0$ for infinitely many $ \gamma$, then for at
least one pair $( \beta_{1}, \beta_{2})$ there must be infinitely many $
\gamma'$, $\ell ( \gamma ') \leqslant \overline{l}$, such that $ \Gamma_{
\beta_{1}}^{ \gamma'} \neq 0$. However, this contradicts the induction hypothesis.

Lastly, we notice that whenever $ \ell( \gamma) > \ell( \beta)$, we can write $
\gamma = \gamma_{1}+ \cdots + \gamma_{\ell ( \gamma)}$ with $ \ell (
\gamma_{i})=1$. As a consequence, 
\begin{equation}
\begin{aligned}
\Gamma_{\beta}^{\gamma} 
=
\Big( \prod_{i =1}^{\ell ( \gamma)} (\Gamma z^{\gamma_{i}}) \Big)_{\beta}
=
\sum_{ \beta_{1}+ \cdots +\beta_{\ell ( \gamma )} = \beta } \prod_{i
=1}^{\ell ( \gamma )} \Gamma_{ \beta_{i}}^{\gamma_{i}} =0\,,
\end{aligned}
\end{equation} 
since $ \beta_{j}=0$ for some $j$, for which $ \Gamma_{
\beta_{j}}^{\gamma_{j} } = \Gamma^{ \gamma_{j}}_{0}=0$ (cf.~with multiindices
of simplest form above). 

Overall there are only finitely many $ \gamma$ such that $ \Gamma_{
\beta}^{ \gamma} \neq 0 $.
 In particular, the sum in \eqref{eq_supp1_wd_gamma} is finite for every $ \beta$
 and therefore the map $ \Gamma$ is well-defined on $ \RR [\![ z_{k}, z_{\n}
]\!]$.

\item Again we proceed by induction over the length of $ \gamma$, starting with multiindices of the
simplest form:
\begin{itemize}
\item $ \gamma =0 $ or $ \gamma = \delta_{k}$, then $ (\Gamma - \mathrm{id})_{
\beta}^{ \gamma} = 0$ by definition of $ \Gamma$,

\item $ \gamma = \delta_{\n}$, then $  (\Gamma - \mathrm{id})_{
\beta}^{ \gamma} = \pi^{ (\n)}_{ \beta} $ which is only non-zero if $|
\delta_{\n}| = |\n| < |
\beta|$. 
Since $\langle\delta_\n\rangle=|\n|$ and $\kappa$ is chosen small enough such that $|\n|<|\beta|$ implies $|\n|<\langle\beta\rangle$, we also obtain $\langle\delta_\n\rangle<\langle\beta\rangle$.
\end{itemize}
Assume that $  (\Gamma - \mathrm{id})_{
\beta}^{ \gamma} \neq  0$ only if $ | \gamma| < | \beta|$ for all $ \gamma $
with $ \ell ( \gamma) \leqslant \overline{l} $.
Then for $ \gamma$ with $ \ell ( \gamma) = \overline{l}+1 $, we again decompose the
multiindex into $ \gamma = \gamma_{1}+ \gamma_{2}$ with $ \ell (
\gamma_{1}) , \ell ( \gamma_{2}) \leqslant \overline{l}$. Now, if 
\begin{equation}\label{eq_supp1_gamma_nontriv}
\begin{aligned} 
 \Gamma_{
\beta}^{ \gamma} = 
\sum_{ \beta_{1}+ \beta_{2}= \beta}\Gamma_{\beta_{1}}^{ \gamma_{1}}
 \Gamma_{\beta_{2}}^{ \gamma_{2}}
\neq 0
\,,
\end{aligned}
\end{equation}
there must exist a decomposition $ \beta_{1}+ \beta_{2}= \beta$ such that for $
i=1,2 $
\begin{equation}
\begin{aligned}
\Gamma_{ \beta_{i}}^{ \gamma_{i}}
\begin{cases} 
=1 \quad &\text{with } \gamma_{i} = \beta_{i} \text{ or}\\
\neq 0 \quad  &\text{with } | \gamma_{i}| < | \beta_{i}|\,,
 \end{cases} 
\end{aligned}
\end{equation}
by  the induction hypothesis. If at least one $i$ satisfies $| \gamma_{i}| < |
\beta_{i}| $, we  conclude 
\begin{equation}
\begin{aligned}
| \gamma| = | \gamma_{1}| + | \gamma_{2}| - |0| < | \beta_{1}| + |
\beta_{2}| - |0| = | \beta| \,,
\end{aligned}
\end{equation}
using the additivity of $ | \cdot | - |0|$.
The remaining case, in which both $ \gamma_{1}=
\beta_{1}$ and $ \gamma_{2}= \beta_{2}$, yields 
$ \gamma = \gamma_{1}+ \gamma_{ 2}= \beta_{1}+ \beta_{2}= \beta$ and 
\begin{equation}
\begin{aligned}
 \Gamma_{
\beta}^{ \gamma} =  
 \Gamma_{
\beta}^{ \beta}
=
1
+
\sum_{\substack{ \beta_{1}+ \beta_{2}= \beta\\ \beta_{i}\neq \gamma_{i}}}\Gamma_{\beta_{1}}^{ \gamma_{1}}
 \Gamma_{\beta_{2}}^{ \gamma_{2}}
=1\,,
\end{aligned}
\end{equation}
because it requires $ |
\gamma_{i}| < |
\beta_{i}|$, $i=1,2$, 
for one of the latter summands to be non-zero, but this contradicts $ \gamma= \beta $.
The same argument can be repeated with $\langle\cdot\rangle$ replacing the homogeneity $|\cdot|$. 

\item
The statement follows almost verbatim along the lines of the previous
induction. First, note that for $ \gamma = \delta_{\n}$, we have
\begin{equation}
\begin{aligned}
0 \neq (\Gamma - \mathrm{id})_{
\beta}^{ \gamma} = \pi^{ (\n)}_{ \beta} \quad \implies \quad \beta
\text{ is populated, and } | \delta_{\n}|_{\prec}= | \n| < | \beta|\,,
\end{aligned}
\end{equation}
with the second property being a consequence of Item 2~above. Because $ \beta $ is
populated,  Lemma~\ref{lem_ord_props} implies $ | \beta| \leqslant  | \beta|_{
\prec} $ and the base case holds.

Because the induction in the previous proof only used additivity
of $ | \cdot | - |0|$, we can proceed exactly the same when replacing $|
\cdot|$ by $| \cdot |_{\prec}$. 
\end{enumerate}
\end{proof}

\begin{proof}[Proof of Lemma \ref{lem_dGamma}]
~
\begin{enumerate}
\item As in the proof of Lemma~\ref{lem_def_gamma} for $ \Gamma$ to be well-defined, we have to show that 
\begin{equation}
\begin{aligned}
\sum_{ \gamma} \pi_{ \gamma} d\Gamma^{\gamma}_{\beta} < \infty \qquad 
\textnormal{for all }\beta \textnormal{ and }\pi \in  \RR [\![ z_{k}, z_{\n} ]\!] \,.
\end{aligned}
\end{equation}
Let $ \beta$ be arbitrary and consider $
\gamma$ to be a multiindex of the simplest form: 
\begin{itemize}
\item if $ \gamma =0$ or $ \gamma = \delta_{k}$, then $
d\Gamma_{\beta}^{\gamma}=0$,

\item if $ \gamma= \delta_{\n}$, then $ d\Gamma_{ \beta}^{\gamma} =
d\pi^{(\n)}_{\beta}\neq 0 $ only if $ |\n| < \alpha + D/\p$, which is the
case for finitely many $\n$ only.
\end{itemize}
Thus, there are only finitely many $ \gamma$ of this simple form such that $
d \Gamma_{ \beta}^{ \gamma} \neq 0$. 
On the other hand, for any $ \gamma$ larger, we may decompose it into two
non-trivial parts $ \gamma
= \gamma_{1}+ \gamma_{2}$ such that
\begin{equation}\label{eq_supp1_dGamma}
\begin{aligned}
d \Gamma_{ \beta}^{\gamma}=
\sum_{ \beta_{1}+ \beta_{2}= \beta } \Gamma_{ \beta_{1}}^{\gamma_{1}}
d \Gamma_{\beta_{2}}^{ \gamma_{2}}
+
\sum_{ \beta_{1}+ \beta_{2}= \beta } 
d \Gamma_{\beta_{1}}^{ \gamma_{1}}
\Gamma_{ \beta_{2}}^{\gamma_{2}}\,.
\end{aligned}
\end{equation}
From the proof of Lemma~\ref{lem_def_gamma}, we know for every $ \beta$ that $ \Gamma_{ \beta}^{ \gamma} \neq
0$ for only finitely many $ \gamma$. Because both sums on the right-hand side
are finite, there can only be finitely many $ \gamma = \gamma_{1} +
\gamma_{2}$ such that $d \Gamma_{ \beta}^{\gamma}\neq 0$.

\item Once more, we proceed by induction over the length $ \ell ( \gamma)$. First, when $ \gamma$ is of the
simplest form $ \gamma = 0, \delta_{k}, \delta_{\n}$, we have 
\begin{equation}
\begin{aligned}
d\Gamma_{ \beta}^{ \gamma} \neq 0 \quad \implies \quad
\gamma = \delta_{\n} \text{ with } 
\begin{rcases}
|\delta_\n| \\
\langle\delta_\n\rangle
\end{rcases}
= |\n| 
< \alpha + \tfrac{D}{\p} 
\leq 
\begin{cases}
|\beta| + \tfrac{D}{\p} \\ 
\langle\beta\rangle + \tfrac{D}{\p} 
\end{cases}
,
\end{aligned}
\end{equation}
concluding the base case. For any larger $ \gamma$, we again use the
representation
\eqref{eq_supp1_dGamma} to see that $
d \Gamma_{ \beta}^{\gamma}\neq 0$ implies $
\Gamma_{ \beta_{1}}^{\gamma_{1}}
d \Gamma_{\beta_{2}}^{ \gamma_{2}}\neq 0$
(or $
\Gamma_{ \beta_{2}}^{\gamma_{2}}
d \Gamma_{\beta_{1}}^{ \gamma_{1}}\neq 0$)
for at least one of the summands. In this case, Lemma~\ref{lem_def_gamma}.2
yields that $ | \gamma_{1}| \leqslant | \beta_{1}|$, whereas the induction
hypothesis gives $ | \gamma_{2}| < | \beta_{2}| + \tfrac{D}{\p}$.
Together this implies 
\begin{equation}
\begin{aligned}
| \gamma| = | \gamma_{1}| + | \gamma_{2}| - |0| 
<
| \beta_{1}| + | \beta_{2}| +  \tfrac{D}{\p} -|0|
= | \beta| +  \tfrac{D}{\p}\,.
\end{aligned}
\end{equation}
The same argument can be repeated with $\langle\cdot\rangle$ replacing the homogeneity $|\cdot|$.

\item Assuming that $ d\pi^{(\n)}_{ \beta} \neq 0$ only if $
[\beta] \geqslant 0$, the base case in the previous item reads that for $ \gamma = 0, \delta_{k}, \delta_{\n}$
\begin{equation}
\begin{aligned}
d\Gamma_{ \beta}^{ \gamma} \neq 0 \quad \implies \quad
\gamma = \delta_{\n} \text{ with }
|\delta_\n|_{\prec} = |\n| < \alpha + \tfrac{D}{\p} \leq |\beta| + \tfrac{D}{\p} \leq | \beta|_{ \prec} \,,
\end{aligned}
\end{equation}
where the last inequality used that $[\beta]\geq0$ and $2\leq\p$.
The remainder of the induction remains the same, replacing $ | \cdot| $ with $|
\cdot |_{ \prec }$.
\end{enumerate}
\end{proof}

\begin{remark}[Exponential formulae]\label{rem_exp_form}
At this point it is worth noting that by multiplicity of $ \Gamma$, we have
\begin{equation}
\begin{aligned}
\Gamma z^{ \gamma}
= 
\prod_{k\geq\kk} z_{k}^{\gamma(k)}
\prod_{\n} \big( z_{\n} + \pi^{(\n)} \big)^{\gamma (\n)}\,.
\end{aligned}
\end{equation}
The latter product can be written in terms of 
\begin{equation}
\begin{aligned}
\prod_{\n} \big( z_{\n} + \pi^{(\n)} \big)^{\gamma (\n)}
&=
\prod_{\n} \sum_{m=0}^{\gamma(\n)} \frac{ \gamma(\n)!}{ m! ( \gamma(\n)-m)!} 
( \pi^{(\n)})^{m}z_{\n}^{ \gamma(\n)-m} \\
&= 
\prod_{\n} \sum_{m=0}^{ \infty} \frac{ 1}{ m! } 
( \pi^{(\n)})^{m}
\partial_{z_{\n}}^{m}
\big(z_{\n}^{ \gamma(\n)}\big)\\ 
&=
\sum_{ \substack{m (\bar\n ) \geqslant 0 \\ \bar\n \in \NN^{1+d}}}
\prod_{\bf{n}}
\frac{1}{m(\n)!} 
( \pi^{(\n)})^{m(\n)}
\partial_{z_{\n}}^{m(\n)}
\big(z_{\n}^{ \gamma(\n)}\big) \\
&=
\sum_{M=0}^{\infty}
\sum_{ \substack{m (\bar\n ) \geqslant 0 \\ \bar\n \in \NN^{1+d}}}
\mathds{1}_{ \sum_{\bar\n} m(\bar\n) = M}
\bigg(
\prod_{\bf{n}}
\frac{1}{m(\n)!} 
( \pi^{(\n)})^{m(\n)}
\partial_{z_{\n}}^{m(\n)}
\bigg)
\prod_{\n}
z_{\n}^{ \gamma(\n)} \\
& = 
\sum_{M=0}^{\infty}
\frac{1}{M!} 
\sum_{\n_{1}, \ldots, \n_{M} \in \NN^{1+d}}
\bigg(
\prod_{i =1}^{M} 
 \pi^{(\n_{i})}
 \partial_{z_{\n_{i}}}
 \bigg) 
\prod_{\n}
z_{\n}^{ \gamma(\n)} \,,
\end{aligned}
\end{equation}
where in the last step we went from (indistinguishable to distinguishable multiindex
assignments): for every $  m : \NN^{1+d} \to \NN$ such that $ \sum_{\n} m
(\n) = M $, there are precisely $ M!/ \prod_{\n} m(\n)!$ many multiindices such that
\begin{equation}
	\begin{aligned}
		|\{ i =1,\dots,M \,: \, \n_{i}= \n  \}| = m( \n) \,.
	\end{aligned}
\end{equation}
Hence, with the notation $ D^{\n} = \partial_{z_{\n}}$, we have 
\begin{equation}
\begin{aligned}
\Gamma z^{ \gamma}
=
\sum_{M =0}^{ \infty}
\frac{1}{M!} 
\sum_{\n_{1}, \ldots, \n_{M}}
\pi^{( \n_{1 })}\cdots \pi^{ ( \n_{M})}
D^{\n_{1}} \cdots D^{\n_{M}} z^{ \gamma} \,.
\end{aligned}
\end{equation}
In particular, evaluating at a multiindex $\beta$, we have
\begin{equation}\label{eq_gamma_prod_rep}
\begin{aligned}
\Gamma^{ \gamma}_{ \beta}
=
\big(
\Gamma z^{ \gamma}
\big)_{ \beta}
= 
\sum_{M =0}^{ \infty}\frac{1}{M!}
\sum_{\n_{1}, \ldots, \n_{M}}
\sum_{ \beta_{1}+ \cdots + \beta_{M+1}= \beta}
\pi^{( \n_{1 })}_{ \beta_{1}}\cdots \pi^{ ( \n_{M})}_{ \beta_{M}}
\big(D^{\n_{1}} \cdots D^{\n_{M}} z^{ \gamma}\big)_{ \beta_{M+1}}\,.
\end{aligned}
\end{equation}

For $d\Gamma$ notice that $(D^{\n})^{\gamma}_{ \beta} = \gamma(\n) 1_{\beta=\gamma -
\delta_{\n}} $, which implies that 
\begin{equation}\label{eq_supp1_expform_dGamma}
\begin{aligned}
d \Gamma z_{\n}^{ \gamma (\n)}
= \gamma(\n) \big(d \Gamma z_{\n}\big) \big( \Gamma z_{\n}^{ \gamma(\n)-1}\big)
= 
 d \pi^{(\n)} \,  \big( \Gamma D^{\n} z_{\n}^{ \gamma(\n)}\big)
\,.
\end{aligned}
\end{equation}
Thus, for $ \gamma $ with $\sum_k\gamma(k)=0$ 
\begin{equation}
	\begin{aligned}
d \Gamma z^{\gamma}
=
d \Gamma \prod_{ \n} z^{ \gamma ( \n)}_{\n}
= \sum_{\n} (d \Gamma z^{\gamma ( \n )}_{\n} ) \Gamma \prod_{ \n ' \neq \n } z^{\gamma ( \n'
)}_{\n'} 
=
\sum_{\n}
 d \pi^{(\n)} \,  \big( \Gamma D^{\n} z_{\n}^{ \gamma(\n)}\big)
 ( \Gamma z^{ \gamma - \gamma( \n ) \delta_{\n}})
	\end{aligned}
\end{equation}
where we inserted \eqref{eq_supp1_expform_dGamma} in the last step.
Because $ d \Gamma$ vanishes on $z_k$
and $\Gamma z_k=z_k$, the identity holds for all
multiindices $ \gamma$.
Lastly, we notice that we can simply rewrite 
\begin{equation}
	\begin{aligned}
		\big( \Gamma D^{\n} z_{\n}^{ \gamma(\n)}\big)
 ( \Gamma z^{ \gamma - \gamma( \n ) \delta_{\n}})
=
\Gamma ( D^{\n} z^{\gamma}) \,,
	\end{aligned}
\end{equation}
such that 
\begin{equation}
\begin{aligned}
d \Gamma z^{\gamma}
=
\sum_{\n} 
d \pi^{(\n)} \, 
\big(\Gamma D^{\n} z^{ \gamma }\big)\,.
\end{aligned}
\end{equation}
In particular, evaluating at a multiindex $ \beta$, we have
\begin{equation}\label{eq_dgamma_prod_rep}
	\begin{aligned}
		d \Gamma^{\gamma}_{\beta}=
		(d \Gamma z^{\gamma})_{\beta}
		= \sum_{ \beta_{1} + \beta_{2} = \beta}
		\sum_{\n} 
		d \pi^{(\n)}_{ \beta_{1}} \, 
		\big(\Gamma D^{\n} z^{ \gamma }\big)_{\beta_{2}}\,.
	\end{aligned}
\end{equation}
Notice that the sum over $\n$ in the above display can be restricted to $ |\n| < \alpha+
D/ \overline{p}$, by assumption in Lemma~\ref{lem_dGamma}, which is a finite sum. 
\end{remark}

\begin{proof}[Proof of Lemma \ref{lem_sets}]
~
\begin{enumerate}
\item Let $ \pi = \sum_{ \gamma } \pi_{ \gamma } z^{\gamma}  \in T $, then
assuming that 
$0 \neq \pi_{ \gamma} \, (D^{\n} z^{ \gamma} )_{ \beta} =  \gamma(\n)\pi_{\gamma} \, 1_{\beta=\gamma -
\delta_{\n}}$ yields
\begin{equation}
\begin{aligned}
[ \beta] = [ \gamma - \delta_{\n}]
= [ \gamma] - [\delta_{\n}] = [\gamma]+1 \geqslant 0\,.
\end{aligned}
\end{equation}
The last inequality holds 
because $ \pi_{ \gamma}$ vanishes if $ \gamma$ is not populated, in particular
implying that 
$ [ \gamma] \geqslant -1$.
Therefore, $ D^{\n} \pi \in \widetilde{T}$.

\item Let  $ \pi =\sum_{ \gamma } \pi_{ \gamma } z^{\gamma} $, then representation  \eqref{eq_gamma_prod_rep}
(recalled here)
\begin{equation}
\begin{aligned}
0 \neq 
\Gamma^{ \gamma}_{ \beta}
=
\big(
\Gamma z^{ \gamma}
\big)_{ \beta}
= 
\sum_{m =0}^{ \infty}
\sum_{\n_{1}, \ldots, \n_{m}}
\sum_{ \beta_{1}+ \cdots + \beta_{m+1}= \beta}
\pi^{( \n_{1 })}_{ \beta_{1}}\cdots \pi^{ ( \n_{m})}_{ \beta_{m}}
\big(D^{\n_{1}} \cdots D^{\n_{m}} z^{ \gamma}\big)_{ \beta_{m+1}}\,,
\end{aligned}
\end{equation}
implies that
\begin{equation}
\begin{aligned}
[ \beta ]
=
[\beta_{1}]+ \cdots + [\beta_{m +1}] 
\geqslant 
-m + [\beta_{m +1}] 
\end{aligned}
\end{equation}
where we used that $ [ \beta_{1}], \ldots , [ \beta_{m}] \geqslant -1 $ since $
\beta_{i}$ must be populated otherwise $ \pi^{(\n_{i})}_{ \beta_{i}} =0$.
Together with 
\begin{equation}\label{eq_tri_supp2}
\begin{aligned}
\big(D^{\n_{1}}
\cdots D^{\n_{m}}\big)_{ \beta_{m+1}}^{ \gamma}
=
\gamma ( \n_{1}) ( \gamma - \delta_{\n_{1}}) (\n_{2}) 
\cdots ( \gamma - \delta_{\n_{1}} - \cdots - \delta_{\n_{m-1}} )(\n_{m})
\delta_{ \beta_{m+1}}^{\gamma - \delta_{\n_{1}}- \cdots - \delta_{\n_{m}}}\,,
\end{aligned}
\end{equation}
this allows us to write
\begin{equation}\label{eq_set_supp1}
\begin{aligned}
[ \beta]
\geqslant 
-m + [\gamma]- ([\n_{1}] + \cdots+ [\n_{m}])
= [\gamma] \,.
\end{aligned}
\end{equation}
Now, if $ \pi \in \widetilde{T}$, then $ (\Gamma \pi)_{ \beta} =
\sum_{\gamma,\, [ \gamma] \geqslant 0} \pi_{ \gamma} \, \Gamma^{ \gamma}_{
\beta}$. Thus, \eqref{eq_set_supp1}
yields that $ [ \beta] \geqslant 0 $, implying that $ \Gamma \pi \in
\widetilde{T}$.

On the other hand, if $ \pi \in T $ we only have to take additionally into account
\begin{itemize}
\item $ \gamma= \delta_{\n}$, for which we simply have
\begin{equation}
\begin{aligned}
0 \neq  \Gamma^{ \gamma}_{
\beta} =
1_{\beta=\delta_{\n}} +
\pi_{\beta}^{(\n)}
\quad\text{ implies that }\quad
[\beta]\geq0 \text{ or } \beta \textnormal{ is purely polynomial}
\, ,
\end{aligned}
\end{equation}

\item $\gamma=\delta_{\kk}+\delta_{\n_1}+\dots+\delta_{\n_{\kk}}$, 
where by multiplicativity of $\Gamma$ and $\Gamma z_k=z_k$ (see Lemma~\ref{lem_def_gamma})
\begin{equation}
\Gamma_\beta^{\delta_{\kk}+\delta_{\n_1}+\dots+\delta_{\n_{\kk}}}
= \sum_{\beta_0+\dots+\beta_{\kk}=\beta} 
\Gamma_{\beta_0}^{\delta_{\kk}} \Gamma_{\beta_1}^{\delta_{\n_1}} \cdots \Gamma_{\beta_{\kk}}^{\delta_{\n_{\kk}}} 
= \sum_{\delta_{\kk}+\beta_1+\dots+\beta_{\kk}=\beta} 
\Gamma_{\beta_1}^{\delta_{\n_1}} \cdots \Gamma_{\beta_{\kk}}^{\delta_{\n_{\kk}}} \, .
\end{equation}
By the just established first bullet, either $[\beta_i]\geq0$ or $\beta_i$ is purely polynomial for all $i=1,\dots,\kk$.
If all $\beta_i$ are purely polynomial, then $\beta$ is also of the form $\delta_{\kk}+\delta_{\m_1}+\dots+\delta_{\m_{\kk}}$;
if at least one $\beta_i$ satisfies $[\beta_i]\geq0$, then 
\begin{equation}
[\beta]=[\delta_{\kk}]+[\beta_1]+\dots+[\beta_{\kk}]
\geq \kk-1 + (\kk-1)(-1)
=0 \, ,
\end{equation}
so that in any case $\beta$ is populated. 
\end{itemize}
Thus, if $ (\Gamma \pi)_{ \beta} =
\sum_{\gamma \text{ pop.}} \pi_{ \gamma} \, \Gamma^{ \gamma}_{
\beta} \neq 0 $ we have that $ [ \beta] \geqslant 0$ by the first
part or $ \beta$ populated by the two bullets above.

\item Let $ \pi = \sum_{ \gamma } \pi_{ \gamma } z^{\gamma}  \in T $, then with
\eqref{eq_dgamma_prod_rep} we can write
\begin{equation}
\begin{aligned}
(d \Gamma \pi)_{ \beta} = \sum_{ \gamma \text{ pop.}} \pi_{\gamma} d
\Gamma^{ \gamma}_{\beta}
=  \sum_{ \gamma \text{ pop.}} \pi_{\gamma} 
\sum_{\substack{\n \\ | \n| < \alpha + \frac{D}{\p}}}
\sum_{\beta_{1}+ \beta_{2} = \beta}(d \pi^{(\n)})_{ \beta_{1}} ( \Gamma
D^{\n} )^{ \gamma}_{\beta_{2}}\,.
\end{aligned}
\end{equation}
If $(d \Gamma \pi)_{ \beta} \neq 0  $ then there exists a populated $ \gamma$ and a
decomposition $ \beta=  \beta_{1}+ \beta_{2} $, with $ [ \beta_{1}] \geqslant 0$ (as
otherwise $ d \pi^{(\n)}_{\beta_{1}} =0$). Hence, 
\begin{equation}
\begin{aligned}
[ \beta] = [ \beta_{1}]+ [ \beta_{2}] \geqslant 0\,,
\end{aligned}
\end{equation}
where we used additionally that $[ \beta_{2}] \geqslant 0 $ since $ \Gamma
D^{\n} T \subseteq \widetilde{T}$ by the first two statements of
the lemma. This concludes that $ d \Gamma \pi \in \widetilde{T}$.
\end{enumerate}
\end{proof}

\begin{proof}[Proof of Lemma \ref{lem_triang}]
~
\begin{enumerate}
\item 
We recall representation \eqref{eq_gamma_prod_rep} and consider a single term
\begin{equation}
\begin{aligned}
\pi^{( \n_{1 })}_{ \beta_{1}}\cdots \pi^{ ( \n_{m})}_{ \beta_{m}}
\big(D^{\n_{1}} \cdots D^{\n_{m}} z^{ \gamma}\big)_{ \beta_{m+1}}
\end{aligned}
\end{equation}
in the expansion. First, we note that for 
$ \beta = \beta_{1}+ \cdots + \beta_{m+1}$
\begin{equation}\label{eq_tri_supp1}
\begin{aligned}
| \beta |_{\prec}
&= 
| \beta_{1} |_{\prec}
+\cdots + 
| \beta_{m+1} |_{\prec}
- m |0|_{\prec}\\
&=
| \beta_{1} |_{\prec}
+\cdots + 
| \beta_{m} |_{\prec}
+ | \gamma|_{\prec} - (| \delta_{\n_{1}}|_{\prec} + \cdots + |
\delta_{\n_{m}}|_{ \prec} )
\end{aligned}
\end{equation}
where in the last step we used that $ \beta_{m+1} = \gamma -
\delta_{\n_{1}} - \cdots - \delta_{\n_{m}}$ by \eqref{eq_tri_supp2}.
Because $ | \delta_{\n_{i}}|_{\prec} = |\n_{i}| < | \beta_{i}| \leqslant |
\beta_{i}|_{\prec}$, by assumption on $ \pi_{ \beta}^{(\n)}$ and
Lemma~\ref{lem_ord_props} (note that $ \beta_{i}$ is populated as otherwise $
\pi_{ \beta_{i}}^{( \n_{i})} =0$), identity \eqref{eq_tri_supp1} then implies 
\begin{equation}\label{eq_tri_supp20}
\begin{aligned}
| \beta |_{\prec}
\geqslant | \beta_{i}|_{\prec} - | \delta_{\n_{i}}|_\prec + | \gamma|_{\prec}\,,
\qquad \forall i=1, \ldots, m \,.
\end{aligned}
\end{equation}

\begin{enumerate}
\item Now, if $\gamma$ is populated and not purely polynomial, then 
\begin{equation}
| \gamma |_{\prec} 
\geqslant | \gamma | 
= 
\alpha
+ (2+(\kk-1) \alpha) \sum_{k} \gamma(k) + \sum_{\n} (|\n| - \alpha) \gamma(\n)  >
\alpha + |\n_{i}| - \alpha 
= |\delta_{\n_{i}}|_{\prec}\,,
\end{equation}
where we used Lemma~\ref{lem_ord_props} in the first inequality. In the
following equality we used the
definition of the homogeneity \eqref{e_def_homo}, before lower bounding
$2+(\kk-1)\alpha=a>0$ and $\sum_k \gamma(k)\geq1$ (because $\gamma$ is populated and not purely polynomial), 
and all
terms of the last sum by zero except $ (|\n_{i}| - \alpha)$.
Note that $ \gamma( \n_{i}) \geqslant 1 $ as otherwise \eqref{eq_tri_supp2}
would vanish. 
Together with \eqref{eq_tri_supp20}, this implies 
\begin{equation}
\begin{aligned}
| \beta |_{\prec} \geqslant 
 | \beta_{i}|_{\prec} - | \delta_{\n_{i}}|_\prec + | \gamma|_{\prec} > |
\beta_{i}| \qquad \forall i =1, \ldots, m \,.
\end{aligned}
\end{equation}

\item On the other hand, if $ \gamma = \delta_{\n}$ is purely polynomial then 
\begin{equation}
\begin{aligned}
 \Gamma_{ \beta}^{ \delta_{\n}} = \delta_{\beta}^{\delta_{\n}} + \pi_{
\beta}^{(\n)}\,,
\end{aligned}
\end{equation}
which 
depends on $ \pi_{ \beta}^{(\n)}$.
\end{enumerate}

\item For the second part of the proof, we keep in mind the identity
\eqref{eq_dgamma_prod_rep} spelled out as
\begin{equation}
\begin{aligned}
d \Gamma^{ \gamma}_{ \beta}
=
\sum_{\substack{\n\\ |\n| < \alpha+ \frac{D}{\p}}} \sum_{ \beta_{1}+ \beta_{2}= \beta}
d \pi^{(\n)}_{ \beta_{1}}
\sum_{ \beta'} \Gamma_{\beta_{2}}^{ \beta '}\  \gamma(\n)
1_{\beta'=\gamma- \delta_{\n}}  \,.
\end{aligned}
\end{equation}
\begin{enumerate}
\item 
We begin with populated $ \gamma $ which is not purely polynomial. 
First, we note that $ \beta_{2} \neq 0 $, since otherwise the triangularity of
$ \Gamma$ (see Lemma~\ref{lem_def_gamma}) implies $ \beta' =0 $ and thus $ \gamma =
\delta_{\n}$, contradicting the assumption on $\gamma$.

In Lemma~\ref{lem_sets} we show that $ \Gamma D^{\n} T \subseteq \widetilde{T}$,
i.e.~$ [ \beta_{2}] \geqslant 0 $, hence, $ |
\beta_{2} |_{\prec }> |0 |_{\prec}$ by Lemma~\ref{lem_ord_props} and in
particular
\begin{equation}
\begin{aligned}
| \beta |_{\prec} = | \beta_{1}|_{\prec} + | \beta_{2}|_{\prec} -
|0|_{\prec} > | \beta_{1}|_{\prec}\,.
\end{aligned}
\end{equation}

\item If $ \gamma = \delta_{\n}$ is purely polynomial, then $d \Gamma_{
\beta}^{\delta_{\n}} = d \pi_{ \beta}^{(\n)}$ which depends at most on $ d \pi_{
\beta}^{(\n)}$ (if it does not vanish).
\end{enumerate}
\end{enumerate}
\end{proof}

\begin{proof}[Proof of Lemma~\ref{lem_dep_pi}]
We recall definitions \eqref{hierarchy} of $ \Pi_{x \beta}^{-}$ and
\eqref{e_def_ord} of the order $| \cdot |_{\prec}$.
\begin{enumerate}
\item For the first contribution to $\Pi^-_{x\beta}$ we note that by additivity of $| \cdot |_{\prec}  -
|0|_{\prec} $, see Lemma~\ref{lem_ord_props}, 
$ \beta = \delta_{k} + \beta_{1} + \cdots+
\beta_l$ implies
\begin{equation}
\begin{aligned}
| \beta |_{\prec} 
&= 
| \delta_{k}|_{\prec}
+ | \beta_{1} |_{\prec} 
+ \cdots + | \beta_l |_{\prec}  
- l | 0 |_{\prec}\\
&\geq
 2+ \kk \alpha + \tfrac{D}{2}k
+ | \beta_{1} |_{\prec} 
+ (l-1)\alpha
 - l | 0 |_{\prec} \,,
\end{aligned}
\end{equation}
where we used $ | \delta_{k}|_{\prec} = 2+ \kk \alpha + \tfrac{D}{2}k$ 
and $|\cdot|_\prec\geq|\cdot|\geq\alpha$ by Lemma~\ref{lem_ord_props} and Lemma~\ref{lem_hom_props}. 
From $|0|_\prec=\alpha+D/2$, $(k-l)D/2\geq0$, and $2+(\kk-1)\alpha=a>0$, we deduce $|\beta|_\prec>|\beta_1|_\prec$.
By symmetry we conclude that also $\beta_2,\dots,\beta_l\prec\beta$.

For the second contribution to $\Pi^-_{x\beta}$ we use again additivity of $|\cdot|_\prec-|0|_\prec$ to see
\begin{equation}
|\beta|_\prec
= |\beta_0|_\prec+\dots+|\beta_l|_\prec - l|0|_\prec \, .
\end{equation}
Since $\beta_0(\n)=0$ for all $\n$ (see Remark~\ref{rem:bphz_inductive}) it holds 
\begin{equation}
|\beta_0|_\prec 
= \alpha\big(1+(\kk-1)\sum_{k\geq\kk}\beta_0(k)\big)
+ 2\sum_{k\geq\kk}\beta_0(k)
+ \tfrac{D}{2}\big(1+\sum_{k\geq\kk}(k-1)\beta_0(k)\big) \, ,
\end{equation}
which since $\beta_0\neq0$ (see Remark~\ref{rem:bphz_inductive}) is bounded below by $2+\kk\alpha+\kk D/2$.
From here we conclude as above, noting that also $(\kk-l)D/2\geq0$. 

\item Looking at the definition \eqref{hierarchy} of $ \Pi_{x \beta}^{-}$ we see that the relevant contribution to $\Pi^-_{x\,\beta+k\delta_\0}+c^{(k)}_\beta$ is given by 
\begin{equation}\label{tmp11}
 - \sum_{k'<\kk} \sum_{l\geq0} \sum_{\substack{\beta_0+\dots+\beta_l=\beta+k\delta_\0 \\ \beta_i\neq0\forall i}} 
\tbinom{k'}{l} c^{(k')}_{\beta_0} W_{k'-l,\ve}(\Pi_{x0}) \Pi_{x\beta_1}\cdots\Pi_{x\beta_l}
+c^{(k)}_\beta \, .
\end{equation}
By the additivity of $|\cdot|_\prec-|0|_\prec$  we have for $\beta_0+\dots+\beta_l = \beta+k\delta_\0$ that 
\begin{equation}
|\beta_0|_\prec + \dots + |\beta_l|_\prec - (l+1)|0|_\prec 
= |\beta|_\prec + k|\delta_\0|_\prec - (k+1)|0|_\prec \, .
\end{equation}
By $|\delta_\0|_\prec=0$ therefore 
\begin{align}
|\beta_0+k'\delta_\0|_\prec
&= |\beta_0|_\prec - k' |0|_\prec \\
&= |\beta|_\prec - |\beta_1|_\prec-\dots-|\beta_l|_\prec - (k-l+k')|0|_\prec \\
&= |\beta+k\delta_\0|_\prec - |\beta_1|_\prec-\dots-|\beta_l|_\prec - (k'-l)|0|_\prec \, .
\end{align}
Note that $k'-l\geq0$ and $|0|_\prec=\alpha+D/2>0$. 
Furthermore, $|\beta_i|_\prec$ for $i=1,\dots,l$ coincides with $|\n|$ if $\beta_i=\delta_\n$ or is strictly positive if $[\beta_i]\geq0$. 
Thus 
\begin{equation}
|\beta_0+k'\delta_\0|_\prec
\leq |\beta+k\delta_\0|_\prec \, ,
\end{equation}
with equality if and only if $l=k'$ and $\beta_1=\dots=\beta_l=\delta_\0$. 
Hence the contributions to \eqref{tmp11} where $\beta_0+k'\delta_\0$ is not strictly smaller (with respect to $\prec$) than $\beta+k\delta_\0$ are of the form 
\begin{equation}
 - \sum_{k'<\kk} \sum_{\beta_0+k'\delta_\0=\beta+k\delta_\0} 
 c^{(k')}_{\beta_0} 
 +c^{(k)}_\beta \, .
\end{equation}
Since $\beta_0(\0)=0$ by Remark~\ref{rem:bphz_inductive} and $\beta(\0)=0$ by assumption, 
the restriction under the second sum implies that $k'=k$, 
which in turn yields $\beta_0=\beta$ and the whole expression vanishes. 
\end{enumerate}
\end{proof}

\begin{proof}[Proof of Lemma \ref{lem_pop}]
	If $\Pi^-_{x\beta}\neq0$ then at least one of the three contributions in \eqref{hierarchy} does not vanish. 
If the second or third contributions do not vanish, then $[\beta]\geq0$:
For the second contribution this follows from $[\beta]=[\beta_0]+[\beta_1]+\dots+[\beta_l]\geq\kk-1-l>-1$, 
where we used that $\beta_0$ has no polynomial components and at least one $k'\geq\kk$ component (see Remark~\ref{rem:bphz_inductive}), and that $l\leq k<\kk$.
For the third contribution this follows from $[\beta=0]=0$. 
Assume now that the first contribution does not vanish, and that $[\beta]<0$. 
Then it follows from $[\beta]=[\delta_k]+[\beta_1]+\dots+[\beta_l]\geq k-1-l\geq-1$, 
that $l=k$ and $\beta_1,\dots,\beta_l$ are purely polynomial, i.e.~$\beta=\delta_k+\delta_{\n_1}+\dots+\delta_{\n_k}$.
In this case $W_{k-l,\ve}(\Pi_{x0})\Pi_{x\beta_1}\cdots\Pi_{x\beta_{l}} = (\cdot-x)^{\n_1+\cdots+\n_k}$, 
which is a polynomial of degree $|\n_1|+\cdots+|\n_k|=|\beta|-2+(k-\kk)\alpha$.
Hence if $k=\kk$ then $T_x^{<|\beta|-2}$ vanishes, 
while for $k>\kk$, $T_x^{<|\beta|-2}$ acts as the identity and $(\mathrm{id}-T_x^{<|\beta|-2})$ vanishes.
\end{proof}

\begin{lemma}\label{lem:no_taylor_for_phi4}
Let $\beta$ satisfy $\sum_{k>\kk}\beta(k)=0$ and $|\n|<|\beta|-2$. 
Then for $0\leq l\leq \kk$, and $\delta_{\kk}+\beta_1+\dots+\beta_l=\beta$ 
\begin{equation}
\partial^\n \big(W_{\kk-l,\ve}(\Pi_{x 0})
\Pi_{x \beta_1}\cdots\Pi_{x \beta_{l}}\big)(x) =0 \, .
\end{equation}
In particular, the Taylor polynomial in \eqref{hierarchy} can be dropped for such multiindices $\beta$.
\end{lemma}

\begin{proof}
By Leibniz rule we obtain linear combinations of the form 
\begin{equation}
\partial^{\n_0}W_{\kk-l,\ve}(\Pi_{x0})(x) \partial^{\n_1}\Pi_{x\beta_1}(x) \cdots \partial^{\n_l}\Pi_{x\beta_l}(x)
\end{equation}
where $\n_0+\cdots+\n_l=\n$ with $|\n|<|\beta|-2$, which implies $|\n|<\langle\beta\rangle-2$. 
This is only non-vanishing by continuity \eqref{est_Pi_taylorremainder} and the model estimate \eqref{est_Pi} 
if $|\n_i|\geq\langle\beta_i\rangle$ for all $i=1,\dots,l$. 
We obtain the contradiction 
$|\n|\geq\langle\beta_1\rangle+\dots+\langle\beta_l\rangle
=\langle\beta\rangle-2+(l-\kk)\alpha
\geq\langle\beta\rangle-2$.
\end{proof}

\section{Analytic ingredients}\label{sec_analytic}

In this section we collect a few results which are standard, but where we couldn't find a precise reference in the literature. For completeness we provide the proofs.

\begin{lemma}[Taylor remainder]\label{lem_interpolation}
Assume that for a family of random smooth functions $f_x$ there exist $\alpha\in\RR$ and $\eta\geq0$ such that for all $\n$ 
\begin{equation}
\EE^\frac1p|\partial^\n f_x(y) |^p 
\lesssim \ve^{\alpha-|\n|} (\ve+|x-y|)^\eta \, .
\end{equation}
Then for all $\n$ and $l\geq0$ 
\begin{equation}
\EE^\frac1p| (\mathrm{id}-T_z^{\leq l}) \partial^\n f_x (y) |^p
\lesssim \ve^{\alpha-|\n|-1-l} |y-z|^{1+l} (\ve+|x-y|+|x-z|)^\eta \, .
\end{equation}
\end{lemma}

\begin{proof}
We distinguish the two regimes $\ve\leq|y-z|$ and $\ve>|y-z|$. 
In the former case we break up the Taylor remainder and estimate by the triangle inequality 
\begin{equation}
\EE^\frac1p| (\mathrm{id}-T_z^{\leq l}) \partial^\n f_x (y) |^p
\lesssim \ve^{\alpha-|\n|}(\ve+|x-y|)^\eta
+ \sum_{|\m|\leq l} \ve^{\alpha-|\n|-|\m|} (\ve+|x-z|)^\eta |y-z|^{|\m|} \, ,
\end{equation}
and conclude by noting that in this regime $\ve^{1+l-|\m|}\leq|y-z|^{1+l-|\m|}$ for $|\m|\leq l$.
In the latter case we use the Taylor remainder in the form of 
\begin{align}
&(\mathrm{id}-T_z^{\leq l})\partial^\n f_x(y) \\
&= \partial^\n f_x(y) 
- \sum_{m_0+\dots+m_d\leq l} \tfrac{1}{\m!} \partial^{\m+\n} f_x(z) (y-z)^\m
+ \sum_{|\m|>l\colon m_0+\dots+m_d\leq l} \tfrac{1}{\m!} \partial^{\m+\n} f_x(z) (y-z)^\m \\
&= \sum_{m_0+\dots+m_d= l+1} \tfrac{l+1}{\m!} \int_0^1 ds \, (1-s)^l \partial^{\m+\n} f_x(sy+(1-s)z) (y-z)^\m \\
&\,+ \sum_{|\m|>l\colon m_0+\dots+m_d\leq l} \tfrac{1}{\m!} \partial^{\m+\n} f_x(z) (y-z)^\m \, .
\end{align}
By assumption we thus obtain 
\begin{align}
\EE^\frac1p| (\mathrm{id}-T_z^{\leq l}) \partial^\n f_x (y) |^p
&\lesssim \sum_{m_0+\dots+m_d=l+1} \ve^{\alpha-|\n|-|\m|}(\ve+|x-y|+|x-z|)^\eta |y-z|^{|\m|} \\
&\,+ \sum_{|\m|>l\colon m_0+\dots+m_d\leq l} \ve^{\alpha-|\n|-|\m|}(\ve+|x-z|)^\eta |y-z|^{|\m|} \, .
\end{align}
In both sums $|\m|\geq1+l$, so that we can bound $\ve^{1+l-|\m|}\leq |y-z|^{1+l-|\m|}$ and conclude. 
\end{proof}

\begin{lemma}[Sobolev]\label{lem:sobolev}
Let $u$ be a random smooth function, and let $1<q<2$, $\p>2$, $k>(1+d)/\p$. 
Then
\begin{equation}
\EE^\frac{1}{q}|u(y)|^q
\lesssim \sum_{n_0+\dots+n_d\leq k}
r^{|\n|-D/\p} \Big( \int_{B_r} dz\, \EE^\frac{\p}{q}\big|\partial^\n u(y+z)\big|^q\Big)^\frac{1}{\p} \, .
\end{equation}
\end{lemma}

The proof is a simple adaptation of the one given in \cite[Section~3.5]{BOT}, 
which we provide for completeness.

\begin{proof}
It suffices to prove
\begin{equation}
\EE^\frac{1}{q}|u(0)|^q
\lesssim \sum_{n_0+\dots+n_d\leq k}
\Big( \int_{B_1} dz\, \EE^\frac{\p}{q}\big|\partial^\n u(z)\big|^q\Big)^\frac{1}{\p} \, ,
\end{equation}
as replacing $u$ by $u(y+r\cdot)$ yields the desired estimate. 
For this consider a random variable $F$ such that $\EE^\frac{1}{q'}|F|^{q'}<\infty$ for $1/q'+1/q=1$, 
and define $\bar u(z)\coloneqq\EE[u(z)F]$.
By the standard Sobolev inequality (e.g.~\cite[Chapter~5.6.3, Theorem~6]{Evans}) we obtain
\begin{equation}
|\bar u(0)|
\lesssim \sum_{n_0+\dots+n_d\leq k}
\Big(\int_{B_1}dz\, \big|\partial^\n\bar u(z)\big|^{\p}\Big)^\frac{1}{\p} \, .
\end{equation}
By Hölder's inequality in probability we thus obtain
\begin{equation}
|\EE[u(0)F]|
\lesssim \EE^\frac{1}{q'}|F|^{q'}
\sum_{n_0+\dots+n_d\leq k}
\Big(\int_{B_1}dz\, \EE^\frac{\p}{q}\big|\partial^\n\bar u(z)\big|^{q}\Big)^\frac{1}{\p} \, .
\end{equation}
Since $F$ was arbitrary, we conclude by the duality 
$\EE^\frac{1}{q}|u(0)|^q = \sup_{F\neq0}|\EE[u(0)F]/\EE^\frac{1}{q'}|F|^{q'}$.
\end{proof}

\begin{lemma}[Fa\`a di Bruno]\label{lem:bruno}
Let $g\colon\RR^N\to\RR$ and $f\colon\RR\to\RR$ be sufficiently smooth functions. 
Then 
\begin{equation}
\tfrac{1}{\beta!}\partial^\beta f(g(x)) 
= \sum_{l\geq0} \tfrac{1}{l!} f^{(l)}(g(x))
\sum_{\substack{\beta_1+\dots+\beta_l = \beta \\ \beta_i\neq0\forall i}}
\tfrac{1}{\beta_1!}\partial^{\beta_1}g(x)\cdots\tfrac{1}{\beta_l!}\partial^{\beta_l}g(x) \, .
\end{equation}
Here $\beta$ is a multiindex on $\{1,\dots,N\}$, 
$\beta!\coloneqq\beta(1)!\cdots\beta(N)!$, and 
$\partial^\beta = \partial_{x_1}^{\beta(1)}\cdots\partial_{x_N}^{\beta(N)}$.
\end{lemma}

As an immediate consequence, we obtain for a formal power series $\Pi\in\RR [\![ z_{k}, z_{\n} ]\!]$ with coefficients $\Pi_\beta$, 
and a sequence of polynomials $W_k$ satisfying $W_k'=kW_{k-1}$, 
that $W_k(\Pi)$ is a formal power series with coefficients given by 
\begin{equation}\label{eq_faadibruno_W}
\big(W_k(\Pi)\big)_\beta
= \sum_{l\geq0} \tbinom{k}{l} W_{k-l}(\Pi_0)
\sum_{\substack{\beta_1+\dots+\beta_l = \beta \\ \beta_i\neq0\forall i}}
\Pi_{\beta_1}\cdots\Pi_{\beta_l} \, .
\end{equation}
Indeed, this follows from the previous lemma with the identification $\Pi_\beta = \frac{1}{\beta!} \partial^\beta \Pi|_{z=0}$.

\begin{proof}
We proceed by induction in the length of the multiindex $\beta$.
For $\beta=0$ the statement is clear.
In the induction step we assume the claim holds for a multiindex $\beta$, 
and prove it for the multiindex $\beta+\delta_k$.
Using the induction hypothesis, and the product and chain rule we obtain
\begin{align}
&\tfrac{1}{(\beta+\delta_k)!} \partial^{\beta+\delta_k} f(g(x)) \\
&= \tfrac{1}{\beta(k)+1} \partial_{x_k}
\tfrac{1}{\beta!} \partial^\beta f(g(x)) \\
&= \tfrac{1}{\beta(k)+1} \partial_{x_k}
\sum_{l\geq0} \tfrac{1}{l!} f^{(l)}(g(x))
\sum_{\substack{\beta_1+\dots+\beta_l = \beta \\ \beta_i\neq0\forall i}}
\tfrac{1}{\beta_1!}\partial^{\beta_1}g(x)\cdots\tfrac{1}{\beta_l!}\partial^{\beta_l}g(x) \\
&= \tfrac{1}{\beta(k)+1} \sum_{l\geq0} \tfrac{1}{l!} f^{(l+1)}(g(x)) \partial_{x_k}g(x)
\sum_{\substack{\beta_1+\dots+\beta_l = \beta \\ \beta_i\neq0\forall i}}
\tfrac{1}{\beta_1!}\partial^{\beta_1}g(x)\cdots\tfrac{1}{\beta_l!}\partial^{\beta_l}g(x) \label{tmp22} \\
&\,+ \tfrac{1}{\beta(k)+1} \sum_{l\geq0} \tfrac{1}{l!} f^{(l)}(g(x))
\sum_{j=1}^l
\sum_{\substack{\beta_1+\dots+\beta_l = \beta \\ \beta_i\neq0\forall i}}
\tfrac{1}{\beta_1!}\partial^{\beta_1}g(x)\cdots
\tfrac{1}{\beta_j!} \partial^{\beta_j+\delta_k}g(x)\cdots
\tfrac{1}{\beta_l!}\partial^{\beta_l}g(x) \, . \label{tmp23}
\end{align}
Using $1/l!=(l+1)/(l+1)!$, we can rewrite \eqref{tmp22} as 
\begin{equation}
\tfrac{1}{\beta(k)+1} \sum_{l\geq0} \tfrac{l+1}{(l+1)!} f^{(l+1)}(g(x)) 
\sum_{\substack{\beta_1+\dots+\beta_{l+1} = \beta+\delta_k \\ \beta_i\neq0\forall i,\,\beta_{l+1}=\delta_k}}
\tfrac{1}{\beta_1!}\partial^{\beta_1}g(x)\cdots\tfrac{1}{\beta_{l+1}!}\partial^{\beta_{l+1}}g(x) \, ,
\end{equation}
and by relabelling $l+1\mapsto l$ this equals
\begin{equation}
\tfrac{1}{\beta(k)+1} \sum_{l\geq1} \tfrac{l}{l!} f^{(l)}(g(x)) 
\sum_{\substack{\beta_1+\dots+\beta_{l} = \beta+\delta_k \\ \beta_i\neq0\forall i,\,\beta_{l}=\delta_k}}
\tfrac{1}{\beta_1!}\partial^{\beta_1}g(x)\cdots\tfrac{1}{\beta_{l}!}\partial^{\beta_{l}}g(x) \, .
\end{equation}
Note that for $l=0$ the condition $\beta_1+\dots+\beta_l=\beta+\delta_k$ can not be satisfied, so that we may start the first sum with $l=0$.
Furthermore, by symmetry this can be rewritten as
\begin{equation}
\tfrac{1}{\beta(k)+1} \sum_{l\geq0} \tfrac{1}{l!} f^{(l)}(g(x)) \sum_{j=1}^l
\sum_{\substack{\beta_1+\dots+\beta_{l} = \beta+\delta_k \\ \beta_i\neq0\forall i,\,\beta_{j}=\delta_k}}
\tfrac{1}{\beta_1!}\partial^{\beta_1}g(x)\cdots\tfrac{1}{\beta_{l}!}\partial^{\beta_{l}}g(x) \, .
\end{equation}
Using $1/\beta_j!=(\beta_j(k)+1)/(\beta_j+\delta_k)!$ and relabelling $\beta_j+\delta_k\mapsto\beta_j$, we can rewrite \eqref{tmp23} as
\begin{align}
\tfrac{1}{\beta(k)+1} \sum_{l\geq0} \tfrac{1}{l!} f^{(l)}(g(x))
\sum_{j=1}^l
\sum_{\substack{\beta_1+\dots+\beta_l = \beta+\delta_k \\ \beta_i\neq0\forall i,\,\beta_j\neq\delta_k}} \beta_j(k)
\tfrac{1}{\beta_1!}\partial^{\beta_1}g(x)\cdots
\tfrac{1}{\beta_l!}\partial^{\beta_l}g(x) \, .
\end{align}
Altogether, this yields
\begin{align}
\tfrac{1}{(\beta+\delta_k)!} \partial^{\beta+\delta_k} f(g(x)) 
= \tfrac{1}{\beta(k)+1} \sum_{l\geq0} \tfrac{1}{l!} f^{(l)}(g(x))
\sum_{j=1}^l
\sum_{\substack{\beta_1+\dots+\beta_l = \beta+\delta_k \\ \beta_i\neq0\forall i}} \beta_j(k)
\tfrac{1}{\beta_1!}\partial^{\beta_1}g(x)\cdots
\tfrac{1}{\beta_l!}\partial^{\beta_l}g(x) \, .
\end{align}
Swapping the last two sums and using that 
$\sum_{j=1}^l\beta_j(k) = (\beta+\delta_k)(k)=\beta(k)+1$, the claim follows.
\end{proof}

\section{Properties of the Appell polynomials}\label{sec_propAppell}

Recall from \eqref{eq:poly} that $ W_{k}$ is defined in terms of 
\begin{equation}\label{eq_supp1_W_def}
	W_{k}(\phi) \coloneqq \frac{\partial^{k}}{\partial \tau^{k}} \frac{e^{\tau \phi}}{\EE[
	e^{\tau Z}]} \bigg\vert_{\tau = 0}\,,
\end{equation}
where $ Z$ is the stationary solution of $ ( \partial_{t}- \Delta ) Z = \zeta$, hence, we
interpret \eqref{eq_supp1_W_def} evaluated at $ Z = Z (0)$.
Furthermore, we defined $ W_{k , \varepsilon}( \cdot ) \coloneqq \varepsilon^{\alpha k}
W_{k} ( \varepsilon^{- \alpha} \cdot ) $, which can be written alternatively in terms of 
\begin{equation}
	\begin{aligned}
		W_{k, \varepsilon}(\phi) \coloneqq \frac{\partial^{k}}{\partial \tau^{k}}
		\frac{e^{\tau \phi}}{\EE[
		e^{\tau Z_{\varepsilon}}]} \bigg\vert_{\tau = 0}\,,
	\end{aligned}
\end{equation}
where $ Z_{\varepsilon} (x) = \varepsilon^{ \alpha } Z ( x_{0}/ \varepsilon^{2},
x_{1, \ldots, d}/ \varepsilon) $ is the stationary solution of $ (
\partial_{t}- \Delta ) Z_{\varepsilon} = \xi_{\varepsilon}$.
Indeed, we have 
\begin{equation}
	\begin{aligned}
		W_{k} ( \varepsilon^{- \alpha} \phi ) 
		&= 
		\frac{\partial^{k}}{\partial \tau^{k}} \frac{e^{\tau
		\varepsilon^{- \alpha}\phi}}{\EE[
	e^{\tau Z}]} \bigg\vert_{\tau = 0}
	=
	\varepsilon^{- k  \alpha}
		\frac{\partial^{k}}{\partial ( \varepsilon^{- \alpha} \tau)^{k}} 
		\frac{e^{ \tau \varepsilon^{- \alpha} \phi }}{\EE [ e^{ \tau Z}]} \bigg\vert_{ \tau =
		0}\\
		& = 
		\varepsilon^{- k  \alpha}
		\frac{\partial^{k}}{\partial \overline{\tau}^{k}} 
		\frac{e^{ \overline{\tau}  \phi }}{\EE [ e^{ \overline{\tau}
		\varepsilon^{\alpha} Z}]} \bigg\vert_{
		\overline{\tau}=0}
		= 
		\varepsilon^{- k  \alpha}
		\frac{\partial^{k}}{\partial \overline{\tau}^{k}} 
		\frac{e^{ \overline{\tau}  \phi }}{\EE[ e^{ \overline{ \tau}
		Z_{\varepsilon}}]} \bigg\vert_{
		\overline{ \tau}=0}
		=\ve^{-k\alpha}W_{k,\ve}(\phi)\,,
	\end{aligned}
\end{equation}
where we performed the change of variable $ \overline{ \tau}= \varepsilon^{- \alpha}
\tau $. It is useful to keep in mind that $ W_{k, \varepsilon}$ is generated through the
formal power series 
\begin{equation}\label{eq_formalgeneratingseries}
	\begin{aligned}
		\frac{e^{ \tau \phi}}{ \EE [ e^{ \tau Z_{\varepsilon}}]} 
		= \sum_{k = 0}^{\infty} \frac{ \tau^{k}}{k!} W_{k, \varepsilon}( \phi)\,.
	\end{aligned}
\end{equation}

The following two lemmas are then an immediate consequence of the representation
\eqref{eq_formalgeneratingseries} by comparison of coefficients. 

\begin{lemma}\label{lem_appell}
	The sequence $( W_{k , \varepsilon})_{k \in \NN}$ forms an Appell sequence, i.e.~$W_{k , \varepsilon}' = k W_{k-1 , \varepsilon}$. 
\end{lemma}

\begin{lemma}\label{lem_zeromean}
	For every $k \in \NN \setminus \{0 \} $, we have  
	$ \EE [ W_{k , \varepsilon}( Z_{\varepsilon}
	) ]= 0 $.
\end{lemma}

\bibliography{cites}
\bibliographystyle{alpha}

\end{document}